\documentclass[11pt, reqno]{amsart}

\usepackage{amsmath}
\usepackage{amssymb}
\usepackage{bbm}
\usepackage{enumitem}
\usepackage{esint} 
\usepackage[colorlinks,citecolor=red,pagebackref,hypertexnames=false]{hyperref}
\usepackage{esint} 
\usepackage{enumitem} 
\usepackage{verbatim} 
\usepackage{dsfont}
\usepackage{titlesec}
\usepackage{titletoc}
\usepackage{etoolbox}
\usepackage{fmtcount}
\usepackage{tikz}
\usetikzlibrary{calc}
\usepackage{fp}
\usepackage{subcaption}

\usepackage{graphicx}
\usepackage{color}

\newcommand{\C}{{\mathbb C}}       
\newcommand{\M}{{\mathcal M}}       %
\newcommand{\R}{{\mathbb R}}       
\newcommand{\N}{{\mathbb N}}

\newcommand{\bS}{{\mathbb S}}
\newcommand{\DD}{{\mathcal D}}

\newcommand{\HH}{{\mathcal H}}

\newcommand{\WW}{{\mathcal W}}

\newcommand{\cB}{{\mathcal B}}
\newcommand{\cG}{{\mathcal G}}

\newcommand{\NN}{{\mathcal N}}

\newcommand{\EE}{{\mathcal E}}

\newcommand{\mfg}{{\mathfrak g}}

\newcommand{\mfc}{\mathfrak{c}}

\newcommand{\bM}{\mathbb{M}}

\newcommand{\diam}{{\rm diam}}
\newcommand{\dist}{{\rm dist}}
\newcommand{\ds}{\displaystyle }

\newcommand{\rf}[1]{{(\ref{#1})}}

\newcommand{\supp}{\operatorname{supp}}

\newcommand{\vphi}{{\varphi}}
\newcommand{\ve}{{\varepsilon}}
\newcommand{\vv}{{\vspace{2mm}}}
\newcommand{\vvv}{{\vspace{3mm}}}
\newcommand{\wt}[1]{{\widetilde{#1}}}
\newcommand{\wh}[1]{{\widehat{#1}}}

\newcommand{\noi}{\noindent}

\newcommand{\rad}{{\operatorname{rad}}}
\newcommand{\Apsi}{\mathcal{A}_{\psi}}
\newcommand{\apsi}{\mathfrak{a}_{\psi}^i}
\newcommand{\apsun}{\mathfrak{a}_{\psi}^+}
\newcommand{\apsdu}{\mathfrak{a}_{\psi}^-}
\newcommand{\aps}{\mathfrak{a}_\psi}
\newcommand{\lca}{\mathfrak{a}}
\newcommand{\A}{\mathcal{A}}
\newcommand{\chara}{\mathds{1}}
\newcommand{\dr}{\frac{dr}{r}}

\newcommand{\hi}{\mathcal{H}^n}
\newcommand{\pStop}{\mathcal{W}_{\Gamma}}
\newcommand{\ppStop}{\mathcal{W}_{\Gamma,0}}

\newcommand{\ttt}{{\rm Top}}

\newcommand{\G}{{\mathrm G}}
\newcommand{\VG}{{\mathrm {VG}}}

\newcommand{\LD}{{\mathrm{LD}}}

\newcommand{\BA}{{\mathrm{BA}}}
\newcommand{\Next}{{\mathrm{Next}}}

\newcommand{\capp}{\operatorname{Cap}}

\newcommand{\tcb}[1]{\textcolor{black}{#1}}
\newcommand{\redeps}{\kappa^2}
\newcommand{\redepsinv}{\kappa^{-2}}
\newcommand{\blueps}{\kappa^{-1}}

\def\XXint#1#2#3{{\setbox0=\hbox{$#1{#2#3}{\int}$ }
		\vcenter{\hbox{$#2#3$ }}\kern-.58\wd0}}

\def\one{\mathds{1}}

\newcommand{\cfive}{c_2}
\newcommand{\ctwo}{c_3}
\newcommand{\csix}{c_4}
\newcommand{\conedouble}{c_5}
\newcommand{\csixdouble}{c_6}
\newcommand{\cseven}{c_7}
\newcommand{\cfour}{c_8}

\newcommand{\Cten}{C_1}
\newcommand{\Celeven}{C_2}
\newcommand{\Cone}{C_3}
\newcommand{\Ctwo}{C_4}
\newcommand{\Cthree}{C_5}

\newcommand{\Whit}{\mathcal{W}}
\newcommand{\av}[1]{\left| #1 \right|}
\newcommand{\ps}[1]{\left( #1 \right)}

\newtheorem{theorem}{Theorem}[section]
\newtheorem{lemma}[theorem]{Lemma}

\newtheorem{coro}[theorem]{Corollary}
\newtheorem{propo}[theorem]{Proposition}
\newtheorem{claim}[theorem]{Claim}

\newtheorem{question}{Question}
\newtheorem*{theorem*}{Theorem}
\newtheorem*{claim*}{Claim}
\newtheorem{assum}[theorem]{Assumption}

\newtheorem{theorema}{Theorem}[section]

\newtheorem{mlemma}{Main Lemma}[section]

\theoremstyle{definition}
\newtheorem{definition}[theorem]{Definition}

\theoremstyle{remark}
\newtheorem{rem}[theorem]{\bf Remark}

\makeatletter
\newcommand{\settheoremtag}[1]{
	\let\oldthetheorem\thetheorem
	\renewcommand{\thetheorem}{#1}
	\g@addto@macro\endtheorem{
	\addtocounter{theorem}{-1}
	\global\let\thetheorem\oldthetheorem}
}
\makeatother

\numberwithin{equation}{section}

\renewcommand{\thepart}{Part \Roman{part}}

\titlecontents{section}
[1.5em] 
{}
{\contentslabel{1.5em}}
{\hspace*{-1.5em}}
{\titlerule*[0.5pc]{}\contentspage}

\titlecontents{subsection}
[3.5em] 
{}
{\contentslabel{2em}}
{\hspace*{-1.5em}}
{\titlerule*[0.5pc]{}\contentspage}

\titleformat{\part}
{\centering\Large\itshape}{ \thepart.}{.15cm}{}

\titleformat{\section}
{\centering\normalfont\scshape}{\thesection.}{.2cm}{}

\titleformat{\subsection}
{\centering \normalfont\itshape}{\thesubsection.}{0.1cm}{}

\titleformat{\subsubsection}{\normalfont\itshape}{\thesubsubsection.}{0.1cm}{}

\newcommand{\brem}{\begin{rem}}
	\newcommand{\erem}{\end{rem}}

\usepackage{xcolor}

\begin{document}
	
	\title[Carleson's $\ve^2$ conjecture]{Carleson's $\ve^2$ conjecture in higher dimensions}



	\author{Ian Fleschler}
\address{Department of Mathematics, Fine Hall, Princeton University, Washington Road, Princeton, NJ
08540, USA}
\email{imf@princeton.edu}

\author{Xavier Tolsa}

\address{ICREA (Barcelona, Catalonia)\\
 Universitat Aut\`onoma de Barcelona \\
and Centre de Recerca Matem\`atica (Barcelona, Catalonia).}
\email{xavier.tolsa@uab.cat}

\author{Michele Villa}

\address{University of Oulu (Oulu, Finland)
and Universitat Aut\`onoma de Barcelona (Barcelona, Catalonia).}

\email{michele.villa@oulu.fi}

\thanks{
I.F. acknowledges the support of the National Science Foundation through the grant FRG-1854147. 
X.T and M.V. are supported by the European Research Council (ERC) under the European Union's Horizon 2020 research and innovation programme (grant agreement 101018680). X.T. is also partially supported by MICINN (Spain) under the grant PID2020-114167GB-I00, the María de Maeztu Program for units of excellence (Spain) (CEX2020-001084-M), and 2021-SGR-00071 (Catalonia).  
M.V. is supported by the Academy of Finland via the project ``Higher dimensional Analyst's Traveling Salesman theorems and Dorronsoro estimates on non-smooth sets", grant No. 347828.
This work is based upon research funded by the  Deutsche Forschungsgemeinschaft (DFG, German Research Foundation) under Germany's Excellence Strategy – EXC-2047/1 – 390685813, while the authors were in residence at the Hausdorff Institute of Mathematics in Spring 2022 during the program ``Interactions between geometric measure theory, singular integrals, and PDEs''.
}

\subjclass[2020]{28A75, 28A78, 28A12}

\keywords{Tangent points, Carleson square function, rectifiability}
\date{\today}

	\begin{abstract}
In this paper we prove a higher dimensional analogue of Carleson's $\varepsilon^2$ conjecture. Given two arbitrary disjoint Borel sets $\Omega^+,\Omega^-\subset \mathbb{R}^{n+1}$, and $x\in\mathbb{R}^{n+1}$, $r>0$, we denote $$\varepsilon_n(x,r) := \frac{1}{r^n}\, \inf_{H^+} \mathcal{H}^n \left( ((\partial B(x,r)\cap H^+) \setminus \Omega^+) \cup ((\partial B(x,r)\cap H^-) \setminus \Omega^-)\right),$$ where the infimum is taken over all open affine half-spaces $H^+$ such that $x \in \partial H^+$ and we define $H^-= \mathbb{R}^{n+1} \setminus \overline {H^{+}}$. Our first main result asserts that any Borel subset of $$\left\{ x\in\mathbb{R}^{n+1} \, :\,  \int_0^1 \varepsilon_n(x,r)^2 \, \frac{dr}{r}<\infty \right\}$$ is $n$-rectifiable. For our second main result we assume that $\Omega^+$, $\Omega^-$ are open and that $\Omega^+\cup\Omega^-$ satisfies the capacity density condition. For each $x \in \partial \Omega^+ \cup \partial \Omega^-$ and $r>0$, we denote by $\alpha^\pm(x,r)$ the characteristic constant of the (spherical) open sets $\Omega^\pm \cap \partial B(x,r)$. We show that, up to a set of $\mathcal{H}^n$ measure zero, $x$ is a tangent point for both $\partial \Omega^+$ and $ \partial \Omega^-$ if and only if\begin{equation*}
  \int_0^{1} \min(1,\alpha^+(x,r) + \alpha^-(x,r) -2) \frac{dr}{r} < \infty.
  \end{equation*} The first result is new even in the plane and the second one improves and extends to higher dimensions the $\varepsilon^2$ conjecture of Carleson.	
	\end{abstract}
	\maketitle
	
	\newpage

 \tableofcontents


\section{Introduction}
In this paper we prove a higher dimensional analogue of Carleson's $\ve^2$ conjecture. Consider a Jordan domain $\Omega^+ \subset \R^2$ and set $\Omega^-= \R^2\setminus \overline{\Omega^+}$. For $x \in \R^2$ and $r>0$ denote by $I^+(x,r)$ and $I^-(x,r)$ the longest open arcs of the circumference  $S(x,r)\equiv\partial B(x,r)$ contained in $\Omega^+$ and $\Omega^-$ respectively (note that they might be empty). Then set
\begin{equation}\label{e:ve}
	\ve(x,r) = \frac{1}{r} \max \left( |\pi r - \HH^1(I^+(x,r))|, \, |\pi r - \HH^1(I^-(x,r))| \right),
\end{equation}
and define the Carleson's $\ve^2$ square function as 
\begin{equation}\label{e:intro-1}
	\mathcal{E}(x)^2 := \int_0^1 \ve(x,r)^2 \, \frac{dr}{r}.
\end{equation}
Carleson's $\ve^2$ conjecture, now a theorem, see \cite{JTV}, asserts the following. 
\begin{theorem}\label{jtv}
	Let $\Omega^+ \subset \R^2$ be a Jordan domain, $\Gamma= \partial \Omega^+$, and $\mathcal{E}$ the associated square function as defined in \eqref{e:intro-1}. Then the set of tangent points for $\Omega^+$ coincides with the subset of those points $x \in \Gamma$ such that $\mathcal{E}(x)<\infty$, up to a set of $\HH^1$-measure zero.
\end{theorem}
\noindent
Let us recall our notion of tangents.  For a point $x \in \R^{n+1}$, a unit vector $u$ and a parameter $a \in (0,1)$, we denote the two sided cone with axis in direction $u$ and aperture $a$ by
\begin{equation*}
	X_a(x,u) = \{ y \in \R^{n+1} \, |\, |(y-x) \cdot u | > a |y-x|\}.
\end{equation*}
Given two disjoint open subset $\Omega^+, \Omega^- \subset \R^{n+1}$,
 we say that $x \in \partial \Omega^+ \cap \partial \Omega^-$ is a tangent point for the pair $\Omega^+, \Omega^-$ if there exists a unit vector $u$ such that for all $a \in(0,1)$ there exists a radius $r>0$ so that
\begin{equation*}
	(\partial \Omega^+ \cup \partial \Omega^-) \cap X_a(x, u) \cap B(x,r) = \varnothing,
\end{equation*}
and moreover one component of $X_a(x,u) \cap B(x,r)$ is contained in $\Omega^+$ and the other in $\Omega^-$. The hyperplane through $x$ orthogonal to $u$ is the tangent to the pair $\Omega^+,\Omega^-$ at $x$. Sometimes, we will call it \textit{true} tangent, to distinguish it from the \textit{approximate} tangent\footnote{Recall that a set $E\subset\R^{n+1}$ has an approximate tangent at $x\in E$ if there exists a unit vector $u$ such that for all $a \in(0,1)$
$$\lim_{r\to0} \frac{\HH^n(E \cap X_a(x, u) \cap B(x,r))}{r^n} = 0.
$$
}.

Carleson's initial conjecture appeared naturally in a two-phase free boundary problem about harmonic measure; the necessary implication in Theorem \ref{jtv} was proven by Bishop \cite{bishop-thesis}, see also \cite{bcgj} and \cite[Theorem 6.3]{garnett-marshall}. The sufficient implication is more recent, see \cite{JTV}. We refer to the introduction of \cite{JTV} for a more thorough explanation about its connection with harmonic measure. 

\vspace{0.2cm}

It is natural to consider whether some analogue of Carleson's conjecture holds in higher dimension\footnote{Not only natural, but also useful: see the applications on free boundary problems and spectral inequalities in the companion paper \cite{FTV}, and below in this introduction.}. In other words, we would like to characterise tangent points of codimension-one `surfaces' in terms of a Carleson's $\ve$-like square function. In analogy with the case of the Analyst's Travelling Salesman Theorem  (ATST), even before being able to precisely state the problem we are confronted with two issues: 
\vspace{0.1cm}
\begin{itemize}
	\item[(A)] What is the right analogue of Carleson's $\ve$ square function?
	\vspace{0.1cm}
	\item[(B)] What sort of higher dimensional sets can we work with?
\end{itemize}
\vspace{0.1cm}
Issue (A) is rather obvious: $\ve$ as defined in \eqref{e:ve} won't possibly make sense in higher dimensions and  some  obvious generalisations of the coefficient $\ve$ may vanish on certain non-flat surfaces; 
Issue (B) might be a little more mysterious. After all, why not simply consider continuous images of $\R^n$ in $\R^{n+1}$? This is by now a well-known fact: connectedness in higher dimension doesn't say much at all about geometry. Indeed, note that Jordan curves possess measure theoretic qualities that are \textit{false} for higher dimensional topological manifolds; for example, curves are $1$-lower content regular: for any $x \in \Gamma$, where $\Gamma$ is a curve, and radius $0 < r < \diam(\Gamma)$, we have $\HH^1_\infty(\Gamma\cap (B(x,r)) \geq r$. This is definitely not true for a topological $2$-plane, for example. Some form of lower bound on the size of sets is usually needed when trying to deduce geometrical properties from the decay rate of integral quantities (see e.g. \cite{AS} in the context of the ATST). The obvious reason is that there might be very little to integrate upon, and thus finiteness of the square function says nothing about geometry. 

\subsection{A new coefficient and the rectifiability result}
To state our first main result, we need some more notation. For $x \in \R^{n+1}$, $r>0$ and an affine half-space $H$ such that $x \in \partial H$, denote by $S_{H}^+(x,r)=S(x,r) \cap H$ and $S_{H}^-(x,r):= S(x,r) \cap (\R^{n+1} \setminus \overline{H})$. Given two disjoint Borel sets $\Omega^+,\Omega^-\subset\R^{n+1}$, we put
\begin{equation}\label{ve_n-intro}
	\ve_n(x,r) := \frac{1}{r^n}\, \inf_H \HH^n \left( (S_H^+(x,r) \setminus \Omega^+) \cup (S_H^-(x,r) \setminus \Omega^-)\right).
\end{equation}
It is clear that if $\Omega^+$ and $\Omega^-$ are themselves complementary (open) half-spaces, then $\ve_n(x,r)=0$ for any  $x \in \partial \Omega^+ = \partial \Omega^-$ and $r>0$. Note that in the plane $\ve_1(x,r)  \lesssim \ve(x,r)$.  
The following result, our first main, shows that, indeed, this quantity does encode geometric information. Put 
\begin{equation}\label{calE_n}
	\mathcal{E}_n(x)^2 := \int_0^1 \ve_n(x,r)^2 \, \frac{dr}{r}.
\end{equation}
We will prove the following.

\begin{theorema}\label{teomain1}
	For $n\geq 1$ let $\Omega^+, \Omega^- \subset \R^{n+1}$ be two disjoint Borel subsets, and let $E\subset\{x \in \R^{n+1} \, |\, \mathcal{E}_n(x) < \infty\}$ be Borel\footnote{When $\Omega^+, \Omega^-$ are open, the set $\{x \in \R^{n+1} \, |\, \mathcal{E}_n(x) < \infty\}$ is Borel, and by the theorem it is $n$-rectifiable. So any subset is also $n$-rectifiable. See Remark \ref{rem-measurability}.}. Then $E$ is $n$-rectifiable. 
\end{theorema}

\noindent
Our proof is in fact quantitative and as such it yields more information than what Theorem \ref{teomain1} reveals. We make this explicit in 

\settheoremtag{A*}
\begin{theorem}\label{teomain1'}
	Let $n \geq 1$ and $\Omega^+, \Omega^- \subset \R^{n+1}$ be two disjoint Borel subsets. Fix $c_0 \in (0,1)$ and let $x_0 \in R^{n+1}$ and $r_0>0$. There exists a number $\delta>0$ such that if $\mu$ is a positive Radon measure such that $\mu(B(x,r))\leq r^n$ for all $x \in \R^{n+1}$ and $r>0$, $E \subset B(x_0,r_0)$ is Borel, and $\mu$ is supported on $E$, then the following is true. If $\mu(B(x_0, r_0))\geq c_0r_0^n$ and if
	\begin{equation*}
		\int_0^{20r_0} \ve_n(x,r)^2 \, \frac{dr}r < \delta \quad\mbox{ for all } x \in E,
	\end{equation*}
	then there exists a Lipschitz graph $\Gamma$ with slope $c_\Gamma\leq 1/10$ such that $\mu(\Gamma) \geq \tfrac{99}{100}\mu(B_0)$. Moreover, the slope $c_\Gamma$ can be taken to be arbitrarily small, assuming $\delta$ small enough.
\end{theorem}
\noindent

\noindent
Theorem \ref{teomain1} is both somewhat surprising and somewhat unsatisfactory. Surprising because, given what was said above, it nevertheless holds  
for arbitrary Borel subsets which do not need to be open and
\textit{without} size assumptions on\footnote{Since we do not assume that $\Omega^+, \Omega^-$ are complementary domains, it might very well happen that $\partial \Omega^+ \neq \partial \Omega^-$; hence the writing $\R^{n+1} \setminus (\Omega^+ \cup \Omega^-)$ to denote what, in the some obvious examples, would simply be the common boundary.} $\R^{n+1}\setminus (\Omega^+ \cup \Omega^-)$ (henceforth denoted by $F$). In fact, our proof yields some control on the behaviour of $\Omega^+$ and $\Omega^-$, albeit very weak:

\begin{coro}\label{coro-L1}
		Notation and assumptions as in Theorem \ref{teomain1}. $\HH^n$-almost every point $x \in E$ is an $L^1$ tangent point\footnote{See Definition \ref{def-L1tan} for the precise definition of $L^1$ tangent points.} for the pair of sets $\Omega^+$ and $\Omega^-$.
\end{coro}

\noindent
Still, Theorem \ref{teomain1} is somewhat unsatisfactory because it achieves less than desired: in Theorem \ref{jtv} the set where the square function is finite coincides with the set of \textit{true} tangents for $\Omega^\pm$. For $E\subset\{\mathcal{E}_n(x)<\infty\}$, its rectifiability implies existence of \textit{approximate} tangents \textit{of} $E$, but, aside from Corollary \ref{coro-L1}, we have no additional control on the behaviour of $F$ or on $\Omega^+$ and $\Omega^-$. 

\subsection{A family of new coefficients, slicing and \textup{(}true\textup{)} tangents}

The coefficient $\ve_n$ does not detect pieces of $F$ of zero $\HH^n$ measure: to prove existence of true tangents (assuming $\Omega^\pm$ to be open) we need a stronger quantity. To proceed further we need further notation. Given an open set $V$ in some manifold $\M$, its \textit{characteristic constant} $\alpha=\alpha_V$ is the number which satisfies $\lambda(V)=\alpha(n-1+\alpha)$, where $\lambda(V)$ is the first Dirichlet eigenvalue of $V$ (defined in terms of the Laplace-Beltrami operator $\Delta_\M$, the relevant case here being the spherical Laplacian). The Friedland-Hayman \cite{FH} inequality states that, given any two disjoint open subsets $V^+, V^-$ of $\mathbb{S}^n$,
\begin{equation}\label{e:FH}
	\alpha_{V^+} + \alpha_{V^-} - 2 \geq 0,
\end{equation}
and equality is attained if and only if $V^+$, $V^-$ are two complementary half-spheres. Note that all the above can be applied to the domains $V^\pm=S(x,r)\cap \Omega^\pm$, where $\Omega^+, \Omega^-$ are open and disjoint subsets of $\R^{n+1}$. In this case, we use the notation $\alpha^\pm(x,r)$ for the relative and rescaled characteristic constants.  More precisely, for $i=+,-$, we let $\alpha^i(x,r)= \alpha_{V^i}(x,r)$ where $V^i(x,r) = \{r^{-1}(x-y) \, : \, y \in S(x,r)\cap \Omega^i\}$. 
Now set 
\begin{equation*}
	\mathcal{A}(x)^2 := \int_0^1 \min(1, \,  \alpha^+(x,r)+\alpha^-(x,r) -2) \, \frac{dr}{r}.
\end{equation*}
By \eqref{e:FH}, this is a nonnegative quantity and vanishes when $\Omega^\pm$ are complementary half-spaces. What's intriguing, is that in \cite{AKN2} the authors observe that, in the planar case $n=1$,
\begin{equation}\label{eqakn**}
	 \ve(x,r)^2 \lesssim \min(1,\alpha_1(x,r) + \alpha_2(x,r) - 2), 
\end{equation}
where $\ve$ here is the classical Carleson's coefficient (see also \cite{Bishop-conjectures} for a related discussion). Then, if $\Omega^\pm$ are complementary Jordan domains in the plane, it follows from \cite{JTV} that $\mathcal{A}(x) < \infty$ implies that $x$ is a tangent point (up to a set of $\HH^1$ measure zero).
The inequality \eqref{eqakn**} is easy to check, see \cite{FTV}. 
It is rather harder to see that
\begin{equation}\label{ve_n<alphas}
	\ve_n(x,r)^2 \lesssim_n \min(1, \alpha^+(x,r)+ \alpha^-(x,r)-2),
\end{equation}
for $\Omega^\pm \subset \R^{n+1}$.
This is a special case of \cite[Theorem C]{FTV}, and the proof goes through a quantitative Faber-Krahn inequality for spherical domains. In \cite{FTV} moreover, it is shown that if $x$ is a tangent point of a pair of open sets $\Omega^\pm$ satisfying the so called capacity density condition (CDC), a rather weak size condition  (satisfied, in particular, when $\Omega^+$ is a Jordan domain in the plane and $\Omega^-= \R^2\setminus \overline{\Omega^+}$), then $\mathcal{A}(x)<\infty$, up to a set of zero $\HH^n$ measure. See Theorem D there. The proof uses the Alt-Caffarelli-Friedman monotonicity formula.

The previous discussion clearly suggests that the square function $\mathcal{A}$ is the good object for a putative Carleson's $\ve^2$ conjecture in higher dimension. Indeed, this is our next result.

\begin{theorema}\label{teomain2}
	For $n\geq 1$ let $\Omega^+, \Omega^- \subset \R^{n+1}$ be two open and disjoint subsets. Suppose that $\Omega^+ \cup \Omega^-$ satisfies the capacity density condition\footnote{See Definition \ref{cdc}.} \textup{(}CDC\textup{)}. Then, up to a set of zero $\HH^n$ measure,
	\begin{equation*}
		\mathcal{A}(x) < \infty \,\, \mbox{ if and only if $x$ is a tangent point of the pair} \,\, \Omega^+, \Omega^-.
	\end{equation*} 
\end{theorema}

As mentioned above, the ``if" part is shown in \cite[Theorem D]{FTV}. It is in this implication that the minimum in the integrand of $\mathcal{A}$ appears naturally, for even at a tangent point $x$ it might happen that at large scales $r$, $\Omega^i \cap S(x,r) = \varnothing$ for either $i=+$ or $i=-$, in which case the characteristic constant $\alpha^+(x,r)$ degenerates to infinity. 

We also remark that $\mathcal{A}$ is in fact a much stronger quantity than what \eqref{ve_n<alphas} suggests. Set 
\begin{equation}\label{ve_s}
	\ve_s(x,r) = \inf_{H} \frac{1}{r^s} \int_{V_{\mathfrak{c}_0}^a(x,r;H)}\left( \frac{\delta_{L_H}(y)}{r}\right)^{n-s} \, d  \HH_\infty^{s}(y) \,\, \mbox{ when } \,\, 0<s<n-1,
\end{equation}
 where the infimum is over all affine half-spaces $H\subset \R^{n+1}$ such that $x\in\partial H$, $\delta_{A}(y):= \dist(y, A)$, and integration with respect to Hausdorff content is defined in \eqref{e:choquet}.
Also $V_{\mathfrak{c}_0}^a(x,r;H)$ is a set of `thick points' in $S(x,r)$ with respect to $F=\R^{n+1}\setminus (\Omega^+\cup\Omega^-)$, which will be defined accurately below, see Definition \ref{def-thickpoints} (for the moment, the reader might simply read $S(x,r) \cap F$ instead). Then the full Theorem C from \cite{FTV} reads as follows.

\begin{theorem}\label{teo-ve<alpha}
	Let $n\geq 1$. Suppose that $0<s\leq n$. Let $V^+, V^- \subset S(x,r)\subset \R^{n+1}$ be relatively open in $S(x,r)$ and disjoint. If $\alpha^\pm(x,r)$ denote the characteristic constants of $V^\pm$,  then for any $\mathfrak{c}_0>0$ and $a \in (0,1)$, we have
	\begin{equation}
		\ve_s (x,r)^2 \lesssim_{s, \mathfrak{c}_0, a} \min(1, \alpha^+(x,r)+\alpha^-(x,r) -2),
	\end{equation} 
	where $\ve_s(x,r)$ is as in \eqref{ve_s} if $s<n$ and as in \eqref{ve_n-intro} when $s=n$.
\end{theorem}
\noindent

As we will see, in order to prove the existence of true tangents, we need the full strength of $\mathcal{A}(x)$, and in particular the fact that $\mathcal{E}_s(x)$, defined as in \eqref{calE_n} but now integrating $\ve_s(x,r)$, is finite for some $s<n$. For $\ve_s$ to access geometric information on $F$, $F\cap S(x,r)$ needs to be somewhat ``thick'', or, phrased differently, $V_{\mathfrak{c}_0}^a(x,r;H) \cap F$ should have sizeable measure. This motivates our next result, which allows for the ``thickness'' of $F$, guaranteed by the CDC, to be transferred onto the slices $S(x,r)\cap F$.

\begin{propo}\label{propo-slicing}
		Let $B(0,r_0)\subset \R^{n+1}$ and let $\Gamma$ be a Lipschitz graph through the origin with slope at most $\tau$ \textup{(}with respect to $\R^n\times \{0\}$\textup{)}. Let $B\subset B(0,r_0)$ be a ball with $\rad(B)\leq \frac1{10}\,r_0$ such that 
		$\dist(B,\Gamma)\geq 100\tau r_0$.
		Let $K\subset B$  and  $G\subset \Gamma$ both be  compact sets. 
		Then, for any $s>1$,
		$$\capp_{s}(K)\,\frac{\HH^n(G)^2}{r_0^n} \leq C(\tau) \int_{G}\int_0^\infty \capp_{s-1}(K\cap S(z,r))\,dr\,d\HH^n(z).$$
	\end{propo}

How to use this? We know from Theorem \ref{teomain1} that $E=\{ \mathcal{A}(x)<\infty\}$ is $n$-rectifiable. In particular, it can be split into parts $E_i$ which lie in Lipschitz graph with small slope. Hence we can locally apply Proposition \ref{propo-slicing} with $G=E_i$ and $K = \overline{B} \cap F = \overline B \setminus (\Omega^+\cup\Omega^-)$ to translate the size condition CDC onto slices and hence access the geometric strength of $\ve_s$, for some appropriate $s<n$.


\subsection{A brief panorama on related literature}
Aside from Theorem \ref{jtv} (over which Theorem \ref{teomain2} is an improvement), there are several related results which characterise tangent points of different types of sets. To the authors knowledge, one of the first and most influential result of this type is that of Bishop and Jones, \cite[Theorem 2]{bishop-jones}, where they prove the following: let $E \subset \R^2$ be compact. Then, up to a set of $\HH^1$-measure zero, $x \in E$ is a tangent point of $E$ if and only if 
\begin{equation*}
	\int_0^1 \beta_{\infty, E}(x,r)^2 \, \frac{dr}{r} <\infty. 
\end{equation*} 
Here $\beta_{\infty, E}(x,r)$ is the well-known Jones' (see \cite{jones}) $\beta$-number $
	\beta_{\infty, E}(x,r) := \inf_{L} \sup_{y \in E \cap B(x,r)} \frac{\delta_{L}(y)}{r},
$
where the infimum is taken over all affine lines in $\R^2$. In fact, the result of Bishop and Jones was partly motivated by Carleson's $\ve^2$ conjecture. In higher dimensions, Theorem \ref{teomain2} should be compared to \cite[Theorem 1.4]{villa-tan}, where tangent points of lower content regular sets are characterised in terms of integral $\beta$ coefficients, and integration is with respect to Hausdorff content. In general, it is more difficult to work with `spherical' square functions, such as $\ve, \ve_n, \ve_s$, as they are unstable with respect to the radii. Indeed, the current proof is twice as long and rather more subtle than that of \cite{villa-tan}. We mention another related result\footnote{There and in other instances the term `cone point' is used instead of `tangent point'.}, \cite{hyde-villa}, where the authors characterise tangent points of general compact sets (no lower size conditions assumed). This might come as a surprise. The point in that work is to use a modified Hausdorff content, which essentially sees any set as `lower regular'.
More akin to Theorem \ref{teomain1} is the characterization of rectifiability \cite{tolsa-beta}, \cite{azzam-tolsa}, later extended in \cite{env}; these works show that a measure with positive upper density and finite lower density is rectifiable if and only if a Dini integral as in \eqref{calE_n} integrating $\beta_{2,\mu}(x,r)$ is finite $\mu$-a.e. We refer the reader to those works for a precise definition of these coefficients. 

 Tangent (or cone) points are of course basic objects in rectifiability, hence there is a vast literature characterizing them in various manners (remark that characterisations are often in terms of approximate tangents, as in Theorem \ref{teomain1}). See for example the works \cite{dabrowski, badger-naples, chang-tolsa, del-nin, orponen-martikainen} and references therein.

\subsection{Some open questions}
We believe that Proposition \ref{propo-slicing} has its standalone interest. Unfortunately, the assumptions we make are rather strong. However, a theorem by Mattila, see  \cite{mattila-acta}, says that slicing by planes holds for general sets. See also \cite[Chapter 10]{Mattila-gmt}. A natural question is then whether Proposition \ref{propo-slicing} holds in this level of generality.
More interesting from the rectifiability point of view, is slicing \textit{with respect to Hausdorff measures/contents}. The question is the following:

\begin{question}
	Is it true that \textup{(}some form of\textup{)} slicing with respect to Hausdorff measure/contents, and either by planes or sphere, implies rectifiability?
\end{question}
\noindent
Of course, the question could be formulated rather more precisely, but we leave it vague deliberately.

\subsection{Structure of the article}
Very broadly, the proof splits into three steps, corresponding to the three parts of the paper. The first and second steps prove Theorem \ref{teomain1}, the third improves on this to prove Theorem \ref{teomain2}.
\begin{itemize}
	\item[\emph{Step 1}.] We work assuming that $\mathcal{E}_n$ is small; we have no assumptions of $\partial \Omega^+ \cup \partial \Omega^-$. Denote by $E$ the set where $\mathcal{E}_n$ is finite. There are three main preliminary results that we achieve in this first step:
	\begin{itemize}
		\item[Step 1.1] Here we prove a key preliminary fact (see Lemma \ref{lemcork1}): the smallness of  $\EE_n$ implies the \textit{existence of quasi-corkscrew balls} (see Definition \ref{def-quasi-corkscrew}).
		\item[Step 1.2] This step, one of the most involved of the paper and carried through in Sections \ref{sec-flat2}, \ref{sec-flat3}, \ref{sec-flat4}, shows that, whenever we have a ball with substantial measure of $E$, and $\ve_n(x,r)$ is small, then $E$ is flat, that is, $\beta_{\infty, E}(B)<\epsilon$.
		\item[Step 1.3] The last step to prove the Main Theorem \ref{teomain1} is to show that wherever the square function is small (and we have substantial measure) then not only $E$ is very flat, but moreover while one side of the minimising $n$-plane is 99\% contained in $\Omega^+$  in terms of Lebesgue measure, the other side is 99\% contained in $\Omega^-$. See Section \ref{sec-splitting}.
	\end{itemize}
In Section \ref{sec-proofs-teo12} we put together these three steps (summarised in Main Lemma \ref{mainlemma1}) to prove Theorem \ref{teomain1}.

\item[Step 2.] This step is summarised in Main Lemma \ref{mainlemma2}. 	We prove Theorem \ref{teomain1}  assuming these two steps in Section \ref{sec-proofs-teo12}.
\begin{itemize}
	\item[Step 2.1.] A preliminary result (Section \ref{sec-fourier}) gives us some precise Fourier analytic estimates on certain auxiliary square functions on Lipschitz graphs. 
	\item[Step 2.2.] Next, via a Corona-type construction, we approximate $E$ by a Lipschitz graph $\Gamma$ to then transfer the Fourier analytic estimates onto $E$. This will allow us to say that, indeed, the approximation is so good that it in fact intersects $E$ on a sizeable portion. See Sections \ref{sec-corona}, \ref{sec-coronaest}.
\end{itemize}

\item[Step 3.] The final step is concerned with improving Theorem \ref{teomain1} to obtain true tangents via slicing and assuming CDC. See sections \ref{sec-slicing}, \ref{sec-mainlemma3}. The proof of Theorem \ref{teomain2} can be found in Section~\ref{sec-proofs-teo12}. 
\end{itemize}

\vv
\section{Preliminaries}

\subsection{Notation}\label{secnot}
 In the paper, constants denoted by $C$ or $c$ depend just on the dimension unless otherwise stated. We will write $a\lesssim b$ if there is $C>0$ such that $a\leq Cb$. We write $a\approx b$ if $a\lesssim b\lesssim a$.
 
 Open balls in $\R^{n+1}$ centered in $x$ with radius $r>0$ are denoted by $B(x,r)$, and closed balls by 
$\bar B(x,r)$. For an open or closed ball $B\subset\R^{n+1}$ with radius $r$, we write $\rad(B)=r$.
An open annulus in $\R^{n+1}$ centered in $x$ with inner radius $r_1$ and outer radius $r_2$ is denoted by 
$A(x,r_1,r_2)$, and the corresponding closed annulus by $\bar A(x,r_1,r_2)$. 
We use the two notations $S(x,r)\equiv \partial B(x,r)$ for spheres
in $\R^{n+1}$ centered in $x$ with radius $r$, so that $\bS^n=S(0,1)$.

The Hausdorff distance between two compact sets $A,B\subset\R^{n+1}$ is defined by 
 \begin{equation*}
 	d_H(A, B) = \max \left\{ \sup_{x \in A } \dist(x, B), \, \sup_{x \in B} \dist(x, A)\right\}.
 \end{equation*}

We say a that a measure $\mu$ in $\R^{n+1}$ has $n$-polynomial growth if there exists some constant $C_0$ such that
$$\mu(B(x,r))\leq C_0\,r^n\quad\text{ for all $x\in\R^{n+1}$ and $r>0$.}$$
If we want to be precise about the constant $C_0$, we will say that $\mu$  $(n,C_0)$-polynomial growth or that it has 
$n$-polynomial growth with constant $C_0$.

Given a set $F\subset\R^{n+1}$, the notation $M_+(F)$ stands for the
set of (positive) Radon measures supported on $F$.
For $s>0$, the $s$-dimensional Hausdorff measure is denoted by $\HH^s$, and the $s$-dimensional Hausdorff
content by $\HH^s_\infty$. We also define integration with respect to  $\HH^s_\infty$ as the Choquet integral
\begin{equation}\label{e:choquet}
	\int_A f(x)^p \, d \HH^s_\infty(s) := \int_0^\infty \HH_\infty^s\left(\{x \in A \, |\, f(x) > t \}\right) \, t^{p-1} \, dt.
\end{equation}
Recall that Hausdorff content is \textit{not a measure}. For more on Hausdorff contents and Choquet integration, see \cite[Preliminaries]{AS}.


\vv

\subsection{Capacities and the CDC}\label{seccap}

 Given a Radon mesaure $\mu$ in $\R^{n+1}$ and $s>0$, we consider the potential
$U_s\mu(x) = \int \frac1{|x-y|^s}\,d\mu(y)$. In the case $s=0$, we set
$U_0\mu(x) = \int
\frac1{2\pi}\,\log \frac1{|x-y|} \,d\mu(y)$. Then, for $s\geq0$, we consider the energy
 $I_s(\mu)$ defined by
$$
I_s(\mu) =  \int U_s\mu(x)\,d\mu(x).
$$
The $s$-capacity of a compact set $F\subset \R^{n+1}$ is equal to
$$
\capp_s(F) = \frac 1{\inf_{\mu\in M_1(F)} I_s(\mu)},
$$
where the infimum is taken over all {\em probability} measures $\mu$ supported on $F$. For subsets $F\subset\bS^n$, we define $\capp_s(F)$ in the same way, understanding that $\bS^n\subset\R^{n+1}$.

For $s>0$, equivalently, we can write
\begin{equation}\label{eqcaps}
\capp_s(F)=\sup\big\{\mu(F):\mu\in M_+(F),\,\|U_s\mu\|_{\infty,F}\leq1\big\},
\end{equation}
where $M_+(F)$ is the family of all Radon measures supported in $F$.
It is easy to check that the capacity $\capp_s$ is
homogeneous of degree $s$. In the case when $s=n-1$ with $n\geq2$, $\capp_{n-1}$ coincides the Newtonian capacity, modulo a constant factor.

In the case $s=0$, $n=1$, $\capp_0(F)$ is the Wiener capacity of $F$. The Wiener capacity is not homogeneous. However, the logarithmic capacity $\capp_L$, defined by
$$\capp_L(F) = e^{-\frac{2\pi}{\capp_0(F)}},$$
is $1$-homogeneous.

We have the following well-known relationship between $s$-capacities and Hausdorff contents.

\begin{lemma}\label{lemcap-contingut}
Let $F\subset\R^{n+1}$ be compact and $0<s<t\leq n+1$. We have
\begin{equation}\label{eq:capHs}
 \HH_\infty^t(F)^{\frac{s}t}\lesssim_{s,t} \capp_{s}(F)\lesssim\HH_\infty^s(F).
 \end{equation}
On the other hand, for $s>0$, we have
\begin{equation}\label{eq:capH0}
\capp_L(E) \gtrsim_s \HH_\infty^s(E)^{\frac1s}.
 \end{equation}
\end{lemma}

The proof of the second inequality in \rf{eq:capHs} follows easily from the identity \rf{eqcaps}, while the first one is a direct consequence of  Frostman's Lemma. See \cite[Chapter 8]{Mattila-gmt}, for example. The estimate \rf{eq:capH0}
follows also from  Frostman's Lemma (see \cite[Lemma 4]{Cufi-Tolsa-Verdera}).

\begin{definition}\label{cdc}
For $n\ge1$, let $\Omega\subset\R^{n+1}$ be an open set. For $s>0$, we say that $\Omega$ satisfies the $s$-capacity density condition ($s$-CDC)
if there exists some constant $c>0$ such that, for every $x\in\partial\Omega$ and $r\in(0,\diam(\Omega))$, 
$$\capp_{s}(B(x,r)\setminus\Omega)\geq c\,r^{s}.$$
In the case $s=0$, we say that $\Omega$ satisfies the $0$-capacity density condition ($0$-CDC) if, for every $x\in\partial\Omega$ and $r\in(0,\diam(\Omega))$, 
$$\capp_L(B(x,r)\setminus\Omega)\geq c\,r.$$
In the case $s=n-1$ (for $\Omega\subset \R^{n+1}$), we just say that $\Omega$ satisfies the CDC.
\end{definition}

The following result is a direct consequence of a more general result due to Lewis \cite{Lewis}.

\begin{lemma}\label{lem-fat}
For $n\ge1$, let $\Omega\subset\R^{n+1}$ be an open set satisfying the CDC. Then there exists some $s>n-1$ such that
$\Omega$ also satisfies the $s$-CDC.
The parameter $s$ depends only on $n$ and the constant $c$ involved in the definition of CDC for $\Omega$.
\end{lemma}

\vv
\subsection{The $\ve_s$ coefficients}\label{sec-main-functions}

In this section, we make precise the definition of $\ve_s$ (recall that we left suspended the definition of the `thick points' set $V_{\mathfrak{c}_0}$) and show the measurability of the various square functions. 
First, however, we rewrite the definition of $\ve_n$, for future reference.

\begin{definition}\label{def-ve_n} Let $\Omega^+, \Omega^- \subset \R^{n+1}$ be disjoint Borel sets. For $x \in \R^{n+1}$, $r>0$ and an affine half-space $H$ with $x \in \partial H$, denote by $S_{H}^+(x,r)=S(x,r) \cap H$ and $S_{H}^-(x,r):= S(x,r) \cap (\R^{n+1} \setminus \overline{H})$. We put
\begin{equation*}
	\ve_n(x,r; H):= \HH^n \left( (S_H^+(x,r) \setminus \Omega^+) \cup (S_H^-(x,r) \setminus \Omega^-)\right)
\end{equation*}
and
\begin{equation}\label{ve_n} 
	\ve_n(x,r) := \inf_H \ve_n(x,r; H),
\end{equation}
where the infimum is taken over all affine half-spaces such that $x \in \partial H$.
\end{definition}

\vvv

\begin{definition}[Thick points]\label{def-thickpoints}
	Let $\mfc_0>0$ and $a \in(0,1)$. Let $\Omega^+, \Omega^-$ be two open disjoint subsets in $\R^{n+1}$ and put $F:= \R^{n+1}\setminus (\Omega^+ \cup \Omega^-)$. Let $x \in F$ and $r>0$ and $L$ be an $n$-plane trough $x$. Then, if $n\geq 3$ we set, for $i=+,-$,
	\begin{equation*}
		T_{\mfc_0,a}^i(x,r, H):= \{y \in S_H^i(x,r) \setminus \Omega^i \, |\, \capp_{n-2}(\overline{B}(y, a \delta_{L_H}(y))\cap S(x,r)\setminus \Omega^i)\geq \mfc_0 \delta_{L_H}^{n-2}(y)\}.
	\end{equation*}		
In the case $n=2$, we set
$$	T_{\mfc_0,a}^i(x,r, H):= \{y \in S_H^i(x,r) \setminus \Omega^i \, |\, \capp_L(\overline{B}(y, a \delta_{L_H}(y))\cap S(x,r)\setminus \Omega^i)\geq \mfc_0 \delta_{L_H}(y)\}.$$
	Then we put
	\begin{equation*}
		V_{\mfc_0}^a(x,r; H) := T_{\mfc_0, a}^+ (x,r; H) \cup T_{\mfc_0, a}^- (x,r; H).
	\end{equation*}
With this, we set
\begin{equation*}
	\ve_s(x,r):= \inf_{H} \frac{1}{r^s} \int_{V_{\mfc_0}^a(x,r; H)} \, \left(\frac{\delta_{L_H}(y)}{r}\right)^{n-s} \, d \HH^{s}_\infty(y).
\end{equation*}
\end{definition}

\vv

\begin{rem}[About measurability]\label{rem-measurability}

\begin{itemize}
	\item[(i)]
First some remarks about the coefficients $\ve_n(x,r)$. 
The fact that $\Omega^\pm$ are Borel ensures that for each fixed x and r, the sets $S(x,r) \cap \Omega^\pm$ are Borel, so that we can compute the coefficient $\ve_n(x,r)$ for all $x, r$. Moreover, for each fixed $x$, the function $r \mapsto \ve_n(x,r)$ is Lebesgue measurable in $\R$, by using Fubini in polar coordinates. We omit the details.
Let us also remark that the function
$
(a,b) \mapsto  \int_a^b \ve_n(x,r)^2\,dr/r$
is continuous  with respect to the parameters $a,b$.

\item[(ii)] As far as $\HH^n$-measurability of $x \mapsto \int_0^1 \ve_n(x,r)^2 \, dr/r$ is concerned, note first that it suffices to prove $\HH^n$-measurability of $x \mapsto \ve_n(x,r; H)$ for any affine half-space $H$. If $\Omega^+$ and $\Omega^-$ are open subsets this can be done directly since, given $H$ an affine half-space, $\ve_n(x,r;H)$ is an upper semi-continuous function of $(x,r)$. In the more general case of $\Omega^+$ and $\Omega^-$ being Borel subsets, we appeal to a coarea inequality of Esmayli and Haj\l{}asz \cite[Theorem 1.1]{EH}, which in particular says that if $X, Y$ are metric spaces, $X$ is boundedly compact, $0 \leq t \leq s < \infty$, $\wt E \subset X$ an $\HH^s$-measurable subset with $\HH^s(\wt E) <\infty$, and $f:X \to Y$ a Lipschitz function, then 
\begin{equation}\label{eh-coarea}
	\int_{Y} \HH^{s-t} (\wt E \cap f^{-1}(y)) \, d \HH^t(y) \lesssim (\mathrm{Lip}(f))^t \HH^s(\wt E),
\end{equation}
and the function $y \mapsto \HH^{s-t}(f^{-1}(y)\cap \wt E)$ is $\HH^t$-measurable.

To use \eqref{eh-coarea}, we put $X=\R^{n+1} \times \R^{n+1} = \R^{2n+2}$ and $Y= \R^{n+1}\times \R \times \R$; for $R \ll \diam(E) +1$ (which for our purposes we can assume to be finite), set $\wt E= E \times (B(0, R)\setminus \Omega^+)$ - note that $\HH^{2n+1}(\wt E)<\infty$. Finally, let $f(x,y)= (x, |x-y|, y_{n+1}-x_{n+1})$. Note that $f(x,y) = (x, r, h)$ if and only if $|x-y|=r$ and $h=y_{n+1} - x_{n+1}$. As mentioned above, it suffices to show measurability of $x \mapsto \ve_n(x,r;H)$. For simplicity we take $H= \R^{n+1} \cap \{x_{n+1} >0\}$. We apply the Esmayli-Haj\l{}asz result with $s=2n+1$ and $t=n+2$, to obtain that 
\begin{equation*}
	(z, r, h) \mapsto \HH^{n-1} (\wt E \cap f^{-1}(z,r,h)) \mbox{ is } \HH^{n+2}\mbox{-measurable}.
\end{equation*}
It follows by Fubini-Tonelli that, for fixed $r$,
\begin{equation*}
	z \mapsto \int_0^r \HH^{n-1} (\wt E \cap f^{-1}(z,r,h)) \, dh \,\, \mbox{is $\HH^n$-measurable.}
\end{equation*}
Now it is enough to note that the integral in the last display equals to $\HH^n(S_H^+(x,r) \setminus \Omega^+)$. Repeating the same argument with $\HH^n(S_H^-(x,r)\setminus \Omega^-)$ proves that $x \mapsto \ve_n(x,r;H)$ is $\HH^n$-measurable.

\item[(iii)]

In the case in which $\Omega^+$ and $\Omega^-$ are two disjoint open sets. The coefficients $\varepsilon_n(x,r)$ are upper semicontinous in $(x,r)$.
Note first that, given $H$ an affine half space, $\ve_n(x,r;H)$ is an upper semi-continuous function of $(x,r)$. Then, since the infimum of upper semicontinuous functions is upper semicontinuous we get that $\ve_n(x,r)$ is upper semicontinuous and thus the map $x \mapsto \int_{s}^1 \ve_n(x,r)^2 \frac{dr}{r}$ is upper semicontinuous in $x$ for each $s$. The set where the square function is finite is the countable union over $s=1/m$ of the intersection of the sets where $\int_{s}^1 \ve_n(x,r)^2 \frac{dr}{r} \leq N$ and thus in the open case it $E$ is borel.

\item[(iv)]
The coefficients \begin{equation*}\widetilde{\ve}_n(x,r):= \lim_{t \rightarrow 0}\,\sup_{
|(x-y,r-s)|<t}\varepsilon_n(y,s),
\end{equation*} which are defined as the upper semicontinuous envelope of $\ve(x,r)$ in $\R^n\times\R$, enjoy better measurability properties than $\varepsilon_n(x,r)$. In fact, for arbitrary Borel sets $\Omega^+$, $\Omega^-$, they are upper semicontinuous in $(x,r)$ and thus
\begin{equation*}
\Big\{x: \int_0^1 \widetilde{\ve_n}(x,r)^2\frac{dr}{r}<\infty\Big\}
\end{equation*}
is a Borel subset of $E$ for which our theorem applies.

\item[(v)]  As far as the quantity $\alpha(x,r)$ is concerned, it is upper semicontinuous in $(x,r)$ when $\Omega^\pm$ are open, as a consequence of the upper semicontinuity of $\lambda(x,r)$, i.e. the first eigenvalue of the spherical laplacian of the rescaled open set in the sphere.  Indeed, for any function $u \in  C^\infty$
in $\Omega$, let
\begin{equation*}
	F_u(x,r) = \frac{\|\nabla u\|_{L^2(\Omega \cap S(x,r))}}{\|u\|_{L^2(\Omega
			\cap S(x,r))}}.
\end{equation*}
For a smooth function $F_{u}(x,r)$ is continuous on $(x,r)$ and thus when taking the infimum over $u$ we get that $\lambda(x,r) = \inf_u F_u(x,r)$ is an upper semicontinuous function of $(x,r)$.

\item[(vi)] Finally,  the square function $\mathcal{E}_s$ is also upper semicontinuous when $\Omega^\pm$ are open. Again it is enough to show that for a fixed plane $H$
\begin{equation*}
(x,r) \mapsto	\frac{1}{r^s} \int_{V_{\mfc_0}^a(x,r; H)} \, \left(\frac{\delta_{L_H}(y)}{r}\right)^{n-s} \, d \HH^{s}_\infty(y)
\end{equation*}
is upper semicontinuous on $\R^{n+1}  \times (0, \infty)$.
It is easy to notice that if $x_i \rightarrow x$ and $r_i \rightarrow r$ then, since capacity over compact sets is upper semicontinuous with respect to Hausdorff distance, 
\begin{equation*}
\limsup_{i \to \infty} V_{\mfc_0}^a(x_i,r_i; H)  \subset V_{\mfc_0}^a(x,r; H).
\end{equation*}
Since upper semicontinuity with respect to Hausdorff distance also holds for the Hausdorff content, and recalling that $\ve_s(x,r)$ is defined as a Choquet integral, upper semicontinuity follows for $\mathcal{E}_s$ as well.
\end{itemize}
\end{rem}
\vv




\section{Auxiliary coefficients}

In this section, we introduce a few coefficients that will help us at the various stages of the proof. 
Throughout this section, let $\ve(x,r)= \ve_n(x,r)$.

\subsection{The symmetry coefficients}
Let  $\Omega^+,\Omega^-\subset \R^{n+1}$ be disjoint Borel sets and let 
$B(x,r)\subset \R^{n+1}$ be a ball.
For a given $x\in\R^{n+1}$, we consider the symmetry map with respect to $x$ given by $T_x(y) = 2x-y$.
For $U \subset \R^{n+1}$, we write $U^c = \R^{n+1} \setminus U$.
For  $x\in\R^{n+1}$, $r>0$, we consider the coefficients
\begin{equation}\label{e:def-gamma}
\gamma(x,r)  = 
\frac1{r^{n}}\,\max_{i=+,-}\HH^n\big((\Omega^i\triangle\, T_x(\Omega^{i,c}))\cap S(x,r)\big),
\end{equation}
(where $\Omega^{i,c}$ stands for the complement of $\Omega^i$)
and
\begin{equation}\label{e:def-a}
a(x,r)  = 
\frac1{r^{n}}\,\max_{i=+,-}\Big|\HH^n(\Omega^i\cap S(x,r)) - \frac12\HH^n(S(x,r))
\Big|.
\end{equation}
\vv

\begin{lemma}\label{lemgamma}
	Let  $\Omega^+,\Omega^- \subset \R^{n+1}$ be disjoint Borel sets and let $x\in\R^{n+1}$, $r>0$. 
	The associated coefficients $\ve(x,r)$, $\gamma(x,r)$, $a(x,r)$, satisfy
	$$2a(x,r) \leq \gamma(x,r) \leq 2\,\ve(x,r).$$
\end{lemma}

\begin{proof}
	First we show $2a(x,r) \leq \gamma(x,r)$. To shorten notation, we write $S=S(x,r)$. Then, for $i=+,-$,  we have
	\begin{align*}
		\HH^n\big((\Omega^i\triangle\, T_x(\Omega^{i,c}))\cap S\big) 
		& = \HH^n\big((\Omega^i\setminus T_x(\Omega^{i,c}))\cap S\big) + \HH^n\big( (T_x(\Omega^{i,c})\setminus \Omega^i)\cap S\big)\\
		& \geq \big(\HH^n(\Omega^i\cap S) - \HH^n(T_x(\Omega^{i,c})\cap S)\big)^+ \!\!+ 
		\big(\HH^n(T_x(\Omega^{i,c})\cap S) - \HH^n(\Omega^i\cap S) \big)^+\\
		& = \big|\HH^n(\Omega^i\cap S) - \HH^n(T_x(\Omega^{i,c})\cap S)\big| =
		\big|\HH^n(\Omega^i\cap S) - \HH^n(\Omega^{i,c}\cap S)\big|,
	\end{align*}
	where we wrote $t^+ = \max(h,0)$ for any $t\in\R$. Observe now that
	$$\big|\HH^n(\Omega^i\cap S) - \HH^n(\Omega^{i,c}\cap S)\big| =
	\big|\HH^n(\Omega^i\cap S) - (\HH^n(S) - \HH^n(\Omega^i\cap S)\big| = 
	\big|2\,\HH^n(\Omega^i\cap S) - \HH^n(S)\big|.
	$$
	Dividing the preceding estimates by $r^n$ and maximizing in $i=1,2$,  we get $\gamma(x,r)\geq 2\,a(x,r)$.
	
	Next we show that $\gamma(x,r) \leq 2\,\ve(x,r)$. Let $H$ be an arbitrary half-space such that $x\in\partial H$. Again to shorten notation, we write $S_H^i = S_H^i(x,r)$.
	Then we split
	\begin{equation}\label{eqa11}
		\HH^n\big((\Omega^i\triangle\, T_x(\Omega^{i,c}))\cap S\big) 
		= \HH^n\big((\Omega^i\triangle\, T_x(\Omega^{i,c}))\cap S_H^+\big)
		+ \HH^n\big((\Omega^i\triangle\, T_x(\Omega^{i,c}))\cap S_H^-\big).
	\end{equation}
	We have
	\begin{align*}
		\HH^n\big((\Omega^i\triangle\, T_x(\Omega^{i,c}))\cap S_H^+\big) 
		&\leq 
		\HH^n(S_H^+\cap \Omega^i \setminus T_x(\Omega^{i,c}))
		+ 
		\HH^n(S_H^+\cap T_x(\Omega^{i,c})\setminus \Omega^i)\\
		& \leq 
		\HH^n(S_H^+ \setminus T_x(\Omega^{i,c}))
		+ 
		\HH^n(S_H^+\setminus \Omega^i).
	\end{align*}
	Since
	$$\HH^n(S_H^+ \setminus T_x(\Omega^{i,c}))
	= \HH^n( T_x(S_H^-\setminus \Omega^{i,c})) = \HH^n(S_H^-\setminus \Omega^{i,c})\leq
	\HH^n(S_H^-\setminus \Omega^j)$$
	with $i\neq j$,
	we get
	$$\HH^n\big((\Omega^i\triangle\, T_x(\Omega^{i,c}))\cap S_H^+\big)
	\leq \HH^n(S_H^+\setminus \Omega^i) + \HH^n(S_H^-\setminus \Omega^j).$$
	Concerning the last summand on the right hand side of \rf{eqa11}, we have
	\begin{align*}
		\HH^n\big((\Omega^i\triangle\, T_x(\Omega^{i,c}))\cap S_H^-\big) 
		&\leq 
		\HH^n(S_H^-\cap \Omega^i \setminus T_x(\Omega^{i,c}))
		+ 
		\HH^n(S_H^-\cap T_x(\Omega^{i,c})\setminus \Omega^i)\\
		& \leq 
		\HH^n(S_H^-\cap \Omega^i )
		+ 
		\HH^n(S_H^-\cap T_x(\Omega^{i,c})).
	\end{align*}
	Now we use that $\HH^n(S_H^-\cap \Omega^i )\leq \HH^n(S_H^-\setminus \Omega^j )$ for $j\neq i$ and that
	$$\HH^n(S_H^-\cap T_x(\Omega^{i,c})) = \HH^n( T_x(S_H^+\cap\Omega^{i,c})) =\HH^n(S_H^+\setminus\Omega^i),$$ and the we obtain
	$$\HH^n\big((\Omega^i\triangle\, T_x(\Omega^{i,c}))\cap S_H^-\big)\leq 
	\HH^n(S_H^-\setminus \Omega^j ) + \HH^n(S_H^+\setminus\Omega^i).$$
	Plugging the preceding estimates into \rf{eqa11}, we obtain
	$$\HH^n\big((\Omega^i\triangle\, T_x(\Omega^{i,c}))\cap S\big) \leq 2\big[
	\HH^n(S_H^+\setminus\Omega^i) + \HH^n(S_H^-\setminus \Omega^j )\big].$$
	By the same arguments, we get
	$$\HH^n\big((\Omega^j\triangle\, T_x(\Omega^{j,c}))\cap S\big) \leq 2\big[
	\HH^n(S_H^+\setminus\Omega^i) + \HH^n(S_H^-\setminus \Omega^j )\big]$$
	(with $j \neq i$) and so
	$$\max_{k=+,-}\HH^n\big((\Omega^k\triangle\, T_x(\Omega_{k,c}))\cap S\big) \leq 2\big[
	\HH^n(S_H^+\setminus\Omega^i) + \HH^n(S_H^-\setminus \Omega^j )\big].$$
	Dividing by $r^n$ and taking the infimum over all $H$, we deduce
	$$\gamma(x,r) \leq 2\,\ve(x,r).$$
	
\end{proof}

\vvv
 
 We also set
 \begin{equation}\label{eqcoefg-a}
 	\mathfrak g(x,r)  = 
 	\frac1{r^{n+1}}\,\max_{i=+,-}\HH^{n+1}\big((\Omega^i\triangle\, T_x(\Omega^{i,c}))\cap B(x,r)\big).
 \end{equation}
 Notice that, for $i=+,-$,
 \begin{align*}
 	\frac1{r^{n+1}}\,\HH^{n+1}\big((\Omega^i\triangle\, T_x(\Omega^{i,c}))\cap B(x,r)\big) &
 	= \frac1{r^{n+1}}\int_0^r\HH^{n}\big((\Omega^i\triangle\, T_x(\Omega^{i,c}))\cap S(x,t)\big)\,dt\\
 	& = \frac1{r^{n+1}}\int_0^r \gamma(x,t)\,t^n\,dt.
 \end{align*}
 We have the following lemma. Its proof resembles that of Lemma \ref{lem1} below so we omit it.
 \begin{lemma}\label{lem1**}
 	There is a constant $\Cten$  such that for every $x\in\R^{n+1}$ and every $R>0$,  we have
 	$$\int_0^R \mathfrak g(x,r) ^2\frac{dr}{r}\leq \Cten \int_0^{R} \gamma(x,r) ^2\frac{dr}{r}
 	\leq 4 \Cten \int_0^{R} \ve(x,r) ^2\frac{dr}{r}.$$
 \end{lemma}
\noindent

 \vvv

 \subsection{The smoothed coefficients}

For a non-negative smooth function $\varphi:\R\to [0,\infty)$ satisfying
$\int_0^{\infty}\varphi(t) t^n\sqrt{\log(e+t)}dt<\infty,
$
set $\psi(x) = \varphi(|x|)$.  For $i=+,-$, consider
$$\apsi(x,r) = \Bigl|c_{\psi} - \frac{1}{r^{n+1}}\int_{\Omega^i}\psi\Bigl(\frac{x-y}{r}\Bigl)dy\Bigl|,\;\text{ where }c_{\psi} = \int_{\R^{n+1}_+}\psi(y)dy.
$$
Then define
\begin{equation}\label{e:def-mathfrak-1+2}
	\aps(x,r)^2 := \apsun(x,r)^2 + \apsdu(x,r)^2,
\end{equation}
and set
$$\Apsi(x)^2 = \int_0^1\aps (x,r)^2\,\frac{dr}{r}.
$$

\begin{lemma}\label{lem1}
	There is a constant $C_{1,\psi}$  such that for every $x\in\R^{n+1}$, every $R>0$, and $M\geq 1$, we have
	$$\int_0^R \apsi(x,r) ^2\frac{dr}{r}\leq C_{1,\psi} \int_0^{MR} \ve(x,r) ^2\frac{dr}{r}  + C_n\int_M^\infty \varphi(t)\, t^n\,\Big(\log^+\frac tM\Big)^{1/2}
	\,dt.$$
\end{lemma}

\begin{proof}
	Observe that, by integrating polar coordinates centered at $x$,
	\begin{align}\label{eqal82}
		\apsi(x,r) & = \left|c_\psi - \frac1{r^{n+1}}\int_0^\infty
		\varphi\Bigl(\frac{s}{r}\Bigl) \HH^n(S(x,s)\cap\Omega^i)\,ds\right|\\
		& = \frac1{r^{n+1}}\left|\int_0^\infty
		\varphi\Bigl(\frac{s}{r}\Bigl) \Big(\frac12\HH^n(S(x,r)) - \HH^n(S(x,s)\cap\Omega^i)\Big)\,ds\right|
		\nonumber\\
		& \leq \frac1{r^{n+1}}\int_0^\infty
		\varphi\Bigl(\frac{s}{r}\Bigl) a(x,s)\,s^n\,ds,\nonumber
	\end{align}
	where recall that $a(x,r)$ was defined in \eqref{e:def-a}.
	Squaring both sides and integrating over $r\in(0,R)$ and using the fact that $a(x,s)\leq \ve(x,s)$ (i.e. Lemma \ref{lemgamma}), yields,
	\begin{align*}
		\left(\int_0^R \apsi(x,r) ^2\frac{dr}{r}\right)^{1/2}
		& \leq \left(\int_0^R  \left(\frac1{r^{n+1}}\int_0^\infty
		\varphi\Bigl(\frac{s}{r}\Bigl)\,\ve(x,s)\,s^n\,ds\right)^2\,\frac{dr}r\right)^{1/2} \\
		&= \left(\int_0^R  \left(\int_0^\infty
		\varphi(t)\,\ve(x,tr)\, t^n\,dt\right)^2\,\frac{dr}r\right)^{1/2}\\
		&\stackrel{\text{Minkowski's inequality}}{\leq} \int_0^\infty
		\left(\int_0^R
		\ve(x,tr)^2\,\frac{dr}r\right)^{1/2}\varphi(t)\, t^n\,dt \\
		&= \int_0^\infty
		\left(\int_0^{tR}
		\ve(x,u)^2\,\frac{du}u\right)^{1/2}\,\varphi(t)\, t^n\,dt.
	\end{align*}
	For $t\geq1$ and some $M>1$, we split
	$$\int_0^{tR}
	\ve(x,u)^2\,\frac{du}u \leq \int_0^{MR}
	\ve(x,u)^2\,\frac{du}u + C\int_{MR}^{tR}\frac{du}u = \int_0^{MR}
	\ve(x,u)^2\,\frac{du}u + C\log^+ \frac tM.$$
	This bound certainly also holds for $t\in(0,1)$, so we get
	\begin{align*}
		\int_0^\infty
		\left(\int_0^{tR}
		\ve(x,u)^2\,\frac{du}u\right)^{1/2}\,\varphi(t)\, t^n\,dt &\leq  \left(\int_0^{MR}
		\ve(x,u)^2\,\frac{du}u\right)^{1/2}  \int_0^\infty
		\varphi(t)\, t^n\,dt \\
		&\quad + \int_M^\infty \varphi(t)\, t^n\,\Big(\log^+\frac tM\Big)^{1/2}
		\,dt,
	\end{align*}
	and the lemma follows.
\end{proof}

\vv

From now on, for the choice $\psi(x) = e^{-|x|^2}$, we denote
$$\mathfrak a^i(x,r) = \apsi(x,r) 
= \bigg|\frac1{r^{n+1}}\int_{\R^{n+1}_+} e^{-|y-x|^2/r^2}\,dy 
- \frac1{r^{n+1}}\int_{\Omega^i} e^{-|y-x|^2/r^2}\,dy\bigg|,$$
and then, as in \eqref{e:def-mathfrak-1+2},
\begin{equation}\label{e:def-mathfrak-1+2-exp}
\mathcal{A}(x)^2 = \int_0^1\mathfrak{a}(x,r)^2\,\frac{dr}{r}.
\end{equation}

\vv



\section{The three Main Lemmas and the proofs of Theorem \ref{teomain1} and Theorem \ref{teomain2}}\label{sec-proofs-teo12}
In this section we put together all our main tools, Main Lemmas \ref{mainlemma1}, \ref{mainlemma2} and \ref{mainlemma3}, to prove our main results, Theorems \ref{teomain1} and \ref{teomain2}.

\subsection{The first two main lemmas and the rectifiability of $E$}\label{sec-proof-teoA}

\noindent
The first main lemma is concerned with deriving some weak, but key, geometric and topological information from the smallness of the square functions $\ve_n(x,r)$ and $\mathfrak{a}(x,r)$. To state it, we need to introduce the following notion.

\begin{definition}\label{def-quasi-corkscrew}
	Let $B\subset\R^{n+1}$ be a ball and $\beta\in(0,1)$, and let $i=+,-$. We say that a ball $B'\subset\R^{n+1}$ is a $\beta$-\textit{almost corkscrew ball} for $B$ and
	$\Omega^i$ if
	$B'\subset B$, $\rad(B')\approx\rad(B)$, and 
	$$\HH^{n+1}(B'\setminus \Omega^h)\leq \beta\,\HH^{n+1}(B').$$
\end{definition}

\vvv

\begin{mlemma}\label{mainlemma1}
	Let $\Omega^+$, $\Omega^- \subset \R^{n+1}$ be  disjoint Borel sets. Fix $c_0 \in (0,1)$. For each appropriately\footnote{See Remark \ref{rem-mainlemma1} for what appropriate means here.} chosen triple of positive and small parameters $(\beta, \epsilon, \gamma)$ 
	there exists a $\delta=\delta(\beta, \epsilon, \gamma,c_0)>0$ such that the following holds true. Let $B_0=B(x_0, r_0)\subset\R^{n+1}$ and let $\mu$ a Radon measure with $(n,1)$-polynomial growth supported on a subset $E \subset  \overline{B_0}$ and such that 
	$
		\mu(B_0) \geq c_0 r_0^n.
	$
	 If
	\begin{equation}\label{mainlemma1-ve<delta}
		\int_0^{20r_0} \big(\ve_n(x,r)^2 + \mathfrak{a}(x,r)^2\big) \, \frac{dr}{r}\leq \delta \,\, \mbox{ for each }\,\, x \in E, 
	\end{equation}
	then:
	\begin{itemize}
		\item[\textup{(i)}] There exists a two balls $B^+, B^- \subset 6 B_0$ such that both $B^+$ and $B^-$ are $\beta$-almost corkscrew for $6B_0$ and $\Omega^+$ and $\Omega^-$, respectively. We have $\rad(B_i) \geq c_1 r_0$ for some universal constant $c_1 >0$.
		\item[\textup{(ii)}] It holds that
		\begin{equation}\label{beta-infty-small'}
			\beta_{\infty,E}(B_0) < \epsilon.
		\end{equation}
		\item[\textup{(iii)}] 
		Assume  that $x_0\in E$ and that		
		$E\cap (B_0\setminus \frac45B_0)\neq\varnothing$ and denote by 
		 $L_0$  a hyperplane infimizing $\beta_{\infty, E}(B_0)$.
	 Then the two components $D^+, D^-$ of $\frac12B\setminus \NN_{2\epsilon r_0}(L_B)$ satisfy \textup{(}up to relabeling\textup{)}
		\begin{equation}\label{e:mainlemma1-b}
			\HH^{n+1}(D^h \cap \Omega^i) \geq (1-\gamma) \HH^{n+1}(D^h) \,\,\mbox{ with } \, \, i =+,-.
		\end{equation}
	\end{itemize}
\end{mlemma}

\vspace{0.3cm}
Above, in (iii), the notation $\NN_t(A)$ stands for the $t$-neighborhood of a set $A$.

\noindent
The second main lemma builds upon the first one to conclude that whenever the square functions are small, then, in fact, the set $E$ intersects, in a large portion, a Lipschitz graph. Recall that the coefficient $\mathfrak{g}(x,r)$ was defined in \eqref{eqcoefg-a}. We also take an even $C^\infty$ function $\overline{\vphi}:\R\to\R$ such that $\one_{[-1,1]}\leq \overline{\vphi}\leq  \one_{[-1.1,1.1]}$, and we denote
\begin{equation}\label{eqassocpsi}
\psi(x) = \vphi(|x|)\quad\mbox{ for $x\in\R^{n+1}$.}
\end{equation}
Next we will consider the coefficients $\mathfrak{a}^\psi(x,r)$ associated to $\psi$.

\begin{mlemma}\label{mainlemma2}
	Let $\Omega^+,\Omega^-\subset\R^{n+1}$ be disjoint Borel sets.
	Fix $c_0\in(0,1)$, $\theta>0$, $\epsilon>0$ and $\delta>0$.  
	Let $B_0\subset \R^{n+1}$ be a ball, let $E$ be compact, and let $\mu$ be a measure with $n$-polynomial growth with constant $1$ supported on $E$
	 satisfying the following conditions:
	\begin{itemize}
		\item[(a)]  $\mu(B_0\cap E)\geq c_0\, \rad(B_0)^n$.
		\item[(b)]   For any ball $B$ with radius at most $200\, \rad(B_0)$ centered at $E$  such that
		$\mu(B)\geq \theta \, \rad(B)^n$, it holds $\beta_{\infty,E}(B)\leq \epsilon$ and moreover, if 
		$E\cap (B\setminus \frac45B)\neq\varnothing$
		and 
		$L_B$ stands for the hyperplane minimizing $\beta_{\infty,E}(B)$, the two components $D^+$, $D^-$ of $\frac12B\setminus L_B$ satisfy \textup{(}up to relabeling\textup{)}
		$$\HH^{n+1}(D^i \cap \Omega^i) \geq (1-\epsilon)\,\HH^{n+1}(D^i) \, \, \mbox{ for } \, i=+,-.$$
		
		\item[(c)] For $\psi$ chosen as in \rf{eqassocpsi}, it holds that
		$$\int_0^{20\rad(B_0)}\big(\mathfrak g(x,r)^2+ \mathfrak{a}^\psi(x,r)^2\big)\, \frac{dr}{r} \leq \delta \,\text{ for every }x\in E.$$
	\end{itemize}
	If $\theta$ is small enough in terms of $c_0$, $\epsilon$ is small enough in terms of $c_0$ and $\theta$, and $\delta$ sufficiently small depending on $\epsilon$, then there exists an $n$-dimensional Lipschitz graph $\Gamma$ with slope at most $1/10$ such that
	$$  \mu(\Gamma) \geq c \,\mu(B_0),$$
	for some absolute constant $0<c\leq 1$.
	Moreover, the slope of $\Gamma$ can be made arbitrarily small assuming $\delta$ small enough. 
\end{mlemma}

\vv

\noindent
Assuming the two main lemmas above, we are ready to prove our first main result, Theorem \ref{teomain1}. 
First we recall some basic facts about densities:  for a set $E\subset \R^{n}$ we set
$$\Theta^{s,*}(x,E) = \limsup_{r\to 0}\frac{\HH^s(E\cap B(x,r))} {\omega_s r^s}, \;\; \Theta^s_*(x,E) = \liminf_{r\to 0}\frac{\HH^s(E\cap B(x,r))}{\omega_s r^s}.
$$
These are called the upper and lower density of $E$ at $x$, respectively. If the two coincide, and thus the limit $\lim_{r \to 0} \HH^s(E \cap B(x,r))/\omega_s r^s$ exists, we denote this number by $\Theta^s(x,E)$ - the density of $E$ at $x$. When dealing with a Radon measure $\mu$, we write $\Theta^{s,*}(x,\mu), \Theta^s_*(x,\mu)$ and $\Theta^s(x,\mu)$. 
For the proof of the following simple lemma, see for example \cite[Theorem 6.1]{Mattila-gmt}.

\begin{lemma} If $E\subset \R^{n}$ satisfies $\HH^s(E)<\infty$, then
	$$2^{-s}\leq \Theta^{s,*}(x,E)\leq 1\text{ for }\HH^s\text{-a.e.} \, x\in E.
	$$
\end{lemma}

\vv

\subsubsection{\bf Proof of Theorem \ref{teomain1} using the Main Lemmas \ref{mainlemma1} and \ref{mainlemma2}}
	We prove first the following
	\begin{claim}\label{claim-A}
		Any Borel subset of finite $\HH^n$ measure of 
		$E_0:=\{x\in\R^{n+1}\, :\, \EE_n(x)<\infty\}$ is $n$-rectifiable.
	\end{claim} 
\begin{proof}
	Indeed, suppose suppose there is subset $E\subset E_0$ with finite $\HH^n$ measure which is not $n$-rectifiable.
	Then there exists another subset $E^*\subset E$ with $0<\HH^n(E^*)<\infty$ which is purely $n$-unrectifiable.
	Since the Carleson square function $\mathcal{E}_n(x)^2<\infty$ for $\HH^n$-a.e. $x\in E$, we have from Lemma \ref{lem1} and Lemma \ref{lem1**} that
	$$\int_0^1 \big( \ve_n(x,r)^2+ \mathfrak a(x,r)^2+\aps(x,r)^2 + \mfg(x,r)^2\big) \, \frac{dr}{r}<\infty\;\text{ for }\HH^n\text{-a.e. } x\in E^*,
	$$
	where $\psi$ is the function from Main Lemma \ref{mainlemma2}.
	By replacing $E^*$ by a subset with positive $\HH^n$ measure if necessary,
	we may assume that
	\begin{equation}\label{e:unifint}
		\lim_{s\to0}
		\int_0^s \big(\ve_n(x,r)^2+ \mathfrak a(x,r)^2+\apsi(x,r)^2 + \mfg(x,r)^2\big) \,\frac{dr}{r}=0\quad \mbox{ uniformly in $E^*$,}
	\end{equation}
	and
	\begin{equation}\label{e:Flingrowth}\HH^n(B(x,r)\cap E^*)\leq 2\omega_n r^n\;\text{ for all }x\in \R^{n+1} \text{ and } r>0.\end{equation}
	This second inequality is a consequence of the fact that $\Theta^{n,*}(x,E^*)\leq 1$ for $\HH^n$-a.e.\ $x\in E^*$. To lighten notation, for the sequel of the proof, we will relabel $E^*$ simply by $E$.
	
	For the choice $c_0=9^{-n}$, pick $\theta>0$, then $\epsilon\in (0,\theta)$ and then $\delta(\epsilon)$ small enough positive numbers so that Main Lemma \ref{mainlemma2} is applicable.  Then, having chosen three parameters $(\beta, \epsilon, \gamma)$ appropriately\footnote{See Remark \ref{rem-mainlemma1} below.}, reduce $\delta= \delta(\beta, \epsilon, \gamma)>0$ if necessary so that
	Main Lemma \ref{mainlemma1} holds with the choice $c_0$ replaced by $\theta$.

	Let $R$ be small enough so that
	\begin{equation}\label{Repseq}\int_0^{7R/\epsilon}\big( \ve_n(x,r)^2+ \mathfrak a(x,r)^2+\aps(x,r)^2 + \mfg(x,r)^2\big) \,\frac{dr}r\leq \delta.
	\end{equation}
	for all $x\in E$.
	
	Denote $\mu=(2\omega_n)^{-1}\HH^n|_E$.  Then $\mu$ has $n$-polynomial growth with constant $1$. Recalling that $\Theta^{n,*}(x,E)\geq 2^{-n}$ for $\HH^n$-a.e.\ $x\in E$,
	we can find a ball $B_0$ centered at $E$ with radius smaller than $R$ such that
	$\mu(B_0)\geq 9^{-n}\rad(B_0)^n = c_0\,\rad(B_0)^n$. 
	
	Now we can apply Main Lemma \ref{mainlemma2} with the measure $\nu =\mu|_{B_0}$, which satisfies $\nu(B_0)\geq  c_0\,\rad(B_0)^n$. The assumptions of this lemma are satisfied thanks to Main
	Lemma \ref{mainlemma1} (and an appropriate choice of $\gamma$ with respect to $\epsilon$).
	So there is a Lipschitz graph $\Gamma$ such that the set $E_1 = E\cap \Gamma$ satisfies $\HH^n(E_1)>0$. This yields a contradiction, since $E$ is purely $n$-unrectifiable.
 \end{proof}
	
	To complete the proof of Theorem \ref{teomain1} it suffices to show that any Borel subset $E\subset E_0$ has $\sigma$-finite measure $\HH^n$. For the sake of contradiction, suppose that $\HH^n|_E$ is non-$\sigma$-finite. We now appeal to a result of the first author and C. De Lellis stating that whenever a subset $A \subset \R^n$ is analytic and $\HH^k|_A$ is non-$\sigma$-finite for some $k$, then there exists a purely $k$-unrectifiable subset $A' \subset A$ with finite and positive $\HH^k$ measure (see \cite[Theorem 1.1]{DLF}). Since $E$ is assumed to be Borel, this means it is analytic, and thus we can find a purely $n$-unrectifiable subset $E' \subset E$ with $0< \HH^n(E') < \infty$. But this contradicts Claim \ref{claim-A} and thus Theorem \ref{teomain1} is proven. 
	
	\begin{rem}
		When $\Omega^+, \Omega^-$ we have a short argument to prove that $\HH^n|_E$ is $\sigma$-finite. We have included it in Appendix A. 
	\end{rem}

\begin{rem}
	That Theorem \ref{teomain1'} holds is an immediate consequence of these two main lemmas and their proofs.
\end{rem}

\subsection{Proof of Corollary \ref{coro-L1}}

 \begin{definition}\label{def-L1tan}
 	Given $x \in \partial \Omega^+ \cap \partial\Omega^-$, $H^+=\{y:(y-x) \cdot u>0\}$, we say that $x$ is an $L^1$ tangent point for the pair $\Omega^+$ and $\Omega^-$ if, for \begin{equation*}\Omega^+_{x,r}=\frac{\Omega^+-x}{r}+x
 	\end{equation*}
 	the function $\one_{\Omega^+_{x,r}}$ converges in $L^1_{loc}$ to $\one_{H^+}$, 
 	and if the equivalent local $L^1$ convergence for $\Omega^-$ and $H^-=\{y:(y-x) \cdot u<0\}$ is satisfied. 
\end{definition}

 \begin{proof}[Proof of Corollary \ref{coro-L1}]
 Let $\mu$ be the  measure constructed as in the proof of Theorem \ref{teomain1}. Then $\mu = c \HH^n|_{E'}$, where $E' \subset E$ satisfies 
\rf{e:unifint} and $c$ is a positive absolute constant chosen so that $\mu$ has $(n,1)$-polynomial growth. By Theorem \ref{teomain1}, $E'$ is $n$-rectifiable. Denote by $E^*$ the subset of $E'$ where $\Theta^n(x,\mu)=c=c\,\Theta^{n}(x, E)$ and such that for any $x \in E^*$, $E'$ has a unique approximate tangent $n$-plane $L_x$. Then, $\mu(E')= \mu(E^*)$. For the sequel of this proof, we relabel $E^*$ as $E$ to ease notation. For each $x \in E$ there exists an $r_x >0$ so that
\begin{equation}\label{e:999}
	E \cap \left( B(x,r) \setminus \tfrac45 B(x,r)\right) \neq \varnothing \, \mbox{ and } \, \mu(B(x,r)) > c \cdot  0.99 \omega_n r^n,
\end{equation}
for $r<r_x$.

Fix now $x \in E$, take a decreasing sequence $\{\delta_j\}$, and for each $j$ pick parameters $\beta_j,\epsilon_j,\gamma_j \rightarrow 0$ appropriately as required by Main Lemma \ref{mainlemma1}.
 Moreover, for each $j \in \N$, choose $r_j>0$ such that \begin{equation*}
    \int_0^{20r_j} \big(\ve_n(y,s)^2 + \mathfrak{a}(y,s)^2\big) \, \frac{ds}{s} \leq \delta_j \,\, \mbox{ for every $y \in E \cap \overline B(x,r_j)$},
 \end{equation*} 
 and such that, for any $r \leq r_j$, \eqref{e:999} is satisfied.
 In particular, for each of these pairs $(x,r)$, $r< r_j$, the measure $\mu$ satisfies the hypothesis of Main Lemma \ref{mainlemma1}, (iii), and thus for every $r<r_j$ there exists a plane $L_r$ so that 
\begin{equation*}
\HH^{n+1}(H_{L_r}^\pm \cap \Omega^\pm \cap B(x,r/2)) \geq (1-C\gamma_j) \HH^{n+1}(H_{L_r}^\pm \cap B(x,r/2))
\end{equation*}
where $H_{L_r}^\pm$ are the complementary half-space associated to $L_r$.
It is not difficult to see that $L_r \to L_x$ as $r \to 0$ in the sense of Hausdorff convergence by recalling that $\beta_{\infty, E}(x,r; L_r)< 2\epsilon_j$.


For $h=+,-$, the subset 
$\left(H_{L_x}^h \cap \Omega^h \cap B(x,r) \right)$
is contained in
\begin{align*}
  \left(H_{L_r}^h \cap \Omega^h \cap B(x,r) \right)\cup \{y \in B(x,r)\,:\, \dist(y,L_x) \leq c(n) d_{H}(L_x \cap B(x,r),L_r \cap B(x,r))\}.
\end{align*}
This means that for every $r \leq r_j'$, where $r_j'$  is taken smaller than $r_j/2$ so that the Hausdorff distance between the planes $L_x, L_r$ in the balls $B(x,r)$ is smaller than $\gamma_j r$, we have
\begin{equation*}
\HH^{n+1}(H_{L_x}^h \cap \Omega^h \cap B(x,r)) \geq (1-c(n)\gamma_j) \HH^{n+1}(H_{L_x}^h \cap B(x,r)).
\end{equation*}
This establishes local $L^1$ convergence.
 
\end{proof}
\vv

\subsection{The third main lemma and existence of true tangents}

In this subsection, we prove Theorem \ref{teomain2}
using Theorem \ref{teo-ve<alpha}, Theorem \ref{teomain1} and our third main lemma, which we now state.

\begin{mlemma}\label{mainlemma3}
Let $\Omega^+$, $\Omega^- \subset \R^{n+1}$ be open and disjoint, and $F=\R^{n+1} \setminus (\Omega^+ \cup \Omega^-)$. Fix $c_0\in (0,1)$. For any $\tau>0$, there exists $\delta>0$ small enough such that the following holds. Let 
$\Gamma\subset\R^{n+1}$ be a Lipschitz graph with slope at most $\delta$, 
let  $B_0=B(x_0, r_0)$ be a ball centered in $E$, where $E\subset \Gamma\cap B_0$ is such that $\HH^n(E \cap \tfrac14 B_0) \geq c_0 r_0^n$ and $E \cap (B_0 \setminus \tfrac{4}{5}B_0) \neq \varnothing$.
If
	$\Omega^+\cup\Omega^-$ satisfies the $s$-CDC for some $s \in (n-1, n)$ and
	\begin{equation*}
		\int_0^{1000 \rad (B_0)} \big(\ve_n(x,r)^2 + \ve_{s-1}(x,r)^2+ \mathfrak{a}(x,r)^2 \big)\,\frac{dr}r < \delta \,\, \mbox{ for all } \, x \in E,
	\end{equation*}
	where $a, \mathfrak{c}_0$ in the definition of $\ve_{s-1}$ are chosen appropriately depending on $\tau$, then
\begin{equation}\label{small-beta-F}
	\beta_{\infty, F}(\tfrac14B_0) < \tau.
\end{equation}
\end{mlemma}

\vv

\begin{proof}[Proof of Theorem \ref{teomain2}]
We are now in the position to prove Theorem \ref{teomain2}. From Theorem \ref{teomain1}, we know that $E$ is $n$-rectifiable. Hence for $\HH^n$ almost all points, it admits an approximate tangent.
By replacing $E$ with a subsets of positive measure, by Theorem \ref{teo-ve<alpha}
we can assume that
\begin{equation}\label{e:41}
	\lim_{t \to 0} \int_{0}^t \big(\ve_n(x,r)^2  + \ve_{s-1}(x,r)^2  + \mathfrak{a}(x,r)^2\big) \, \frac{dr}{r} = 0 \,\, \mbox{ uniformly on $E$}.
\end{equation}

To reach a contradiction, let $T_E$ be the subset of tangent points of $E$.
Let $x \in E \setminus T_E$ and assume moreover that $E$ has an approximate tangent at $x$. Then there exists an $a \in (0,1)$ such that for all radii $r>0$ we have that
\begin{equation}\label{e:40}
		(\partial \Omega^+ \cap \partial \Omega^-)\cap B(x,r) \cap X_{a}(x,u_x) \neq \varnothing,
\end{equation}
where $u_x$ is the unit vector perpendicular to the approximate tangent of $E$ at $x$. 
We write $x\in H_{a}$ if this happens. Then we have
$$E\setminus T_E = \bigcup_{k\geq1} H_{1/k}.$$

We will show now that $\HH^n(H_{1/k})=0$ for every $k\geq 1$. To this end, for a given $k$, by the rectifiability of $H_{1/k}$ we can write
$$H_{1/k} =  \bigcup_{j \geq 1}\big(\Gamma_j^k\cap H_{1/k}\big) \cup Z_k,$$
where $\HH^n(Z_k)=0$ and each $\Gamma_j^k$ is a Lipschitz graph with slope at most $\delta_k$, where $\delta_k>$ will be chosen soon, depending on $k$.

Next, set $E_j^k := \Gamma_j^k \cap H_{1/k}$ and suppose that $\HH^n(E_j^k)>0$ for some $j \geq 1$. Let moreover $x\in E_j^k:= \Gamma_j^k\cap H_{1/k}$ such that $\Theta^n(x,E_j^k)=1$.
Note that \eqref{e:40} implies that there exists a sequence of radii $r_i \to 0$ so that 
\begin{equation}\label{e:47}
	F\cap   A(x, \tfrac12 r_i, r_i) \cap X_{1/k}(x,u)= 	(\partial \Omega^+ \cap \partial \Omega^-) \cap A(x, \tfrac12 r_i, r_i) \cap X_{1/k}(x,u) \neq \varnothing.
\end{equation}
By \rf{e:41} there exists $t_0>0$ small enough such that
$$\int_{0}^{1000t_0} \big(\ve_n(x,r)^2  + \ve_{s-1}(x,r)^2  + \mathfrak{a}(x,r)^2\big) \, \frac{dr}{r}\leq \delta_k
\quad\mbox{ for all $x\in E_j^k$.}$$
Then we can choose $t_0>0$ sufficiently small so that if $r_i\leq t_0$, we have
\begin{align}
	&	\int_{0}^{4000r_i} \big(\ve_n(x,r)^2 + \mathfrak{a}^\pm(x,r)^2 + \ve_{s-1}(x,r)^2\big)\,\frac{dr}r < \delta_k, \label{e:42}\\
	& \HH^n(B(x,r_i)\cap E_j^k) \geq c_0 r_i^n,    \label{e:43}\\
	& E_j^k \cap (B(x, 4r_i) \setminus \tfrac45 B(x, 4r_i))\neq \varnothing  \mbox{ and  } \label{e:43b}\\ 
	& \HH^n(B(x,4r_i)\cap E \cap X_{a/2}(x,u)) \ll \HH^n(B(x,4r_i) \cap E).\label{e:44}
\end{align}
For each such $r_i$, $E_j^k$ satisfies the hypotheses of Main Lemma \ref{mainlemma3} (with $B_0= B(x,4r_i)$). Hence we deduce that, for any given $\tau>0$, if $\delta_k$ is small enough, 
$$\beta_{\infty, F}(B(x,r_i)) < \tau\ll \frac1k.$$
It is immediate to check that this contradicts \rf{e:47} for $\tau$ small enough. Indeed,   
thanks to \eqref{e:44} and \eqref{e:43} we can conclude that there are $n+1$ points $y_1,...,y_{n+1}$ contained in $E \cap B(x, r_i) \setminus X_{a/2}(x,u)$ and such that they are linearly independent with good constants (i.e., such that the $n$-dimensional volume of their convex hull is comparable to $r_i^n$). Then, if $z\in F\cap   A(x, \tfrac12 r_i, r_i) \cap X_{1/k}(x,u)$, the volume of the simplex with vertices $z,y_1,\ldots,y_{n+1}$ is bounded away from $0$, which implies that 
$\beta_{\infty, F}(B(x,r_i)) \gtrsim  \frac1k$.
\end{proof}


\part{Weak geometric and topological results}

This part of the paper is devoted to the proof of the following first main lemma. Throughout Part~I and Part II, we will keep writing $\ve(x,r)=\ve_n(x,r)$.

\section{Existence of almost corkscrews}\label{sec-cork}

For the convenience of the reader, we state again the first main lemma:

\begin{lemma}\label{mainlemma1-text}
	Let $\Omega^+$, $\Omega^- \subset \R^{n+1}$ be disjoint Borel sets. Fix $c_0 \in (0,1)$. For each appropriately\footnote{See Remark \ref{rem-mainlemma1} for what appropriate means here.} chosen triple of positive and small parameters $(\beta, \epsilon, \gamma)$ 
	there exists a $\delta=\delta(\beta, \epsilon, \gamma)>0$ such that the following holds true. Let $B_0=B(x_0, r_0)$ be a ball and let $\mu$ be a Radon measure  with $(n, C_0)$-polynomial growth supported on a subset $E \subset  \overline{B_0}$ and such that 
	\begin{equation*}
		\mu(B_0) \geq c_0 r_0^n.
	\end{equation*}
	If
	\begin{equation}\label{mainlemma1-ve<delta'}
		\int_0^{20r_0} \big(\ve(x,r)^2 + \mathfrak{a}^\pm(x,r)^2\big) \, \frac{dr}{r}< \delta \,\, \mbox{ for each }\,\, x \in E, 
	\end{equation}
	then 
	\begin{itemize}
		\item[\textup{(i)}] There exists a two balls $B^+, B^- \subset 6 B_0$ such that both $B^+$ and $B^-$ are $\beta$-almost corkscrew for $6B_0$ and $\Omega^+$ and $\Omega^-$, respectively. Further, we have $\rad(B^\pm) \geq c_1 r_0$ for some universal constant $c_1 >0$.
		\item[\textup{(ii)}] It holds that
		\begin{equation}\label{beta-infty-small}
			\beta_{\infty,E}(B_0) < \epsilon.
		\end{equation}
	\item[\textup{(iii)}] 
		Assume  that $x_0\in E$ and that		
		$E\cap (B_0\setminus \frac45B_0)\neq\varnothing$ and denote by 
		 $L_0$  a hyperplane infimizing $\beta_{\infty, E}(B_0)$.
	 Then the two components $D^+, D^-$ of $\frac12B\setminus \NN_{2\epsilon r_0}(L_B)$ satisfy \textup{(}up to relabeling\textup{)}
		\begin{equation}\label{e:mainlemma1-b'}
			\HH^{n+1}(D^h \cap \Omega^i) \geq (1-\gamma) \HH^{n+1}(D^i) \,\,\mbox{ with } \, \, i =+,-.
		\end{equation}
	\end{itemize}	
\end{lemma}
\vv

\begin{rem}\label{rem-mainlemma1}
    The choice of parameters is as follows. In Lemma \ref{lemcork1} we show that for any $\beta$ there exists a $\delta_1(\beta)$ so that (i) holds. Next, in Lemma \ref{lem-beta<epsilon}, we show that for any $\epsilon$ there exists a $\delta_2(\epsilon)$ so that (ii) holds; thus for both (i) and (ii) to hold for the same $\delta$, it suffices to pick $\delta_3(\beta, \epsilon)= \min(\delta(\beta), \delta(\epsilon))$ (see Remark \ref{rem-beta<epsilon}). In Lemma \ref{lem-corkscrew-everywhere}, we show that, given $\lambda>0$, as long as $\epsilon \ll \lambda$, we can find a $\delta_4=\delta_4(\lambda, \epsilon)$ so that if \eqref{mainlemma1-ve<delta'} holds with a $\delta_5=\delta_5(\beta, \epsilon, \lambda)$ satisfying $\delta_5<\min(\delta_4, \delta_3)$, then (i), (ii) and the conclusions of Lemma \ref{lem-corkscrew-everywhere} holds with this choice of parameters. Finally, in Lemma \ref{lem-99domain}, we show that for any $\gamma>0$, conclusion (iii) holds as long as we choose $\delta_6=\delta(\beta, \epsilon, \lambda, \gamma)$ with $\delta_6 \leq \delta_6$, $\delta_6 \ll \beta$ and $\delta_6 \ll \gamma$.
	In the end, we take $\delta=\delta_6$.
\end{rem}

\vv


In the remainder of this section we will prove Lemma \ref{mainlemma1-text}(i). Recall the definition of $\beta$-almost corkscrew in Definition \ref{def-quasi-corkscrew}.

\begin{lemma}\label{lemcork1}
Let $\beta,\delta>0$ and let $x,y\in\R^{n+1}$ be such that $|x-y|=R$ and
$$\int_0^{8R} \big(\ve_n(x,t)^2+\ve_n(y,t)^2\big)\,\frac{dt}t \leq \delta.$$
If $\delta=\delta(\beta)>0$ is small enough, then there exist $B^+,B^-\subset B(x,6R)$ such that each $B^i$ is a $\beta$-almost corkscrew ball
for $B(x,6R)$ and $\Omega^i$, with $\rad(B^i)\geq c_1R$, where $c_1$ is an absolute constant.
\end{lemma}
\noindent

We need an intermediate result. 

\begin{lemma}\label{lemaux99}
Let $\theta\in (0,1/2)$.
Let $I\subset \R$ be an interval and consider a measurable subset $G\subset I$ such that $|G|\geq \cfive\,|I|$. Then there exists an interval $J\subset I$ which satisfies the following properties:
\begin{itemize}
\item[(a)] $\ell(J) \geq c(\cfive,\theta)\,\ell(I)$,
\item[(b)] Any dyadic descendant $\wt J$ of $J$ such that $\ell(\wt J) \geq \theta\,\ell(J)$ satisfies 
$$|G\cap \wt J|\geq \frac{\cfive}2\,|\wt J|.$$
\end{itemize}
\end{lemma}

We prove now the existence of almost corkscrew balls assuming Lemma \ref{lemaux99}.
Let us sketch the proof in words first, to aide the intuition.

By pigeon-holing we can assume that the best approximating half spaces are very close to one half space (call it $H^+_{u_k}$) for many radii. Denote this set of radii by $A_k$ and call the union of spheres $S(x,t)$ for $t \in A_k$ by $G_x$ (see Figure \ref{figSec5}). Then there is a location, represented by the ball $B$ in Figure \ref{figSec5}, where $B \cap S(x,t) \subset H^+_{u_k}$, $r \in A_k$. Since for all these radii we can also assume that $\ve_n(x,t) \ll 1$, this means that $S(x,t) \cap B$ will be essentially void of $\Omega^-$ for each $t \in A_k$. Basically, we can imagine $G_x \cap B$ to be contained in $\Omega^+$.

Let us now look at the spheres $S(y,t)$ which intersect $B$. Since $A_k$ is rather large, this means that $S(y, r) \cap G_x$ is large for most spheres $S(y,r)$. Hence we can find $n+1$ (or 2 in Figure \ref{figSec5}) linearly independent cones $X(v_i, \gamma)$ of aperture $\gamma$ and axis $v_i$, such that if $\Delta= X(v_i, \gamma) \cap S(y,r)$ then $\Delta$ contains plenty of $G_x$, and thus of $\Omega^+$. Since we know that $\ve_n(y, r) \ll 1$, this forces $\Delta$ to intersect $H^+_{y, r}$. By then, for any $n+1$-tuple of points $y_0,...,y_n$, its geodesic convex hull\footnote{ That is, the convex hull in $S(y,r)$.} must lie in $H^+_{y,r}$. Since $\ve_n(y,r)\ll 1$, this in turn implies that most of this hull must belong to $\Omega^+$ - and this is true for essentially all spheres centered at $y$ and intersecting $B$. Hence by Fubini most of $B$ is in $\Omega^+$.

\begin{figure}[!tbp]
	\centering

	\begin{subfigure}[b]{0.6\textwidth}
		\includegraphics[width=\textwidth]{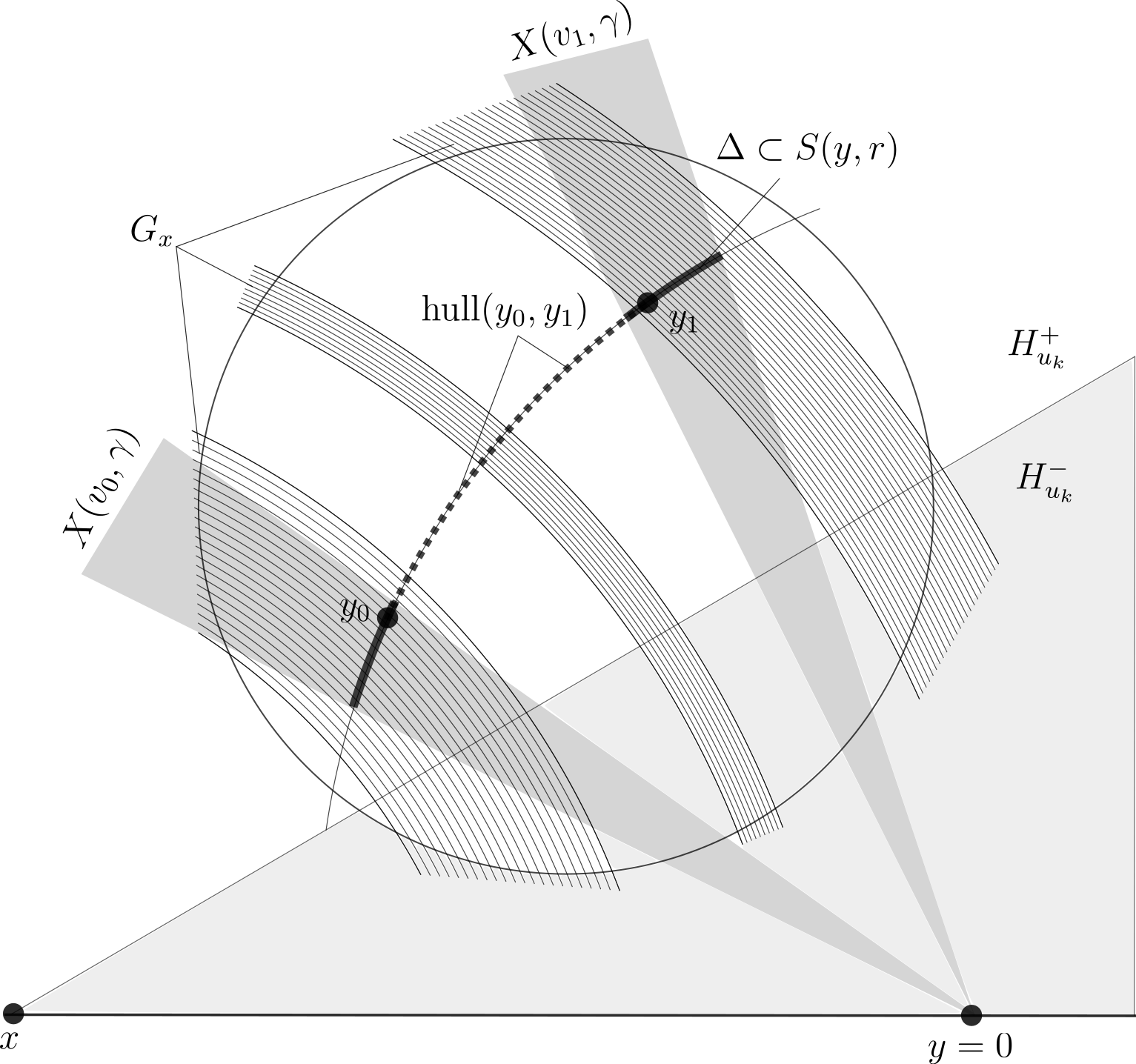}
	\end{subfigure}

 	\vspace{1cm}

	\begin{subfigure}[b]{\textwidth}
		\hspace{1cm}
			\begin{minipage}[b]{0.45\textwidth}
			\includegraphics[width=\textwidth]{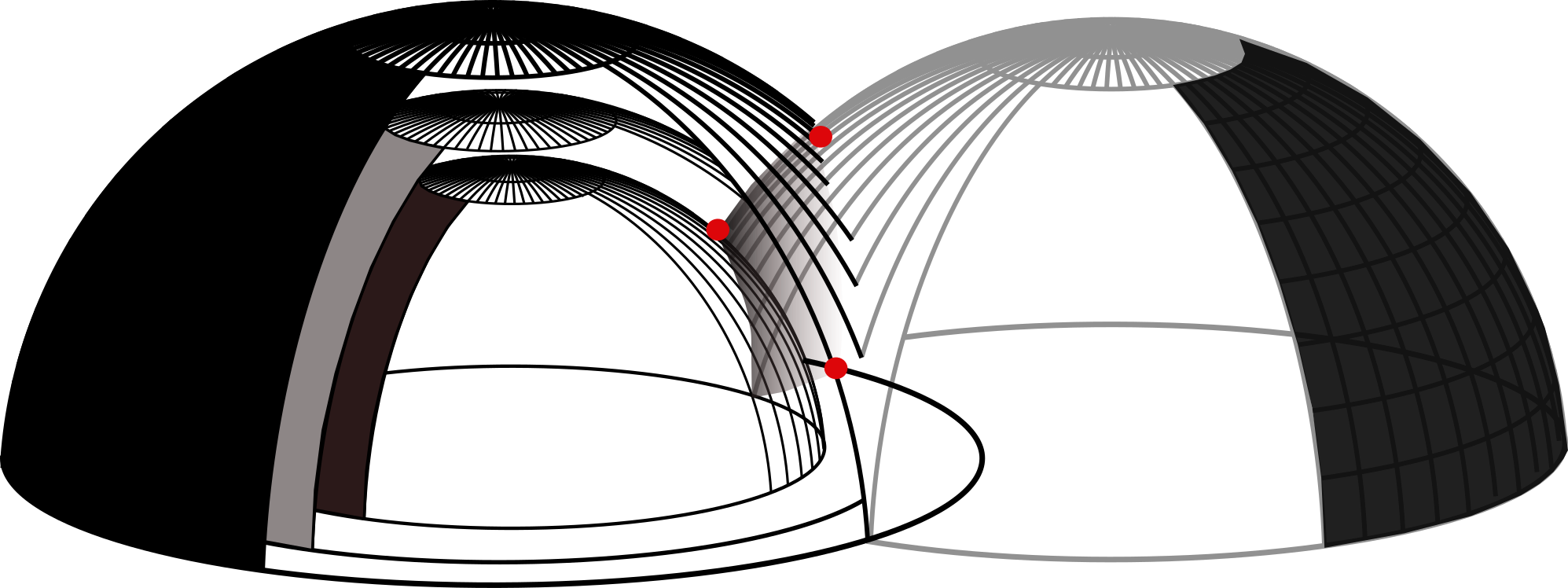}
		\end{minipage}
		\hspace{2cm}
		\begin{minipage}[b]{0.45\textwidth}
			\includegraphics[scale=0.07]{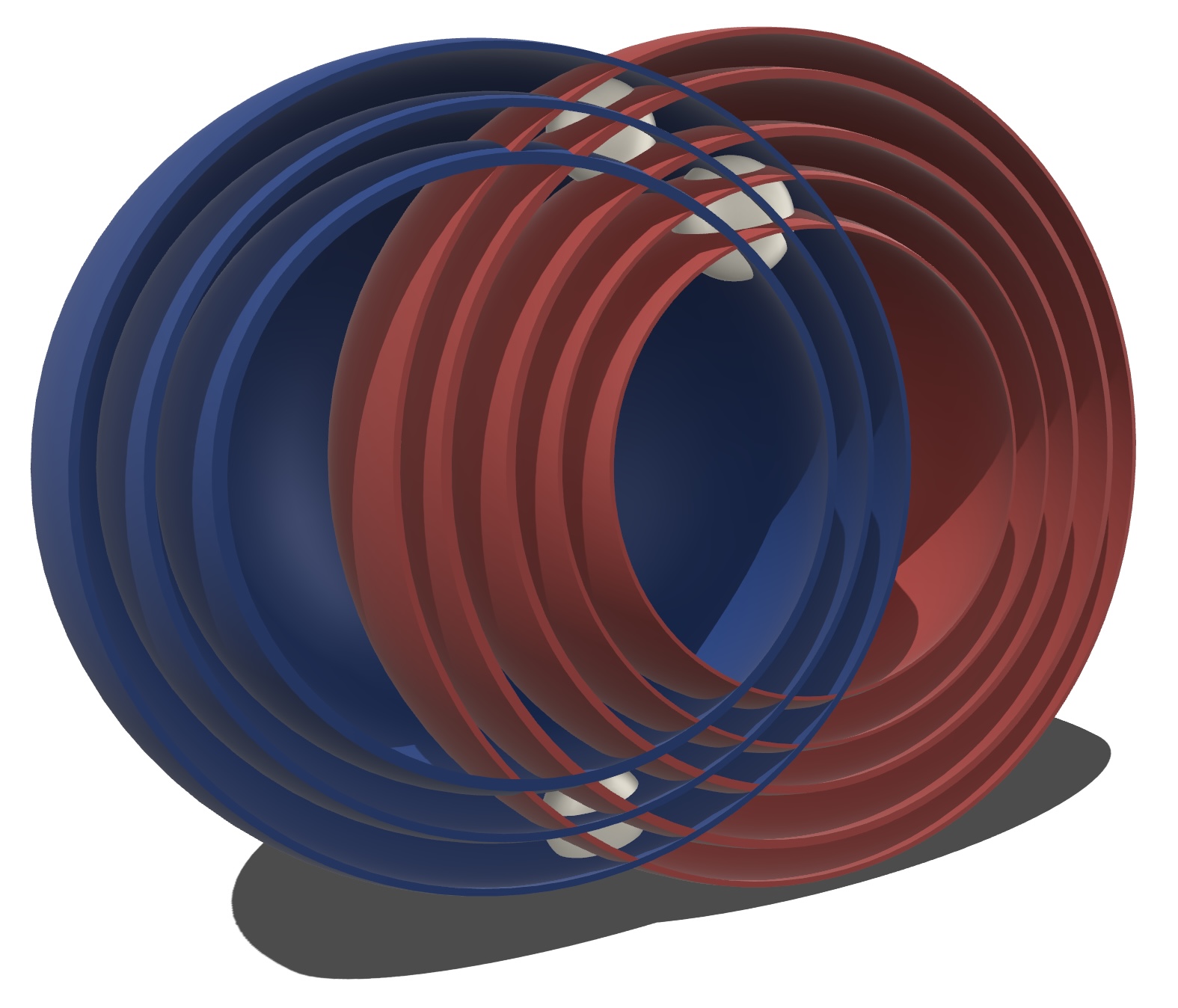}
		\end{minipage}
		
	\end{subfigure}

\caption{The set up of the proof of Lemma \ref{lemcork1}. On top, the labelled 2d case. At the bottom left, the 3d case: the spheres on the left represent $G_x$ and the three red dots the points $y_1,y_2,y_3$. The shaded area will contain the hul$\{y_1,y_2,y_3\}$. At the bottom right, another 3d rendering.} 
\label{figSec5}
\end{figure}

\begin{proof}[Proof of Lemma \ref{lemcork1}]
We will just show the existence of the almost corkscrew ball $B^+$ since the arguments for $B^-$ are analogous.
For a fixed $x\in\R^{n+1}$, we identify each half-space $H\subset \R^{n+1}$ such that $x\in\partial H$ with the inner unit vector to $\partial H$, which we denote by $u_H$. Then, for some $\tau\in (0,1/10)$ to be fixed below, we consider a maximal
$\tau$-separated family of unit vectors $u_i$, $1\leq i \leq N_\tau$, so that $\bS^n\subset \bigcup_i B(u_i,\tau)$.

We let
$$A_i = \big\{t\in [2R,3R]:u_{H_{x,t}}\in B(u_i,\tau)\big\},$$
where $u_{H_{z,t}}$ is the unit vector identifying the subspace $H_{z,t}$ that minimizes $\ve_n(z,t)$ (this might not be unique, in which case we simply choose one).
So we have $[2R,3R]=\bigcup_{i=1}^{N_\tau} A_i$. We choose $k$ such that $A_k$ has maximal Lebesgue measure among all the $A_i$'s, so that 
$$|A_k|\geq \frac1{N_\tau}\,R.$$

For $I=[2R,3R]$ and
some constant $\theta\in (0,1/2)$ to be fixed below, consider a subinterval $J\subset I$ satisfying the properties given by Lemma
\ref{lemaux99}, with $A_k$ in place of $G$ in the lemma. Denote by $r_J$ the center of $J$.
Notice that 
\begin{equation}\label{eqcork1}
\dist_H(H_{u_k} \cap B(x,8R),H_{x,r}\cap B(x,8R)) \lesssim \tau\,R\quad \mbox{ for all $r\in A_k$.}
\end{equation}
Then we infer that, for some absolute constant $\tau$ small enough, there is a geodesic ball $\Delta_0\subset \bS^n$ of radius $1/10$
such that
\begin{equation}\label{eqtrans00}
\Delta_{x,r}:= \{z\in S(x,r):r^{-1}(z-x)\in\Delta_0\}\subset 
H_{x,r}\quad \mbox{ for every $r\in A_k$,}
\end{equation} 
and moreover 
\begin{equation}\label{eqtrans1}
\dist(\Delta_{x,r},L_{x,y})\gtrsim R \quad \mbox{ for all $r\in [R/2,4R]$},
\end{equation}
where $L_{x,y}$ is the line through $x,y$.

Let $\xi_J \in S(x,r_J)$ be the geodesic center of $\Delta_{x,r_J}$
(recall that $r_J$ is the center of the interval $J$). 
Without loss of generality we assume that $y=0$.
Observe that 
$\xi_J\in S(0,|\xi_J|) \cap S(x,r_J)$ and $|\xi_J| = |\xi_J-y|\approx R$. Further, the condition \rf{eqtrans1} ensures that $\dist(\xi_J,L_{x,y})\gtrsim R$ and so the intersection of both spheres
is ``transversal" near $\xi_J$\footnote{Roughly speaking, this is because the points $x,y,\xi_J$ are far from being colinear.}.
From this fact and the condition (b) in Lemma \ref{lemaux99} with $\theta$ small enough, 
we infer that for any $\gamma>0$ there exists some small constant $\ctwo(\gamma)\in (0,1/10)$ such that, for every $r>0$ and every geodesic ball $\Delta\subset S(0,r)=S(y,r)$ such that
\begin{enumerate}[label=\roman*.]
\item $|r-|\xi_J||\leq \ctwo R$,
\item $\dist(\xi_J,\Delta) \leq \ctwo R$,
\item $\rad_{S(0,r)}(\Delta) \geq \gamma \ctwo\,R,$
\end{enumerate}
it holds
\begin{equation}\label{eqcond82}
\HH^n(\Delta \cap G_x)\gtrsim_{\tau,\gamma} \HH^n(\Delta),
\quad \mbox{ 
where \quad$G_x = \bigcup_{t\in A_k} S(x,t)$.}
\end{equation}

Now we choose unit vectors $v_0,\ldots,v_{n}\in S(0,1)\cap B(|\xi_J|^{-1}\xi_J,\frac12 \ctwo)$
 such that
$\dist(v_i,L_{v_0,\ldots,v_{i-1}})\gtrsim_{\ctwo} 1$ for $1\leq i\leq n$, where $L_{v_0,\ldots,v_{i-1}}$ stands for the $(i-1)$-plane through 
$v_0,\ldots,v_{i-1}$, and we consider the (one-sided) cones
$$X(v_i,\gamma) = \{z\in\R^{n+1}\setminus\{0\}: ||z|^{-1}z - v_i|\leq \gamma/2\},$$
for $i=0,\ldots,n$.
Choosing $\gamma$ small enough, it is clear that, for all $r\approx R$ and all $y_i\in X(v_i,\gamma)\cap S(0,r)$, the geodesic convex hull of $y_0,\ldots, y_{n}$ in $S(0,r)$, which we denote by ${\rm hull}_{S(0,r)}(y_0,\ldots, y_{n})$,
satisfies
\begin{equation}\label{eqfk35}
\HH^n({\rm hull}_{S(0,r)}(y_0,\ldots, y_{n})) \gtrsim_{\ctwo,\gamma} r^n\approx_{\ctwo,\gamma} R^n.
\end{equation}

To shorten notation, we write $r_1=|\xi_J|-\ctwo R$ and $r_2=|\xi_J|+\ctwo R$. Then, for $j=0,\ldots,n$ we have
\begin{align*}
\HH^{n+1}(A(0,&r_1,r_2) \cap X(v_j,\gamma) \cap G_x\cap\Omega^+)\\&  =  
\HH^{n+1}(A(0,r_1,r_2) \cap X(v_j,\gamma) \cap G_x) -
\HH^{n+1}(A(0,r_1,r_2) \cap X(v_j,\gamma) \cap G_x\setminus\Omega^+).
\end{align*}
Regarding the last term, notice that $A(0,r_1,r_2) \cap X(v_j,\gamma)$ is contained in a small neighborhood of $\xi_J$ 
and then, assuming $\ctwo$ small enough with respect to $\tau$, 
$$A(0,r_1,r_2) \cap X(v_j,\gamma) \cap G_x \cap S(x,t) \subset H_{x,t}\quad \mbox{ for every $t\in A_k$.}$$
Therefore, by the definition of $G_x$,
\begin{align}\label{eqal734}
\HH^{n+1}(A(0,r_1,r_2) \cap X(v_j,\gamma) \cap G_x\setminus\Omega^+) & \leq
\int_{A_k} \HH^n(S(x,t) \cap H_{x,t}\setminus \Omega^+)\,dt\\
&\lesssim R^n\,\int_{A_k}\ve_n(x,t)\,dt \nonumber\\
& \leq R^n
\left(\int_0^{4R} \ve_n(x,t)^2\,\frac{dt}t\right)^{1/2}\left(\int_0^{4R} t\,dt\right)^{1/2} \nonumber\\& \lesssim \delta^{1/2}R^{n+1}.\nonumber
\end{align}

\noindent
We denote  
$${\rm Bad}_j = \big\{r\in [r_1,r_2]: \HH^{n}(S(0,r) \cap X(v_j,\gamma) \cap G_x\setminus\Omega^+) \geq \delta^{1/4}\,\HH^{n}(S(0,r) \cap X(v_j,\gamma) \cap G_x)\big\}.$$
In a sense, ${\rm Bad}_j$ is the set of radii for which
 there is a lot of mass from $S(0,r) \cap X(v_j,\gamma) \cap G_x$ 
 outside $\Omega^+$.
By \rf{eqcond82} and the condition iii above, we have
$$\HH^n(S(0,r)\cap X(v_j,\gamma) \cap G_x) \gtrsim_{\gamma} \HH^n(S(0,r)\cap X(v_j,\gamma)) \approx_{\gamma} R^n.
$$
Then, by Chebyshev and \rf{eqal734}, we obtain
\begin{align*}
\HH^1({\rm Bad}_j) & \leq \delta^{-1/4}\int_{r_1}^{r_2}\frac{\HH^{n}(S(0,r) \cap X(v_j,\gamma) \cap G_x\setminus\Omega^+)}{\HH^{n}(S(0,r) \cap X(v_j,\gamma) \cap G_x)}\,dr \\
&\lesssim_\gamma \frac{\delta^{-1/4}}{R^n}\int_{r_1}^{r_2}\HH^{n}(S(0,r) \cap X(v_j,\gamma) \cap G_x\setminus\Omega^+)\,dr\\
& \approx_\gamma \frac{\delta^{-1/4}}{R^n}
\HH^{n+1}(A(0,r_1,r_2) \cap X(v_j,\gamma) \cap G_x\setminus\Omega^+) \lesssim_\gamma \delta^{1/4}\,R.
\end{align*}

Let 
$${\rm Good}_0 = \big\{r\in [r_1,r_2]:\ve_n(y,r) \leq\delta^{1/4}\big\}.$$
By Chebyshev and Cauchy-Schwarz, we have
\begin{align}\label{eqal843**}
\HH^1\big([r_1,r_2]\setminus {\rm Good_0}\}\big) & \leq \delta^{-1/4} \int_0^{4r} \ve_n(y,t)\,dt \\
& \leq 
\delta^{-1/4} \left(\int_0^{4R} \ve_n(y,t)^2\,\frac{dr}r\right)^{1/2}\left(\int_0^{4R} r\,dr\right)^{1/2} \lesssim\delta^{1/4}R.\nonumber
\end{align}
Finally, let
 ${\rm Good}={\rm Good}_0\setminus \bigcup_{j=0}^n {\rm Bad}_j$. From the last estimates we deduce 
\begin{equation}\label{eqfk35.5}
\HH^1( [r_1,r_2]\setminus {\rm Good}) \leq \HH^1( [r_1,r_2]\setminus {\rm Good}_0) + \sum_{j=0}^{n}
\HH^1( {\rm Bad}_j)\lesssim \delta^{1/4}R.
\end{equation}
Observe also that for $r\in{\rm Good}$, since $r\not\in {\rm Bad}_j$,
it hol ds
\begin{align*}
\HH^{n}(S(0,r) \cap X(v_j,\gamma)\cap\Omega^+) &\geq
\HH^{n}(S(0,r) \cap X(v_j,\gamma) \cap G_x\cap\Omega^+) \\ 
&\geq (1 - \delta^{1/4})\,\HH^{n}(S(0,r) \cap X(v_j,\gamma) \cap G_x)\gtrsim_\gamma r^n.
\end{align*}
Since $\ve_n(0,r)=\ve_n(y,r)\ll1$ for $r\in{\rm Good}$, we deduce that the half-sphere $S_{y,r}^+= S(y,r) \cap H_{y,r}$ satisfies
$$S_{y,r}^+ \cap X(v_j,\gamma)\neq \varnothing.$$
 Since $S_{y,r}^+$ is geodesically convex in $S(y,r) \equiv S(0,r)$ we infer that, for all $y_j\in S_{y,r}^+ \cap X(v_j,\gamma)$, $0\leq j \le n$,
$${\rm hull}_{S(0,r)}(y_0,\ldots, y_{n})\subset S_{y,r}^+.$$
Consequently, if we let 
$$V_r = \bigcap_{\substack{0\leq j \leq n\\y_j\in X(v_j,\gamma)\cap S(0,r)}} \!\!{\rm hull}_{S(0,r)}(y_0,\ldots, y_{n}),$$
it is clear that $V_r\subset  S_{y,r}^+.$
Therefore,
\begin{equation}\label{eqfk36}
\HH^n(V_r\setminus \Omega^+) \leq \HH^n(S_{y,r}^+\setminus \Omega^+) \leq r^n\,\ve_n(y,r) \leq \delta^{1/4}r^n
\quad\mbox{ for all $r\in {\rm Good}$.}
\end{equation}
Next we let
$$V= \bigcup_{r\in [r_1,r_2]} V_r.$$
It is immediate to check that
$$V= \bar A(0,r_1,r_2)   \cap \bigcap_{\substack{0\leq j \leq n\\ y_j\in X(v_j,\gamma)}} \!\!{\rm hull}(y_0,\ldots, y_{n}),$$
and from \rf{eqfk35} it follows easily that $V$ contains some ball $B^+$ with $\rad(B^+)\approx_\gamma R$.
Further, by \rf{eqfk36} and \rf{eqfk35.5} we deduce that
\begin{align*}
\HH^{n+1}(B^+\setminus \Omega^+) & \leq \HH^{n+1}(V\setminus \Omega^+)
= \int_{r_1}^{r_2} \HH^n(V_r\setminus \Omega^+)\,dr\\
&\lesssim \int_{{\rm Good}} \HH^n(V_r\setminus \Omega^+)\,dr + R^n\,\HH^1([r_1,r_2]\setminus {\rm Good})
\lesssim \delta^{1/4}\,R^{n+1}.
\end{align*}
So $B^+$ is a $(C\delta^{1/4})$-almost corkscrew.
Notice that its radius does not depend on $\delta$, but on $\tau$ and $\gamma$, which are fixed geometric parameters.
\end{proof}
\vv
\begin{proof}[Proof of Lemma \ref{lemaux99}]
We need to introduce some notation. Given an arbitrary half open-closed interval
$L\subset \R$ and $k\geq0$, we denote by $\DD_k(L)$ the family half open-closed intervals contained in $L$ with side length $2^{-k}\ell(L)$ generated by splitting $L$ dyadically $k$ times. We also set $\DD(L) = \bigcup_{k\geq0}\DD_k(L)$.

Let $I$ and $G$ be as in the statement of the lemma. Without loss of generality, we assume that $I$ is half open-closed. Given $L\in\DD(I)$,
we say that $L$ has low density, and we write $L\in\LD$, if $|G\cap L|\leq \frac{\cfive}2 \,\ell(L)$. We also  denote by $\LD_{\max}$ the family of maximal intervals from $\LD\cap \DD(I)$. Since the intervals $L\in \LD_{\max}$ are disjoint
and they satisfy
$$|L|\leq \frac1{1-\frac{\cfive}2}\, |L\setminus G| =\frac2{2-\cfive}\,|L\setminus G|,$$
taking also into account that
$|I\setminus G|\leq (1-\cfive)\,\ell(I)$, 
we deduce
\begin{equation}\label{eqLDL}
\Big|\bigcup_{L\in\LD} L\Big| = \sum_{L \in\LD_{\max}} \ell(L) \leq \frac2{2-\cfive} \sum_{L \in\LD_{\max}} |L\setminus G|\leq \frac2{2-\cfive}\,|I\setminus G| \leq \frac{2-2\cfive}{2-\cfive}\,\ell(I) =: \csix\,\ell(I),
\end{equation}
with $\csix\in(0,1)$.

We say that an interval $J\in\DD(I)$ is good, and we write $J\in\cG$, if every subinterval $L\in\DD_k(J)$ with $k\in[0,-\log_2 \theta]$ satisfies $L\not\in\LD$. Otherwise we say that
$J$ is bad and we write $J\in\cB$.
In this language, the lemma will be proved if we find a good interval $J\in\DD(I)$ such that $\ell(J)\gtrsim c(\cfive,\theta)\ell(I)$.

We write $N_\theta = \lceil-\log_2 \theta\rceil$.
Given $J\in\cB$, we denote by $\Next(J)$ the subfamily of the intervals from $\DD_{N_\theta}(J)$ which are not contained in any interval
from $\DD(J)\cap\LD$. Notice that the condition $J\in\cB$ ensures that there exists some
$L\in\DD(J)\cap\LD$ such that $\ell(L)\geq 2^{-N_\theta}\ell(J)$, so that $\DD_{N_\theta}(J)\setminus\Next(J)\neq\varnothing$. In particular, this implies that
\begin{equation}\label{eqnext33}
\sum_{L\in\Next(J)} \ell(L)\leq (1-2^{-N_\theta})\,\ell(J).
\end{equation}
In the case $J\in\cG$, we set $\Next(J)=\varnothing$. Of course, the preceding estimate is trivially true in this case too.

Next we define inductively a family of intervals $\ttt\subset\DD(I)$. First we set $\ttt_0=I$. For $k\geq 1$, we set
$$\ttt_k = \bigcup_{J\in\ttt_{k-1}}\Next(J),$$
and then
$$\ttt=\bigcup_{k\geq0}\ttt_k.$$
Notice that $\ttt_k\subset\DD_{kN_\theta}(I)$ and by construction, if $L\in \DD_{kN_\theta}(I)\setminus \ttt_k$, then 
there exists some $J\in \DD(I)$ with $L\subset J \subset I$ such that either $J\in\LD$ or $J\in \cG\cap\ttt$. 
Observe also that, by \rf{eqnext33},
$$\sum_{L\in\ttt_k} \ell(L) \leq \big(1-2^{-N_\theta}\big)\,\sum_{L'\in\ttt_{k-1}}\ell(L')\leq \cdots \leq \big(1-2^{-N_\theta}\big)^k\,\ell(I).$$
So for $k$ large enough depending on $\csix$ and $N_\theta$, so on $\cfive$ and $\theta$, 
$$\sum_{L\in\ttt_k} \ell(L) \leq \frac{1-\csix}2\,\ell(I)$$
and thus
$$\sum_{L\in\DD_{k\,N_\theta}(I)\setminus\ttt_k} \ell(L) \geq \frac{1+\csix}2\,\ell(I)> \csix \,\ell(I).$$
From \rf{eqLDL} we infer that there exists some $L\in\DD_{k\,N_\theta}(I)\setminus\ttt_k$ which is not contained in any $J\in\LD$. Hence, by the discussion above, for such $L$ there exists $J\in \DD(I)$ with $L\subset J \subset I$ such that $J\in \cG\cap\ttt$. This is the interval
we were looking for, which satisfies the properties stated in the lemma.
\end{proof}

\vv

We also need the following lemma.

\begin{lemma}\label{lembola*}
Let $\Omega^+,\Omega^- \subset\R^{n+1}$ be disjoint Borel sets. Given $x\in \R^{n+1}$, $r>0$, and $a>0$,
let $\Delta\subset S(x,r)$ be a geodesic ball such that $a\,r\leq \rad_{S(x,r)}(\Delta)\leq \frac\pi2\,r$, so that $\Delta$ is contained in a half-sphere in  $S(x,r)$. Denote by  $H^+_{x,r}$ the half-space minimizing $\ve_n(x,r)$.
Then there exists some $\delta=\delta(a)>0$ such that:
\begin{itemize}
\item[(a)] If
$$\HH^n(\Delta\cap \Omega^+)\geq a\,\HH^n(\Delta) \quad \text{ and }\quad \ve_n(x,r)\leq\delta,$$
then $\Delta\cap H^+_{x,r}\neq\varnothing$.
\item[(b)] If
$$\HH^n(\Delta\setminus \Omega^+)\leq \delta\,\HH^n(\Delta) \quad \text{ and }\quad \ve_n(x,r)\leq\delta,$$
then $\frac34 \Delta\subset H^+_{x,r}$, where $\frac34 \Delta$ is the spherical ball in $S(x,r)$ concentric with $\Delta$	such that $\rad_{S(x,r)}(\frac34\Delta)= \frac34\rad_{S(x,r)}(\Delta)$.
\end{itemize}
\end{lemma}

\begin{proof}
Set $\sigma = \HH^n|_{S(x,r)}$ and $S^\pm= S(x,r)\cap H_{x,r}^\pm(x,r)$.
To see (a), notice that if $\Delta\cap H^+_{x,r}=\varnothing$, then $\Delta\subset S^+$, and thus
$$\sigma(S^-\setminus \Omega^-) \geq \sigma(S^-\cap \Omega^+)\geq \sigma(\Delta\cap \Omega^+)\geq a\,\sigma(\Delta)
\gtrsim_a r^n.$$
For an appropriate choice of $\delta=\delta(a)$, this contradicts the fact that $\ve_n(x,r) \leq \delta$ and proves (a).

To show (b), suppose that $\frac{3}{4}\Delta \subsetneq H^+_{x,r}$. Put $D:=H^-_{x,r} \cap \Delta$. 
Then there is a fixed numerical constant $c>0$ such that $\sigma(D)\geq c\, \sigma(\Delta)$. Put also $F := \R^{n+1} \setminus(\Omega^+ \cup \Omega^-)$. If $\delta$ is sufficiently small with respect to $c$, we have
\begin{align*}
    \sigma( D \cap (F \cup \Omega^+))& = \sigma(\Delta \cap (F \cup \Omega^+)) - \sigma((\Delta\setminus D) \cap (F\cup \Omega^+))\\
    & \geq (1-\delta) \sigma(\Delta) - \sigma(\Delta\setminus D) \\
    & \geq \left((1-\delta)-(1-c)\right) \sigma(\Delta)\gtrsim \sigma(D).
\end{align*}
 We then see that 
\begin{align*}
    \sigma(S^- \setminus \Omega^-)  = \sigma(S^- \cap (F\cup \Omega^+)) \geq \sigma(D \cap (F \cup \Omega^+)) \gtrsim \sigma(D) \gtrsim_{a} r^n.
\end{align*}
Again, for an appropriate choice of $\delta=\delta(a)$, this contradicts the hypothesis that $\ve_n(x,r) \leq \delta$.
\end{proof}
\vv

Below we will also need the following variant of Lemma \ref{lemcork1}.

\begin{lemma}\label{lemcork2}
Let $\delta,a,R>0$ and let $x,y\in\R^{n+1}$ be such that $aR\leq|x-y|\leq R$ and
$$\int_0^{8R} \big(\ve_n(x,t)^2+\ve_n(y,t)^2\big)\,\frac{dt}t \leq \delta.$$
Let $\eta>0$ and let $B_1^+$ be a ball contained in $B(x,R)$ such that 
$$2B_1^+\cap L_{x,y}\neq\varnothing , \quad\rad(B_1^+)\geq a\,R, \quad \mbox{ and }\quad \HH^{n+1}(B_1^+\cap \Omega^+) \geq \eta\,\HH^{n+1}(B_1^+),$$
where $ L_{x,y}$ is the line passing through $x$ and $y$.
For any $\beta>0$, 
if $\delta=\delta(a,\beta,\eta)>0$ is small enough, then there exists a $\beta$-almost corkscrew ball $B^+\subset 2B_1^+$  for $\Omega^+$ such that $\rad(B^+)\geq c(a,\eta)\,\rad(B_1^+)$, with $c(a,\eta)>0$.
\end{lemma}

\begin{proof}
The arguments are similar to the ones for Lemma \ref{lemcork1}, and so we will sketch the proof. 

Let $x_{B_1^+}$ be the center of $B_1^+$ and let $R_1=|x-x_{B_1^+}|$.  Since $x\not\in 2B_1^+$, we have $R_1\geq 2aR$.
Let $I_0 = [R_1-\rad(B_1^+),R_1+\rad(B_1^+)]$. For each $t\in I_0$, consider the spherical ball
$\Delta_t = B_1^+ \cap B(x,t).$
Let
$$G_0:=\{t\in I_0:\HH^n(\Delta_t\cap\Omega^+)\geq \eta^2\rad(B_1^+)^n\}.$$
To bound $\HH^1(G_0)$ from below we write
\begin{align*}
\eta\,\HH^{n+1}(B_1^+) &\leq 
\HH^{n+1}(B_1^+\cap\Omega^+) = \int_{I_0} \HH^n(\Delta_t\cap\Omega^+)\,dt \\
& \leq \eta^2\,\rad(B_1^+)^n\,\HH^1(I_0\setminus G_0) + C\,\rad(B_1^+)^n \,\HH^1(G_0)\\
&\lesssim \eta^2\,\HH^{n+1}(B_1^+) + \rad(B_1^+)^n \,\HH^1(G_0).
\end{align*}
For $\eta$ small enough, this implies that
$$\HH^1(G_0)\gtrsim \eta\,\rad(B_1^+).$$
Notice that, for $t\in G_0$, the spherical ball $\Delta_t$ satisfies
$\HH^n(\Delta_t)\geq \HH^n(\Delta_t\cap\Omega^+)\geq \eta^2\rad(B_1^+)^n$. Thus,
\begin{equation}\label{eqggtlp2}
t\gtrsim \rad_{S(x,t)}(\Delta_t) \gtrsim\eta^2\,\rad(B_1^+)\geq \eta^2a\,t.
\end{equation}

Now, by Chebyshev, arguing as in
\rf{eqal843**}, it follows easily that
$$\HH^1\big(\{t\in I_0:\ve_n(x,t)> \delta^{1/4}\}\big)  \leq C\,\delta^{1/4}R \lesssim_a \delta^{1/4}\,\rad(B_1^+).$$
So assuming $\delta$ small enough (depending on $a$ and $\eta$), we deduce that the set
$$G_1:=\{t\in G_0:\ve_n(x,t)\leq \delta^{1/4}\}$$
satisfies $\HH^1(G_1)\geq \frac12\,\HH^1(G_0)\gtrsim  \eta\,\rad(B_1^+)$.
Observe now that, by Lemma \ref{lembola*} (a), \rf{eqggtlp2}, and the fact that $\ve_n(x,r)\leq \delta^{1/4}$ for $t\in G_1$, assuming $\delta$ small enough (depending on $\eta$ and $a$) it follows that
\begin{equation}\label{intersec68}
\Delta_t\cap H^+_{x,t}\neq\varnothing\quad \text{ for all $t\in G_1$.}
\end{equation}

As in Lemma \ref{lemcork1},
 we identify each half-space $H\subset \R^{n+1}$ such that $x\in\partial H$ with the inner unit vector to $\partial H$, which we denote by $u_H$. Then, for some $\tau\in (0,1/10)$ to be fixed below, we consider a maximal
$\tau$-separated family of unit vectors $u_i$, $1\leq i \leq N_\tau$, so that $\bS^n\subset \bigcup_i B(u_i,\tau)$.
We let
$$A_i = \big\{t\in G_1:u_{H_{x,t}}\in B(u_i,\tau)\big\},$$
where $u_{H_{x,t}}$ is the unit vector identifying the subspace $H_{z,t}$ that minimizes $\ve_n(x,t)$.
So we have $G_1=\bigcup_{i=1}^{N_\tau} A_i$. 
We choose $k$ such that $A_k$ has maximal Lebesgue measure among all the $A_i$'s, so that 
$$\HH^1(A_k)\geq \frac1{N_\tau}\,\HH^1(G_1)|\gtrsim \frac \eta{N_\tau}\,\rad(B_1^+).$$

For 
some constant $\theta\in (0,1/2)$ to be fixed below, consider a subinterval $J\subset I_0$ satisfying the properties given by Lemma
\ref{lemaux99}, with $I_0,A_k$ in place of $I,G$ in the lemma, respectively. Denote by $r_J$ the center of $J$.
Consider now the cone $X(B_1^+)$ made up of all the half-lines with origin in $x$ and passing through $B_1^+$. That is, 
$X(B_1^+)$ is the smallest cone with vertex in $x$ that contains $B_1^+$. Analogously, let $X(\tfrac32 B_1^+)$ made up of all the half-lines with origin in $x$ and passing through $\tfrac32B_1^+$.
Notice that, by \rf{intersec68}, 
$$X(B_1^+)\cap H^+_{x,t} \cap S(x,t) = X(B_1^+)\cap S^+(x,t)\neq\varnothing\quad \text{ for all $t\in G_1$,}$$
and thus, in particular, for all $t\in A_k$. 
Together with the fact that
$$\dist_H(H_{u_k} \cap B(x,8R),H_{x,r}\cap B(x,8R)) \lesssim \tau\,R\quad \mbox{ for all $r\in A_k$,}
$$
 we infer that, for some constant $\tau$ small enough (depending on $a,\eta,\delta$), there is a geodesic ball $\Delta_0\subset \bS^n$ of radius comparable to $\tau R$
such that
\begin{equation}\label{eqtrans00**}
\Delta_{x,r}:= \{z\in S(x,r):r^{-1}(z-x)\in\Delta_0\}\subset 
H_{x,r}\cap X(\tfrac32 B_1^+)\quad \mbox{ for every $r\in A_k$,}
\end{equation} 
and moreover 
\begin{equation}\label{eqtrans1**}
\dist(\Delta_{x,r},L_{x,y})\geq c aR \quad \mbox{ for all $r\in I_0$},
\end{equation}
where $L_{x,y}$ is the line through $x,y$ and $c$ is an absolute constant.
This estimate is analogous to \rf{eqtrans1} in the proof of Lemma \ref{lemcork1}.
Now we continue the proof as in that lemma, allowing the different parameters there to depend on $a$ and $\eta$. 
It is easy to check that the ball $B^+$ found at the end of that proof should be contained in in $CB_1^+$. To ensure that this is contained in $2B_1^+$, we could just replace the ball $B_1^+$ in these arguments by a suitable small ball $\wt B_1^+$  contained in $B_1^+$. We leave the details for the reader.
\end{proof}
\vv


\section{Weak flatness: finding an analytic variety}\label{sec-flat2}

In the this section and the two following ones, we will carry out the proof of Lemma \ref{mainlemma1-text} (ii), which we recall (with a
slight change in the condition \rf{eq10r0} below).

\begin{lemma}\label{lem-beta<epsilon} 
	Given $c_0 \in (0,1)$, there exists $\epsilon_0>0$ so that for each $\epsilon \in(0, \epsilon_0)$ we can find $\delta=\delta(\epsilon)>0$ such that the following holds true. 
Let $\Omega^+$, $\Omega^- \subset \R^{n+1}$ be disjoint Borel sets and 
	 $B_0=B(x_0, r_0)$. Suppose also that there is a Radon measure $\mu$ with $(n, 1)$-polynomial growth supported on a subset $E \subset \overline{B_0}$ such that 
	\begin{equation*}
		\mu(B_0) \geq c_0 r_0^n.
	\end{equation*}
	If
	\begin{equation}\label{eq10r0}
		\int_0^{10 r_0} \big(\ve_n(x,r)^2 + \mathfrak{a}(x,r)^2\big) \, \frac{dr}{r} < \delta \quad\mbox{ for all } \, x \in E, 
	\end{equation}
	then $\beta_{\infty, E}(B_0) < \epsilon$. 
\end{lemma}

Notice that in the assumption \rf{eq10r0} the integral runs over $r\in (0,10r_0)$ while in Lemma \ref{mainlemma1-text}
we wrote the integral over $(0,20r_0)$. This is minor technical improvement which will be useful later to prove 
Lemma \ref{mainlemma1-text} (iii) in Section \ref{sec-splitting}.

\begin{rem}\label{rem-beta<epsilon}
	For each pair $(\beta, \epsilon)$, in order for conclusions (i) and (ii) of Lemma \ref{mainlemma1-text} to hold at the same time, it suffices to choose $\delta(\beta, \epsilon)= \min(\delta(\beta), \delta(\epsilon))$.
\end{rem}

Let us fix Borel sets $\Omega_j^+,\Omega_j^-\subset \R^{n+1}$ so that $\Omega_j^+\cap \Omega_j^-=\varnothing$ for each $j$. For the remainder of this section, we will write $\ve(x,r)= \ve_n(x,r)$. We denote by $\ve_j(x,r)$ and $\mathfrak{a}_j(x,r)$ the coefficients
$\ve(x,r)$ and $\mathfrak{a}(x,r)$  associated to the pair $\Omega_j^+,\Omega_j^-$.

The following lemma is an immediate consequence of the weak * compactness of the unit ball of $L^\infty$.

\begin{lemma} \label{l:convbasic}
Let $\{\Omega_j^\pm\}_j$ be a sequence of pairs of disjoint Borel sets in $\R^{n+1}$ as above. 
Then
there are non-negative functions $g^+,g^-\in L^\infty(\R^{n+1})$ such that, for a subsequence $\Omega_{j_k}^{\pm}$ it holds
    $$\one_{\Omega_{j_k}^\pm}\to g^\pm\quad \mbox{weakly $*$ in $L^\infty(\R^{n+1})$},$$
and    $$g^+ + g^- \leq 1.$$
\end{lemma}
\vv

Under the assumptions of the preceding lemma, 
we denote 
  \begin{align}
  	\mathfrak{a}_0(x,r)^2 & = \mathfrak{a}^+_0(x,r)^2+\mathfrak{a}^-_0(x,r)^2 \nonumber \\ 
  & \bigg|\frac{\pi^{(n+1)/2}}2 - \frac1{r^{n+1}}\int g^+(y)\,e^{-|y-x|^2/r^2}\,dy\bigg|^2 + 
\bigg|\frac{\pi^{(n+1)/2}}2 - \frac1{r^{n+1}}\int g^-(y)\, e^{-|y-x|^2/r^2}\,dy\bigg|^2.\label{e:def-mathfrak-a-0}
\end{align}
\vv


Our next objective is to prove a key geometric fact: roughly speaking, the set of points where the square function $\mathfrak{a}_0$ is zero (in the limit) must lie in an $n$-dimensional analytic variety. Moreover, that the square function $\ve$ is zero (in the limit) implies that the quasi-corkscrews found in Lemma \ref{lemcork1} are real corkscrews, albeit in a measure theoretic sense. This is made precise in the lemma below. 

\begin{lemma}\label{l:analytic-variety}
 Fix $c_0>0$. Let $\{\Omega_j^\pm\}_j$  be as above and $g^+,g^-$ as in Lemma \ref{l:convbasic}.  
    Let $B_0\subset \R^n$ be a ball and, for each $j$, let $E_j\subset B_0$ be a compact set such that
   \begin{equation}\label{e:small-alpha-compactness}  
   \int_0^{8\,\rad(B_0)} \big(\ve_j(x,r)^2 +\mathfrak{a}_j(x,r)^2\big) \, \frac{dr}{r}\leq \frac1j\quad \mbox{ for all $x\in E_j$.}
   \end{equation}
   Suppose that the sets $E_j$ converge in Hausdorff distance to a compact set $E_0$.
   Let  $\{\mu_j\}$ be a sequence of measures with $(n,1)$-polynomial growth supported on $E_j$ and converging weakly to a measure $\mu_0$, so that $\supp(\mu_0) \subset E_0$ and it has $(n,1)$-polynomial growth and suppose moreover that $\mu_0(\overline{B}_0)\geq c_0 \,\rad(B_0)^n$.
Then the following holds:
\begin{itemize}
\item[(a)] There are two balls $B^+,B^-\subset 6B_0$ such that $\rad(B^\pm)\approx \rad(B_0)$
with $$g^+=1\;\,\text{$\HH^{n+1}$-a.e.\ in $B^+$}\quad \text{ and } \quad g^-=1\;\, \text{$\HH^{n+1}$-a.e.\ in $B^-$.}$$
\item[(b)]  There is an $n$-dimensional real analytic variety $Z\subset\R^{n+1}$ such that $$\supp(\mu_0) \cap 8B_0 \subset Z \subset \R^{n+1}\setminus (B^+\cup B^-).$$
The analytic variety equals
\begin{equation}\label{eqdefZ}
Z = \bigcap_{k \geq 0} \left\{ x \in \R^{n+1}  :\, \mathfrak{a}_0(x, 2^{-k}) = 0 \right\},
\end{equation}
with $\mathfrak{a}_0(x,2^{-k})$ defined by \rf{e:def-mathfrak-a-0}.
\end{itemize}
\end{lemma}

We prove this lemma, we rescale and take $B_0=B(0,1)$. Using Lemma \ref{l:convbasic} we find a subsequence of Borel sets $\{\Omega_{j_k}^\pm\}$ whose characteristic functions $\one_{\Omega_{j_k}^\pm}$ converge weakly-$*$ to functions $g^\pm \in L^\infty(\R^{n+1})$. First we show the following.\\

\begin{claim}\label{claim:alpha-0}
   For all $x \in B(0,8) \cap \supp(\mu_0)$ and all $r \in (0, 8)$ it holds that 
    \begin{equation}\label{e:alpha-equals-0}
    \mathfrak{a}_0^+(x,r) = 0, 
    \end{equation}
    where we recall that
    \begin{equation*}
    \mathfrak{a}_0^+(x,r)= \left| \frac{\pi^{(n+1)/2}}{2} - \frac{1}{r^{n+1}} \int g^+(y) e^{-|y-x|^2/r^2} \, dy\right|.
    \end{equation*}
    and $g^+$ is the weak limit in $L^{\infty}$ of $\one_{\Omega_{j_k}^+}$ as per Lemma \ref{l:convbasic} (2). 
    
\end{claim}
    \begin{proof}[Proof of Claim \ref{claim:alpha-0}] For any $r>0$, the mapping $x \mapsto \mathfrak{a}_0^+(x,r)$ is continuous on $\mathbb{R}^{d+1}$. Hence to prove \eqref{e:alpha-equals-0} it suffices to show that
    \begin{equation}\label{e:factor-r3}
        \int_{B(0,8)} \int_{0}^8 \mathfrak{a}_0^+(x,r)^2 \, r^3 \, dr d\mu_0(x) =0.
    \end{equation}
    Recall that $r \in (0,8)$; then, using \eqref{e:small-alpha-compactness}, we conclude that 
    \begin{equation}\label{e:alphajk-factor-r3}
    \int_{B(0,8)} \int_0^8 \mathfrak{a}_{j_k}^+(x,r)^2 \, r^3 \, dr d\mu_{j_k}(x) \leq C j_k^{-1}.
    \end{equation}
\noindent
Now let $\Phi: \R^{n+1} \to \R$ be a non-negative smooth function compactly supported on $B(0,8)$, and let $\phi:\R \to \R$ be a non-negative smooth function compactly supported on $(0,8)$. For $x \in \R^{n+1}$, $r \in \R$, set 
\begin{align*}
    & h_{j_k}(x,r) := \Phi(x) \cdot \phi(r) \, r^3 \left( \frac{1}{r^{n+1}}  e^{-\left| \frac{\cdot}{r} \right|^2} * \one_{\Omega_{j_k}^+} (x) - \frac{\pi^{(n+1)/2}}{2} \right), \mbox{ and } \\
    & h_{0}(x,r) := \Phi(x) \cdot \phi(r) \, r^3 \left( \frac{1}{r^{n+1}}  e^{-\left| \frac{\cdot}{r} \right|^2} * g^+ (x) - \frac{\pi^{(n+1)/2}}{2} \right).
\end{align*}
Since $\one_{\Omega_{j_k}^+}$ converges weakly-$*$ in $L^\infty(\R^{n+1})$ to $g^+$, we see that 
\begin{equation*}
    h_{j_k}(x,r) \to h_0(x,r)\,\, \mbox{ pointwise.}
\end{equation*}
Note that $\{(x,r) \mapsto h_{j_k}(x,r)\}$ is a uniformly bounded sequence of functions on $\overline{B(0,8) \times (0,8)}$, with uniformly bounded derivatives. By the Arzel\`{a}-Ascoli Theorem, we deduce that $h_{j_k}$ converges uniformly to $h_0$ on compact subsets - up to a subequence which is immediately relabelled. After these preliminaries, we compute
\begin{align*}
\int_{B(0,8)} \int_0^8 h_0(x,r) \, dr d\mu_0(x) & = \int_{B(0,8)} \int_0^8 h_0(x,r) \, dr (\mu_0-\mu_{j_k})(x) \\
& \,\, + \int_{B(0,8)} \int_0^8 \left(h_0-h_{j_k}\right)(x,r) \, dr d\mu_{j_k}(x) \\
& \,\,+ \int_{B(0,8)}\int_0^8 h_{j_k}(x,r) \, dr d\mu_{j_k}(x).
\end{align*}
The first integral on the right hand side of the equation above tends to $0$ as $k \to \infty$ since $dr \, d\mu_{j_k}$ converges weakly to $dr \, \mu_0$ (clearly $h_0$ is continuous). 
The second integral also converges to $0$ as $k \to \infty$, this time by the uniform convergence of $h_{j_k}$ to $h_0$ on compact subsets. As for the third integral, we bound it above by the left hand side of \eqref{e:alphajk-factor-r3}. All in all, we obtain that 
\begin{equation*}
    \int_{B(0,8)} \int_0^8 h_0(x,r) \, dr \, d\mu_0(x) = 0.
\end{equation*}
Since $\Phi$ and $\phi$ are arbitrary non-negative smooth functions compactly supported on $B(0,8)$ and $(0,8)$, respectively, we conclude that \eqref{e:factor-r3} holds true. This, as previously mentioned, implies \eqref{e:alpha-equals-0} and thus Claim \ref{claim:alpha-0}.
\end{proof}

\vv

\begin{proof}[Proof of Lemma \ref{l:analytic-variety}, (b)]
We define
\begin{equation*}
    Z = \bigcap_{k \geq 0} \left\{ x \in \R^{n+1} \, |\, \mathfrak{a}_0^+(x, 2^{-k}) = 0 \right\}. 
\end{equation*}
By Claim \ref{claim:alpha-0}, if $x \in B(0,8) \cap \supp(\mu_0)$ and $r \in (0,8)$, then $\mathfrak{a}_0^+(x,r)=0$. In particular, this implies that $\supp(\mu_0) \cap B(0,8) \subset Z$.

We verify $Z$ is a real analytic variety. This means there exists a real analytic function whose zero set is $Z$. Considering $\R^{n+1} \subset \C^{n+1}$, it is enough to construct instead an entire holomorphic function whose real zero set is $Z$.
We consider the following extension of the convolution of two functions $f:\C^{n+1}\to\C$, $g:\R^{n+1}\to\R$:
\begin{equation*}
    (g*f)(x+iy):= \int_{\R^{n+1}}g(u)f(x+iy-u)du,
\end{equation*}
whenever the integral makes sense.
For $a>0$, $z=(z_1,\ldots,z_{n+1})$, with $z_k=x_k+iy_k$, we define $f_{a}(z):=a^{(n+1)/2}e^{-a\sum_{k=1}^{n+1} (x_k+iy_k)^2}$ for $a>0$ and the holomorphic extension 
\begin{equation*}
{\mathfrak{a}}_0^+(z,r)^2:= (g^+* f_{1/r^2}(z) -\pi^{(n+1)/2})^2\quad \text{for $z\in\C^{n+1}$.}
\end{equation*}
 The convolution is well defined since $g^+ \in L^{\infty}$ and $f_{a}$ has exponential decay with fixed $y$.
The desired function suitable to show that $Z$ is a real analytic variety can be taken as
\begin{equation*}
F(z):= \sum_{k \geq 0} 2^{-k}e^{-2^{2k+1} k^2}{\mathfrak{a}}_0^+(z,2^{-k})^2.
\end{equation*}
See the appendix in \cite{F} for a detailed proof that this function is entire holomorphic.

To conclude the proof of Lemma \ref{l:analytic-variety}, we need to show that $Z \neq \R^{n+1}$. In fact, we claim that for every $x \in \supp(\mu_0) \cap B(0, 8) \subset Z$, and for all $r \in (0, 8)$ such that $\mu_0(B(x,r))\geq c_0 r^n$ there exists a ball $B \subset B(x,r)$ with $\rad(B) \approx r$ such that $B \cap Z  = \varnothing$. Indeed, let $x \in \supp(\mu_0) \cap B(0,8)$. Let $\{x_j\}$ be a sequence of points in $\R^{n+1}$ so that $x_j \in \supp(\mu_j)$ and $|x_j-x| \to 0$. There exists a subsequence $j_k$ such that $\mu_{j_k}(B(x,r))> \frac{c_0}{2} \, r^n$ for $k$ large enough. For all these $k$'s, this implies that there exists a point $y_{k} \in E_{j_k}$  such that $r\geq |x_{j_k}-y_k| \gtrsim_{c_0} \, r$. Recall that, since $x_{j_k
}, y_k \in E_{j_k}$, then it holds that
\begin{align*}
& \int_0^{8|x_{j_k}-y_k|} (\ve_{j_k}(x_{j_k}, t)^2+ \ve_{j_k}(y_k, t)^2) \, \frac{dt}{t} \\
& \, \, \, \,
\leq \int_{0}^{8 \rad (B_0)} (\ve_{j_k}(x_{j_k}, t)^2 + \mathfrak{a}_{j_k}(x_{j_k},t)^2 )\, \frac{dt}{t} +  \int_{0}^{8 \rad (B_0)} (\ve_{j_k}(y_k, t)^2 + \mathfrak{a}_{j_k}(y_{k},t)^2) \, \frac{dt}{t} \leq \frac{2}{j_k}. 
\end{align*}
This is due to the hypothesis \eqref{e:small-alpha-compactness}.
In particular, for an arbitrary $\delta>0$, we can pick $k$ large enough so that the hypotheses of Lemma \ref{lemcork1} hold, and hence it may be applied with the preferred $\beta$. In fact, we first pick a sequence of numbers $\beta_i$ converging to $0$. We then find a subsequence $\{j_{k_i}\}$ such that for each $i$, the conclusion of Lemma \ref{lemcork1} holds with $\beta_i$. That is, for each $i$, there exists two balls $B^\pm_{k_i} \subset B(x_{j_{k_i}}, C(c_0) r) \subset B(x,r)$ such that each of them is an almost $\beta_i$-corkscrew ball with $\rad(B_{k_i}^\pm) \geq c_1 |x_{j_{k_i}}-y_{k_i}| \geq c_1 C(c_0) r$. 
We can assume the balls $B^\pm_{k_i}$ to be convergent in Hausdorff distance. Then we let $B^\pm$ be two balls with $\rad(B^\pm) \approx \rad(B^\pm_{k_i})$ and such that $B^\pm \subset B^\pm_{k_i}$ for all sufficiently large $k_i$. For $\kappa>0$, let $\phi_\kappa$ be a smooth cut off of $B^+$ so that $\phi=1$ on $(1-\kappa)B^+$.
Then we see that, for any $\eta>0$, there is an $i_0$ so that for all $i \geq i_0$, 
\begin{equation*}
   \eta+  \int_{B^+} g^+(y) \, dy \geq \eta + \int_{\R^{n+1}} g^+ (y) \phi_\kappa(y) \, dy \geq \int_{\R^{n+1}} \one_{\Omega_{j_{k_i}}}(y) \phi_\kappa(y) \, dy = \int_{B^+ \cap \Omega_{j_{k_i}}} \phi_\kappa(y) \, dy. 
\end{equation*}
Moreover, 
\begin{equation*}
    \int_{B^+ \cap \Omega_{j_{k_i}}} \phi_\kappa \, dy >|(1-\kappa)B^+ \cap \Omega_{j_{k_i}}|.
\end{equation*}
This is true for each $\kappa>0$ and hence 
$
    \eta + \int_{B^+} g^+(y)\, dy \geq |B^+ \cap \Omega_{j_{k_i}}| .$
Since $B^+ \subset B_{k_i}^+$ for all sufficiently large $i$, and $\rad(B^+) \approx \rad(B^+_{k_i})$, it follows that $|B^+ \cap \Omega_{j_{k_i}}| \geq (1-C\beta_i)|B^+|$, where $C$ depends on $n,c_0$ and $c_1$. In particular, it is independent of $i$. We conclude that 
\begin{equation*}
    \eta + \int_{B^+}g^+ \, dy \geq |B^+|.
\end{equation*}
Since $\eta$ was arbitrary, and $g^+ \leq 1$, this implies that $g^+ = 1$ a.e. on $B^+$. 

To conclude the proof of the claim, we want to show that $Z \cap B^+ = \varnothing$. Suppose that $z \in Z \cap B^+$. But then there exists an $m$ such that $B(z, 2^{-m}) \subset B^+$ and, because $g^+= 1$ a.e., we get that 
\begin{equation*}
   2^{(n+1)m} \int_{\R^{n+1}} g^+(y) e^{-2^{2m}|z-y|^2} \, dy \geq 2^{(n+1)m} \int_{B^+}  e^{-2^{2m}|z-y|^2} \, dy \stackrel{m\to\infty}\longrightarrow \pi^{(n+1)/2}.
\end{equation*}
In particular, $\mathfrak{a}_0^+(z, 2^{-m})^2$ is bounded away from $0$ for $m$ large enough. Thus $F(z) \neq 0$ and we run into a contradiction. This ends the proof of the claim. Then Lemma \ref{l:analytic-variety}, (2) is proven.


\end{proof}

In fact, a closer look the proof of Lemma \ref{l:analytic-variety} shows that we have proven the following lemma, which implies Lemma \ref{l:analytic-variety}, conclusion (a).

\begin{lemma}\label{l:analytic-variety1.5}
 Fix $c_0>0$. Let $\{\Omega_j^\pm\}_j$ be as above and $g^+,g^-$ as in Lemma \ref{l:convbasic}.  
    Let $B_0\subset \R^n$ be a ball and, for each $j$, let $E_j\subset B_0$ be a compact set such that
   \begin{equation}\label{e:small-alpha-compactness2}  
   \int_0^{8\,\rad(B_0)} \big(\ve_j(x,r)^2 +\mathfrak{a}_j(x,r)^2\big) \, \frac{dr}{r}\leq \frac1j\quad \mbox{ for all $x\in E_j$.}
   \end{equation}
    Suppose that the sets $E_j$ converge in Hausdorff distance to a compact set $E_0$.
   Let  $\{\mu_j\}$ be a sequence of measures with $(n,1)$-polynomial growth supported on $E_j$ and converging weakly to a measure $\mu_0$, so that $\supp(\mu_0) \subset E_0$ and it has $(n,1)$-polynomial growth and suppose moreover that $\mu_0(\overline{B}_0)\geq c_0 \,\rad(B_0)^n$. 
Then for every $x\in\supp\mu_0\cap B_0$, and radius $0<r<r_0$ such that $\mu_0(B(x,r)) \geq c_0 r^n$
there are two balls $B^+,B^-\subset  B(x,r)$ such that $\rad(B^\pm)\approx_{c_0} r$
satisfying $$g^+=1\;\,\text{$\HH^{n+1}$-a.e.\ in $B^+$}\quad \text{ and } \quad g^-=1\;\, \text{$\HH^{n+1}$-a.e.\ in $B^-$.}$$
\end{lemma}
\vv

\vv

\section{Weak flatness: the analytic variety contains a hyperplane}\label{sec-flat3}

In this section we improve upon Lemma \ref{l:analytic-variety} to conclude that the analytic variety $Z$ in fact contains an $n$-dimensional affine plane. This is yet not enough to end the proof of Lemma \ref{lem-beta<epsilon} since \textit{a priori} there could be pieces of $E_0$ lying outside this plane. We will see in the next section that this cannot be the case.

\begin{lemma}\label{l:analytic-variety2}
  Let $\{\Omega_j^\pm\}_j$ be a sequence of Borel sets satisfying the assumptions of Lemma \ref{l:analytic-variety}.
Also, let $\mu_0,B_0,Z$ be as in Lemma \ref{l:analytic-variety}. Then $Z$ contains a hyperplane $L$
 such that $\mu_0(L\cap \overline{B_0}) >0$.
\end{lemma}

Before getting into the proof of this Lemma, we recall some facts about real analytic varieties. A real analytic variety in an open
set $W\subset \R^{n+1}$ is the zero set of a function $F:W \to \R$ which is real analytic.
The zero set $Z$ of a real analytic function in a compact subset of $W$ and not vanishing everywhere can be written as
\begin{equation*}
    Z = \bigcup_{k=0}^{n} V_k, \,\, \mbox{ where } \, \, V_k = \bigcup_{\ell \geq 0} U_\ell^k,
\end{equation*}
and $U_\ell^k$ are $k$-dimensional real analytic submanifolds and the union over $\ell>0$ is finite. 
 First, each $U_\ell^k$ can be locally realised as the graph of a real analytic map defined on some $k$-dimensional affine plane. 
Second, we have that $U_\ell^k \cap U_{\ell'}^{k'} = \varnothing$ unless $k=k'$ and $\ell=\ell'$. Finally, for each $k$, the closure of $V_k$ contains all the other $V_{k'}$ for $k' \leq k$. 

Under the assumptions of Lemma \ref{l:analytic-variety}, it is easy to see that there exists an $\ell$ such that 
\begin{equation*}
    \mu_0(U_\ell^n)>0.
\end{equation*}
This follows from the fact that the Hausdorff dimension of $\bigcup_{k=0}^{n-1}V_k$ is at most $n-1$, and then the polynomial growth 
of degree $n$ of $\mu_0$ implies that $\mu_0\left(\bigcup_{k=0}^{n-1}V_k\right)=0$.

\begin{proof}[Proof of Lemma \ref{l:analytic-variety2}]
Using the notation above, let $U:=U_\ell^n$ be such that $\mu_0(U_\ell^n)>0$. We can further assume that $U$ is given by a real analytic graph over an affine $n$-plane $P_U$ (up to restricting to a local chart).
    Without loss of generality, assume that $P_U$ is the horizontal plane $\{x_{n+1}=0\}$.    Then the analytic function 
    $\Phi_U$ whose graph is $U$ is of the form
    \begin{equation}
        \Phi_U(x)= \sum_{\alpha} a_{\alpha} x^\alpha.
    \end{equation}
    Our objective is to show that $\nabla^2 \Phi_U$ vanishes in a set of positive measure $\mu_0$. This implies that $\nabla^2 \Phi_U$ vanishes identically, and then as in \cite{JTV}, $U$ is contained in a hyperplane which in turn is contained in $Z$,
    which will prove the lemma.
    
    The polynomial growth of degree $n$ of $\mu_0$ ensures that $\mu_0|_U$ is absolutely continuous with respect to $\HH^n|_U$. Denote 
    $g=\frac{d\mu_0|_U}{d\HH^n|_U}$.
      We may assume that $P_U$ is tangent to $Z$ in the origin, so that there are neither constant nor linear terms in the above Taylor expansion of $\Phi_U$, and that the origin is a Lebesgue point of $g$ (with respect to $\HH^n|_U$).         
      Now, up to an orthogonal 
    change of coordinates (by the orthogonal classification of quadrics)  there is a radius $r>0$ such that for all $x \in B(0,2r) \subset P_U$,
    \begin{equation*}
        \Phi_U(x)= \sum_{i=1}^{n} a_i x_i^2 + \Xi(x),  
    \end{equation*}
    where $|\Xi(x)|\leq C\,|x|^3$. 
    It suffices to prove that $\nabla^2 \Phi_U(0)=0$. Aiming for a contradiction, suppose this is not the case. Under this assumption, we have:
    
\begin{claim} \label{claim33}
For any $r\in(0,1)$ small enough, there exists $z_0\in \supp\mu_0\cap B(0,r)$, $r_0\in (\frac1{100}r,r)$, and four disjoint closed balls $B_1^+,B_2^+,B_1^-,B_2^-$ satisfying the following properties:
\begin{itemize}
\item[(a)] For $i=1,2$, $\rad(B_i^\pm)\approx r_0^2$, $B_i^\pm\subset G^\pm$, and $g^\pm=1$ $\HH^{n+1}$-a.e.\ in $B_i^\pm$.

\item[(b)] There exists a circumference $\Gamma_0\subset S(z_0,r_0)$ that intersects the balls
$\frac18B_1^+,\frac18B_1^-,\frac18B_2^+,\frac18B_2^-$. 

\item[(c)] There is a half-circumference $\wt \Gamma_0\subset \Gamma_0$ such that $B_i^\pm\cap \Gamma_0\subset\wt\Gamma_0$ for $i=1,2$, and moreover, if we denote by $\Lambda_0^+$ (resp.\ $\Lambda_0^-$) the shortest sub-arc from $\wt\Gamma_0$ which joins $B_1^+$ and $B_2^+$ (resp. $B_1^-$ and $B_2^-$), then either 
there exists a ball $B_i^-$ such that $B_i^-\cap \wt\Gamma_0\subset \Lambda_0^+$ or there exists a ball $B_i^+$ such that $B_i^+\cap \wt\Gamma_0\subset \Lambda_0^-$.
\end{itemize}
\end{claim}

Let us see how Lemma \ref{l:analytic-variety2} follows from the claim. Denote by $P_0$  the plane that contains 
$\Gamma_0$.
From the fact that $z_0\in \supp\mu_0$, it follows that there exists a sequence of points $z_j\in\supp\mu_j$ converging to $z_0$.
For each $j$, let $P_j$ be vertical plane parallel to $P_0$ through $z_j$. 
Let $\rho=\frac1{20}\min\{\rad(B_i^h):i=1,2,\,h=\pm\}$ and set $J_0=[r_0-\rho,r_0+\rho]$, so that $\HH^1(J_0)\approx r^2$.
From the properties of $\wt \Gamma_0$, we deduce that,
for $j$ large enough and every $t\in J_0$, there exist a circumference $\Gamma_{j,t}$ and a half-circumference $\wt\Gamma_{j,t}\subset\Gamma_{j,t}$ both centered in $z_j$ with radius $t$ such that $\Gamma_{j,t}$ intersects $\frac14B_1^+,\frac14B_1^-,\frac14B_2^+,\frac14B_2^-$ and
$\frac12B_i^\pm\cap \Gamma_{j,t}\subset\wt\Gamma_{j,t}$ for $i=1,2$.
Further, from the weak convergence in $L^\infty$ of $\one_{\Omega_j^\pm}$ to $g^\pm$, up to a subsequence, and the fact that
$g^\pm=1$ $\HH^{n+1}$-a.e.\ in $B_i^\pm$, we get 
\begin{equation}\label{eqomj88}
\HH^{n+1}(\Omega_j^+\cap B_i^+)\to \HH^{n+1}(B_i^+)\quad \text{ and }\quad\HH^{n+1}(\Omega_j^-\cap B_i^-)\to \HH^{n+1}(B_i^-)\quad
\text{as $j\to\infty$, for $i=1,2$.}
\end{equation}

For each $j\geq 1$ and a given $\gamma\in (0,1)$, let
\begin{align*}
I_j= \big\{ t\in J_0: \ve_j(z_j,t) \leq \gamma,&\;\,
\HH^n(B_i^+ \cap S(z_j,t)\setminus\Omega_j^+)\leq \gamma\,\HH^n(B_i^+ \cap S(z_j,t)),\\
&\text{ and }\, 
\HH^n(B_i^- \cap S(z_j,t)\setminus\Omega_j^-)\leq \gamma\,\HH^n(B_i^- \cap S(z_j,t))
\big\}.
\end{align*}
Notice that $\HH^n(B_i^\pm \cap S(z_j,t))\approx \rad(B_i^\pm)^n \approx r^{2n}$ for each $t\in I_j$.
By an easy application of Chebyshev's inequality, from the assumption \rf{e:small-alpha-compactness2} and \rf{eqomj88}, 
we deduce that $\HH^1(I_j)>0$ for $j$ large enough, possibly depending on $\gamma$ and $\HH^1(J_0)/r$ and so on
$r$. By Lemma \ref{lembola*}, the condition that $t\in I_j$ with $\gamma$ small enough implies that 
$$\tfrac12 B_i^\pm\cap S(z_j,t) \subset H^\pm_j(z_j,t),$$
where $H^+_j(z_j,t)$ is the half-space that minimizes $\ve_j(z_j,t)$ and $H^-_j(z_j,t)=\R^{n+1}\setminus H^+_j(z_j,t)$.

Suppose that one of the balls $B_i^-$, say $B_1^-$, satisfies $B_i^-\cap \wt\Gamma_0\subset \Lambda_0^+$.
From the fact that $(\tfrac12 B_1^+ \cup \tfrac12 B_2^+)\cap  S(z_j,t) \subset H^+_j(z_j,t)$, we deduce that, for $j$ large enough and a fixed $t\in I_j$,
\begin{equation}\label{eqhull55}
{\rm hull}_{S(z_j,t)}\big((\tfrac12 B_1^+ \cup \tfrac12 B_2^+)\cap  S(z_j,t)\big)\subset H^+_j(z_j,t).
\end{equation}
Let $\Lambda_{j,t}$ be the shortest arc contained in $\wt\Gamma_{j,t}$ which intersects the two geodesic balls $\tfrac12 B_1^+\cap  S(z_j,t)$, 
$\tfrac12 B_2^+\cap  S(z_j,t)$. This arc is strictly contained in the half-circumference $\wt\Gamma_{j,t}$
because $\frac12B_i^+\cap \Gamma_{j,t}\subset\wt\Gamma_{j,t}$ for each $i$.
 So $\Lambda_{j,t}$ is a geodesic curve from $S(z_j,t)$ which joins 
$\tfrac12 B_1^+\cap  S(z_j,t)$ and 
$\tfrac12 B_1^+\cap  S(z_j,t)$, and consequently 
$$\Lambda_{j,t}\subset {\rm hull}_{S(z_j,t)}\big((\tfrac12 B_1^+ \cup \tfrac12 B_2^+)\cap  S(z_j,t)\big)\subset H^+_j(z_j,t).$$
However, $\Lambda_{j,t}$ intersects $\tfrac14 B_1^-$, which is a contradiction because 
$\tfrac14 B_1^-\subset H_j^-(z_j,t)$.

In the case when 
one of the balls $B_i^+$ satisfies $B_i^+\cap \wt\Gamma_0\subset \Lambda_0^-$ we argue analogously, interchanging the roles of the balls $B_1^+,B_2^+$ with $B_1^-,B_2^-$.
So in any case, we get a contradiction, which proves the lemma, assuming the Claim \ref{claim33}.
\end{proof}

\vv
\begin{proof}[Proof of Claim \ref{claim33}]
Since $\nabla^2 \Phi_U(0)\neq0$, 
 there is at least an $i$ with $a_i \neq 0$. Now pick $1\leq i \leq n$ so that $|a_i| \geq |a_j|$ for $j \neq i$. Without loss of generality $i=1$ and $a_1>0$.
    
   For $\theta\in(0,1/2)$, consider the cones in $\R^n$ defined by
   $$X_\theta^l = \{x=(x_1,x') \in \R \times \R^{n-1}: |x'|<-\theta x_1\}$$
   and
   $$X_\theta^r = \{x=(x_1,x') \in \R \times \R^{n-1}: |x'|<\theta x_1\},$$
   and set $X_\theta = X_\theta^l\cup X_\theta^r$.
    There exists some $\theta>0$ and $r\in(0,1)$ such that, for $x\in X_{\theta}\cap B(0,10r)$ ,
    $$\Phi_U(x)>\frac{a_1}{2} x_1^2.$$ 
In this way, $Z\cap B(0,r)$ lies above the tangent plane $P_U$ in the cone $X_\theta$. By reducing $r$ if necessary, we also assume that
$$\partial^2_{11}\Phi_U(x)\geq \frac12\,a_{11}\geq \frac14 \,|\partial^2_{ij}\Phi_U(x)|\quad 
\mbox{ for all $x\in B(0,10r)$, $1\leq i,j\leq n$}$$
and 
\begin{equation}\label{eqdirectional}  
  \partial^2_{uu} \Phi_U(x)\geq \frac{a_1}{2}\quad
    \mbox{ for all $x\in B(0,r)$ and every unit vector $u\in X_\theta$.}
\end{equation}    

Since $0$ is a Lebesgue point of $g$, then $0$ is also a density point of $\supp\mu_0$ in $U$ with respect to $\HH^n|_U$.
We take $\tau\in(0,10^{-3})$ small enough (to be chosen below) and then we assume $r>0$ small enough so that
\begin{equation}\label{eqtau**1}
\HH^n(\supp\mu_0\cap B(0,t))\geq (1-\tau)\,\HH^n(U\cap B(0,t)) \quad \mbox{ for all $t\leq r$.}
\end{equation}
 Next we consider three $n$-dimensional cubes $R_l$ ,$R_c$, $R_r$ contained in $P_U=\{x_{n+1}=0\}\cap B(0,r)$ satisfying
\begin{itemize}
\item $R_l$ is centered at $(-\frac r2,0,\ldots,0)$, $R_c$ is centered at the origin, and $R_r$ is centered at $(\frac r2,0,\ldots,0)$.
\item  We have $\ell(R_i)=\kappa_0r$ for $i=l,c,r$, where $\kappa_0$ is such that $\kappa_0\in(0,\frac1{10n^{1/2}}) $, $\kappa_0\approx1$, and small enough so that $5R_l\subset X_\theta^l$ and $5R_r\subset X_\theta^r$ and so that
any vector of the form $y-x$ with $y\in R_l\cup R_r$ and $x\in R_c$ belongs to $X_\theta$.
\end{itemize}
Notice that the parameter $\kappa_0$ above depends on the aperture of the cone $X_\theta$ and so on $\nabla^2\Phi_U$, but not on $r$.

Next we take $N = (\conedouble\ell(R_l))^{-1}$, with $\conedouble \in (0,1)$ to be chosen below (depending on $\nabla^2\Phi_U(0)$ among other parameters, but not on $r$) and we split $R_l$ into $N^n$ $n$-dimensional cubes with disjoint interiors and equal  side length. We denote by $Q_l^1,\ldots Q_l^N$
this family of $n$-dimensional cubes. Notice that  their side length is comparable to $r^2$ and that $N\approx r^{-1}$. By standard arguments, we find a subfamily $\{Q_l^k\}_{k\in K_l}$ of $\{Q_l^1,\ldots, Q_l^N\}$ such that the $n$-dimensional cubes  $\{10Q_l^k\}_{k\in K_l}$ are pairwise disjoint and moreover
$$\sum_{k\in K_l} \HH^n(g(Q_{l}^k))\gtrsim \HH^n(g(R_l)) \gtrsim \HH^n(U\cap B(0,r))$$
(here and in the estimates below we allow the implicit constants to depend on the parameter $\kappa_0$ above).
Provided $\tau$ is small enough, then it easily follows that there is  a subfamily $\{Q_l^k\}_{k\in H_l} \subset \{Q_l^k\}_{k\in K_l}$ of the $n$-cubes 
$Q_l^k$
such that 
\begin{equation}\label{j1kbigf}
\HH^n(\supp\mu_0\cap g(Q_l^k))\approx \HH^n(g(Q_l^k)),\end{equation}
and
\begin{equation}\label{eqsjh3}
\sum_{k\in H_l} \HH^n(\supp\mu_0\cap g(Q_{l}^k))\gtrsim \HH^n(U \cap B(0, r)).
\end{equation}
See Lemma 8.6 from \cite{JTV} for more details.

Next, note that the condition \rf{j1kbigf} ensures that we can apply Lemma \ref{l:analytic-variety1.5} for each $k\in H_l$ to find a ball $B^+_k$ satisfying
$$8B^+_k\subset \NN_{\ell(Q_{l}^k)}(g(Q_l^k))\quad
\text{ with
$g^+=1$ $\HH^{n+1}$-a.e.\ in $8B_k^+$ and
$r(B_k^+)\approx \ell(Q_{l}^k)$},$$
where $\NN_\ell(A)$ stands for the $\ell$-neighborhood of $A$.
Observe that, because of the fact that  $\Phi_U(x)>\frac{a_1}{2} x_1^2$ for all $x\in X_{\theta}\cap B(0,10r)$, we have $\dist(g(Q_l^k),P_U)\approx\ell(R_l)^2\approx r^2$. On the other hand,
$$\dist(B_k^+,g(Q_l^k))\leq \ell(Q_l^k)\leq \frac{\ell(R_l)}N =\conedouble \ell(R_l)^2.$$
So if we choose $\conedouble $ small enough, then the balls $8B_k^+$ are contained in $\R^{n+1}_+$ and far from $P_U$.

Next, let $\Delta_k$, $k\in H_l$, be the projection of the balls $B_k^+$, $k\in H_l$, on the hyperplane $P_U$. The $n$-dimensional balls $\Delta_k$, $k\in H_{l}$, are disjoint, and moreover, $$\sum_{k\in H_{l}}\mathcal{H}^n(g(\tfrac{1}{10}\Delta_j))\gtrsim \sum_{k\in H_{l}}\mathcal{H}^n(g(\Delta_j))\stackrel{(\ref{eqsjh3})}{\gtrsim} r^n.$$
Therefore, by standard arguments (once again, see Lemma 8.6 from \cite{JTV} for more details), we find a family of indices $M_l\subset H_l$ such that
\begin{equation}\label{eq:onetenthbigF**}
\HH^n(\supp\mu_0\cap g(\tfrac1{10} \Delta_k))\approx \HH^n(g(\tfrac{1}{10}\Delta_k))\approx \rad(\Delta_k)^n\quad \text{ for every }k\in M_l\end{equation} 
and
$$\sum_{k\in M_l} \HH^n(\supp\mu_0\cap g(\tfrac1{10} \Delta_k))\approx r^n.$$

Since \rf{eq:onetenthbigF**} holds, we may apply Lemma \ref{l:analytic-variety1.5}
for each $k\in M_l$ to find
 some ball $B^-_k$ satisfying
$$8B^-_k\subset \NN_{\tfrac1{10} \rad(\Delta_k)}(g(\tfrac1{10} \Delta_k))\quad
\text{ with
$g^-=1$ $\HH^{n+1}$-a.e.\ in $8B_k^-$ and
$\rad(B_k^-)\approx \rad(\Delta_k)$.}$$
Again, since $\Phi_U(x)>\frac{a_1}{2} x_1^2$ for all $x\in X_{\theta}\cap B(0,10r)$, the balls $8B^-_k$ are contained in $\R^{n+1}_+$ and far away from
$P_U$. Further, if $\Pi_{P_U}$ stands for the orthogonal projection on $P_U$, 
by construction the $n$-dimensional ball $\Pi_{P_U}(B^-_k)$ is contained deep inside $\Pi_{P_U}(B^+_k)$ for
each $k\in M_l$. In fact, by shrinking the balls $B^-_k$ if necessary, we can assume that
$$\Pi_{P_U}(B^-_k)\subset \Pi_{P_U}(\tfrac12 B^+_k)\quad \mbox{ for each $k\in M_l$.}$$

Now we denote
$$W_l = \bigcup_{k\in M_l} \Pi_{P_U}(\tfrac12B^-_k).$$
By the disjointness of the $n$-dimensional cubes $10Q^k_l$, $k\in M_l$, the $n$-dimensional balls  $\Pi_{P_U}(\tfrac12B^-_k)$
are disjoint and we deduce that $\HH^n(W_l)\approx r^n.$

Next we define an analogous family of balls $\{B^\pm_k\}_{k\in M_r}$ and a set $W_r$, replacing the left
$n$-dimensional cube $R_l$ by the right one $R_r$.

We claim that there is some $x\in R_c\cap \Pi_{P_U}(\supp\mu_0)$ such that
$$W_l \cap (2x-W_r)\neq \varnothing.$$
In fact, for an arbitrary point $y_r\in W_r$,
the set $\{2x - y_r:x\in R_c\cap \Pi_{P_U}(\supp\mu_0)\}$ is of the form $R\setminus X$, where $R\subset P_U$ is an $n$-dimensional cube
 of side length $2\ell(R_c)$ centered at $-y_r\in R_l$  
 (so it contains $R_l$) and $X$ is an exceptional set with $\HH^n$-measure  at most
 $2^n\HH^n(R\setminus \Pi_{P_U}(\supp\mu_0))\leq c\HH^n(g(R)\setminus \supp\mu_0)\leq
 c\tau\ell(R)^n$. So for $\tau$ small enough\footnote{Recall that $\tau$ was introduced in \rf{eqtau**1} and notice that $\tau$ may depend on the parameter $\kappa_0$ above, and so on $\theta$ and $\nabla^2\Phi_U(0)$. However, $\tau$ does not depend on $r$.} , $\{2x - y_r:x\in R_c\cap \Pi_{P_U}(\supp\mu_0)\}$
 intersects $W_l$, since $\HH^n(W_l)\approx \ell(R)^n\gg c\tau\ell(R)^n$.

The preceding argument shows that
 there exist $y_l\in W_l$, $y_r\in W_r$, and $x_0\in R_c\cap \Pi_{P_U}(\supp\mu_0)$ such that $y_l = 2x_0-y_r$, or equivalently,
$$x_0= \frac{y_l+y_r}2.$$
Observe that, in particular, this implies that $|x_0-y_l|= |x_0-y_r|\approx r$.

Let $k\in M_l$ be such that $y_l\in\Pi_{P_U}(\tfrac12B^-_k)$ and $h\in M_r$  such that
$y_r\in\Pi_{P_U}(\tfrac12B^-_{h})$. By construction, there are points $y_l^\pm\in \frac12B_k^\pm$ and
$y_r^\pm\in \frac12B_h^\pm$ such that
$$\Pi_{P_U}(y_l^-)= \Pi_{P_U}(y_l^+) = y_l\quad \text{and}\quad \Pi_{P_U}(y_r^-)= \Pi_{P_U}(y_r^+) = y_r.$$
We claim that the circumference centered at $g(x_0)$ (observe that $g(x_0)\in  U\cap\supp\mu_0$) with radius $|x_0-y_l|$ intersects the four balls $B_k^\pm$ and $B_h^\pm$. To see this, notice that
$$\big||g(x_0)- y_l^\pm| - |x_0-y_l|\big| \leq (1-\cos\alpha^\pm)\,r
\lesssim (\alpha^\pm)^2\,r,$$
where $\alpha^\pm$ is the slope of the line passing through $g(x)$ and $y_l^\pm$, which satisfies $\alpha^\pm\lesssim r$ (taking into account that $|x_0-y_l|= |x_0-y_r|\approx r$ and the quadratic behavior of $g$ close to the origin). Thus,
$$\big||g(x_0)- y_l^\pm| - |x_0-y_l|\big| \lesssim r^3\ll r^2\approx r(B_k^\pm)
.$$
Analogously,
$$
\big||g(x_0)- y_r^\pm| - |x_0-y_r|\big| \lesssim r^3\ll r^2\approx r(B_h^\pm).$$
Thus,  since $|x_0-y_r| = |x_0-y_l|$,
 the aforementioned circumference passes through the balls $B_k^\pm$, $B_h^\pm$.

Let $z_0=g(x_0)$ and $r_0=|x_0-y_l|$ and let $P_0$ be the vertical plane containing $x_0,z_0,y_l,y_r$. Let $\Gamma_0$ be the circumference contained in $P_0$, centered in $z_0$, and with radius $r_0$.
Using the condition \rf{eqdirectional}, 
it is easy to check that there is an arc $\wt\Gamma_0\subset\Gamma_0$ satisfying the
required properties in the claim, with the balls $B_1^\pm$, $B_2^\pm$ in the claim replaced by $8B_k^\pm$, $8B_h^\pm$. To see this, let $H_{z_0}$ the open half-space whose boundary equals the tangent to $U$ at $z_0$ and containing $U\cap B(0,r)\setminus \{z_0\}$. It is easy to see that the four balls
$8B_k^\pm$, $8B_h^\pm$ are contained in $H_{z_0}$, taking into account \rf{eqdirectional} with $u$ being a unit vector with the same direction as $y_r-x_0$
 and choosing the constant $\conedouble $ above small enough if necessary.  
\end{proof}

\vv

\vv


\section{Weak flatness: the set $E_0$ is flat}\label{sec-flat4}

In this section we conclude the proof of Lemma \ref{lem-beta<epsilon}, and thus that of Main Lemma \ref{mainlemma1}(ii), by showing that the \textit{whole of} the limit set $E_0$ from Lemma \ref{l:analytic-variety} is contained in an $n$-dimensional plane.  

\begin{lemma}\label{lemflat1}
 Let $\{\Omega_j^\pm\}_j$, $\{E_j\}_j$, $Z$, and $L$ be as in Lemmas \ref{l:convbasic}, \ref{l:analytic-variety}, 
 and \ref{l:analytic-variety2}. Suppose that the sets $E_j$ converge in Hausdorff distance to a compact set $E_0$.
 Then $E_0\subset L$.
\end{lemma}

\begin{proof}
Let $L$ be a hyperplane such that $\mu_0(L\cap \overline{B_0})>0$, whose existence is ensured by Lemma \ref{l:analytic-variety2}.
Without loss of generality we assume that $L=\R^n\times\{0\}$. Aiming for a contradiction, suppose that $x_0\in E_0\setminus L$, with $x_0$ belonging to the upper half space $\R^{n+1}_+$, say.
Notice that $\supp\mu_0\subset E_0\subset \overline{B_0}$. 
Let $\Pi_L$ denote the orthogonal projection on $L$ and let $y_0\in E_0 \cap L$ be a Lebesgue point for $E_0$ (with respect to $\HH^n|_L$) such that $y_0\neq \Pi_L(x_0)$. The existence of $y_0$ follows from the fact that $\mu_0(L\cap \overline{B_0})>0$.

Denote by $\theta$ the angle between the line through $x_0,y_0$ and the vertical line through $x_0$. Because of the choice of $y_0$, we have $\theta>0$. Let $\delta_0,\delta$ be small positive parameters to be chosen below, with $0<\delta\leq\delta_0\ll \sin\theta$. Since $y_0$ is a  Lebesgue point for $E_0$ with respect to $\HH^n|_L$, we can pick $\rho\in \big(0,\frac1{10}|x_0-y_0|\,\sin\theta\big)$ such that $E_0$ is $(\delta\rho)$-dense in $B(y_0,\rho)\cap L$. 
 By Lemma \ref{l:analytic-variety1.5}, there is a ball $B^+\subset B(y_0,\delta_0\rho)$ such that $g^+=1$ $\HH^{n+1}$-a.e.\
in $B^+$ with $\rad(B^+)\approx \delta_0 \rho$. By reducing and moving $B^+$ a little if necessary, we may assume that $2B^+\cap L=\varnothing$.

Let $P$ be the vertical plane (i.e., orthogonal to $\R^n\times\{0\}$) containing $x_0$ and the center of $B^+$, which we denote by $z_0^+$. Consider the circumference $\Gamma$ contained in $P$, centered in $x_0$, with
radius $|x_0-z_0^+|$, so that $z_0^+\in \Gamma$. Notice that $\Gamma$ intersects $L\cap P$ in two points $z_1,z_2$, and the closest one to $z_0^+$, which we assume to be $z_1$, is contained in $B(y_0,\rho)$. See Figure 2. Hence there exists some point $z_1'\in E_0\cap L$ such that $|z_1-z_1'|\leq \delta\rho$.
 By Lemma \ref{l:analytic-variety1.5},  there is a ball $B^-\subset B(z_1',\delta\rho)$ such that $g^-=1$ $\HH^{n+1}$-a.e.\ with $\rad(B^-)\approx\delta\rho$. 

Denote by $z_0^-$ the center of $B^-$. Notice that $|z_0^- - z_1|\lesssim \delta\rho$, so that
\begin{equation}\label{eqdist64}
\big| |x_0-z_0^+| - |x_0-z_0^-|\big|  = \big| |x_0-z_1| - |x_0-z_0^-|\big| \leq |z_0^- - z_1|\lesssim \delta\rho.
\end{equation}
Let $P_0$ be the vertical plane containing $x_0$ and $z_0^-$. Consider the circumference $\Gamma_0$ contained in $P_0$, centered in $x_0$, with
radius $r_0=|x-z_0^-|$, so that $z_0^-\in \Gamma_0$. It is immediate to check that $\dist_H(\Gamma,\Gamma_0)\to0$ as $\delta\to0$.
This follows from \rf{eqdist64} and the fact that 
$$\dist(z_0^-,P) \leq |z_0^- - z_1|\lesssim \delta\rho,$$
since $z_1\in P$. Then we choose $\delta$ small enough so that
$$\dist_H(\Gamma,\Gamma_0)\leq \frac18\,\rad(B^+)$$
Recall that $\rad(B^+)\approx\delta_0\rho$, and so $\delta$ depends on $\delta_0$,
among other parameters. In this way $\Gamma_0$ intersects both $\frac18 B^+$ and $\frac18 B^-$,
and it intersects $L\cap P_0$ in points $w_1$, $w_2$ such that 
$$|z_i-w_i|\lesssim   \frac18\,\rad(B^+)\approx\delta_0\rho.$$

Let $\wt\Gamma_0$ the shortest arc of $\Gamma_0$ that contains $w_2$ and $z_0^-$ and intersects the closure of $\frac18 B^+$. For $\delta_0$ and $\delta$ small
enough, we have 
\begin{equation}\label{eqaks84}
\HH^1(\wt\Gamma_0) <\frac12 \HH^1(\Gamma_0).
\end{equation}
 This is due to the fact that, if we let $\wt\Gamma_0'= \Gamma_0\cap \R^{n+1}_-$ (where
$\R^{n+1}_-$ is the lower half-space), it holds that $\HH^1(\wt\Gamma_0') <\frac12 \HH^1(\Gamma_0)$, since $x_0$ belongs to the upper half space. Since $B^+$ and $B^-$ are in a $\delta_0\rho$-neigborhood of the endpoints $w_1$, $w_2$ of $\wt\Gamma_0'$, \rf{eqaks84} follows if $\delta_0$ is taken small enough.

Observe that one of the end-points of $\wt\Gamma_0$ is $w_2$, and the other which we denote by $w_0$ either coincides with $z_0^-$ or it belongs to $\frac18 B^+$. From now on, suppose that $w_0=z_0^-$.

\begin{figure}\label{figure2**}
	\centering
	\includegraphics[scale=1.2]{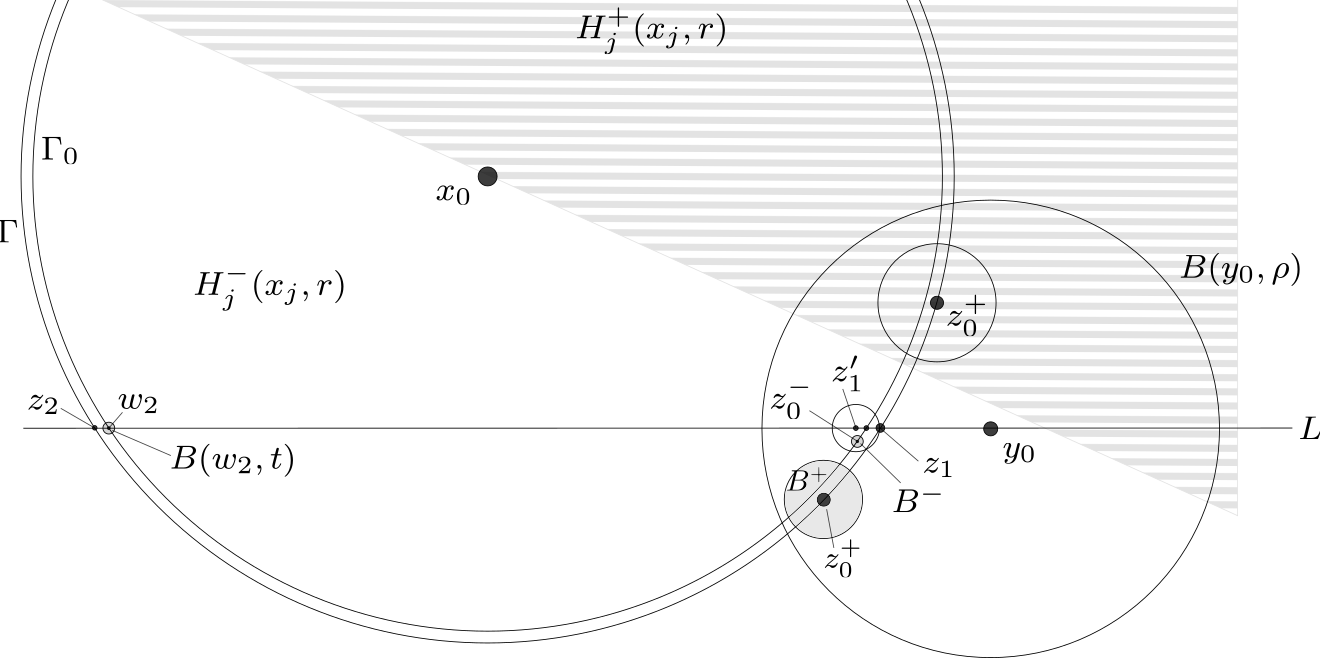}
	\caption{The setup of the proof of Lemma \ref{lemflat1}, in the case where $w_0=z_0^-$.   The figure also shows the case where $w_0 \in 1/8B^+$ (see the ball centered in the other point labelled $z_0^+$ in the shaded area).}
\end{figure}

\begin{claim*}
There exists some $t>0$ such that $g^+=1$ $\HH^{n+1}$-a.e.\ in $B(w_2,t)$.
\end{claim*}

Assume this for the moment. It is immediate to check that this implies that $\mathfrak a_0(w_2,r)>0$ for any $r>0$ small enough. By the
definition of $Z$ in \rf{eqdefZ}, this cannot happen because $w_2\in L \subset Z$. So we get the desired contradiction that proves the lemma.
\vv

Now we turn to the proof of the claim. For contradiction, for some small $\tau>0$ and some $t_0\in(0, \frac18\,\rad(B^-))$, both to be fixed below, suppose that
$$\int_{B(w_2,t_0)}g^+\,d\HH^{n+1} \leq (1-2\tau)\,\HH^{n+1}(B(w_2,t_0)).$$
Since $x_0\in E_0$, there exist a sequence of points $x_j\in\supp\mu_j\subset E_j$ converging to $x_0$.
From the weak convergence in $L^\infty$ of $\one_{\Omega_j^\pm}$ to $g^\pm$, up to a subsequence, and the fact that
$g^\pm=1$ $\HH^{n+1}$-a.e.\ in $B^\pm$, we get that for $j$ large enough,
\begin{equation}\label{eqomeg72}
	\HH^{n+1}(\Omega_j^+\cap B(w_2,t_0))\leq (1-\tau)\,\HH^{n+1}(B(w_2,t_0))
.
\end{equation}
Also, for any given $\gamma'>0$,
\begin{equation}\label{eqomeg73}
\HH^{n+1}(B^\pm\setminus \Omega_j^\pm)\leq \gamma'\,\HH^{n+1}(B^\pm)
\end{equation}
for $j$ large enough too.

For each $j\geq 1$, consider the interval of radii 
$$J_j= \big[|x_j-w_2|-t_0,|x_j-w_2|+t_0|\big]$$
Next, for a given $\gamma\in (0,1)$, let
$I_j$ the subset of those $t\in J_j$ such that 
$\ve_j(x_j,t) \leq \gamma$, $\HH^n(B(w_2,t_0)\cap S(x_j,t)\setminus\Omega_j^+)\geq \csixdouble\tau\,t_0^n$, and $\HH^n(B^h\cap S(x_j,t)\setminus\Omega_j^h)\leq \gamma\,\HH^n(B^h \cap S(x_j,t))$ for $h=\pm$.
It is not difficult to check that $I_j\neq\varnothing$ for $\csixdouble>0$ being a suitable absolute constant. Indeed, consider first the set
$$I_j^0 = \big\{t\in J_j: \HH^n(B(w_2,t_0)\cap S(x_j,t)\setminus\Omega_j^+)\geq \csixdouble\tau\,t_0^n\big\}.$$
By \rf{eqomeg72} we have 
\begin{align*}
\tau\,\HH^{n+1}(B(w_2,t_0)) & \leq 	\HH^{n+1}(B(w_2,t_0)\setminus \Omega_j^+) \\
& = 
\int_{J_j\setminus I_j^0} 	\HH^{n}(B(w_2,t_0) \cap S(x_j,t)\setminus \Omega_j^+)\,dt +
\int_{I_j^0} 	\HH^{n}(B(w_2,t) \cap S(x_j,t)\setminus \Omega_j^+)\,dt\\
& \leq \int_{J_j} \csixdouble\tau\,t_0^n\,dt + \HH^1(I_j^0)\,\sup_{t\in I_j^0}\HH^{n}(B(w_2,t_0) \cap S(x_j,t))\\
& \leq 2\csixdouble\tau\,t_0^{n+1} + C\,t_0^n\,\HH^1(I_j^0).
\end{align*}
Since $\HH^{n+1}(B(w_2,t_0)) = C_{n} t_0^{n+1}$ for a suitable constant $C_{n}>0$, rearranging terms, for $\csixdouble= C_{n}/4$ we get
$$\HH^1(I_j^0) \gtrsim \frac{C_{n} \tau\,t_0^{n+1} - 2\csixdouble\tau\,t_0^{n+1}}{t_0^n}\approx \tau \,t_0.$$
Next notice that
$$I_j = \big\{t\in I_j^0: 
\ve_j(x_j,t) \leq \gamma \text{ and } \HH^n(B^h\cap S(x_j,t)\setminus\Omega_j^h)\leq \gamma\,\HH^n(B^h \cap S(x_j,t))\text{ for $h=\pm$}\big\}.$$
Taking into account that $\HH^1(I_j^0) \gtrsim  \tau \,t_0$, together with  \rf{e:small-alpha-compactness2} and \rf{eqomeg73} (with $\gamma'$ small enough), by an easy application of the Chebyshev inequality we deduce
that $\HH^1(I_j)>0$ for $j$ large enough, and thus $I_j\neq\varnothing$, as wished.

For $j$ large enough and $r\in I_j$, from the fact that
$\ve_j(x_j,r) \leq \gamma$ and $\HH^n(B(w_2,t_0)\cap S(x_j,r)\setminus\Omega_j^+)\geq \csixdouble\tau\,t_0^n$,
we deduce that for $\gamma$ small depending on $\tau$ and $t_0$ (and thus for $j$ large enough),
\begin{equation}\label{eqdk821}
B(w_2,t_0) \cap S(x_j,r)\cap H_j^-(x_j,r)\neq \varnothing.
\end{equation}
Also, 
again by the fact that $\ve_j(x_j,r) \leq \gamma$ together the condition that $\HH^n(B^\pm\cap S(x_j,t)\setminus\Omega_j^\pm)\leq \gamma\,\HH^n(B^\pm \cap S(x_j,t))$, applying Lemma \ref{lembola*},
we deduce
\begin{equation}\label{eqdk822}
\tfrac12 B^\pm \cap S(x_j,r)\subset H_j^\pm(x_j,r).
\end{equation}

By \rf{eqdk821} and \rf{eqdk822} we can take $w_3\in B(w_2,t_0) \cap S(x_j,r)\cap H_j^-(x_j,r)$ and $z_3\in\tfrac12 B^- \cap S(x_j,r)\subset H_j^-(x_j,r)$. Further, we can choose $z_3$ such that 
$$|z_3-z_0^-|\leq |x_0-x_j| + t_0,$$
because $\big|r-|x_0-z_0^-|\big|\leq t_0.$
Then the geodesic $\wt \Gamma_j$ in $S(x_j,r)$
joining $z_3$ and $w_3$ intersects $B^+$.
Indeed, let $P_j$ be the plane passing through $x_j,z_3,w_3$. Since $x_j\to x_0$
 as $j\to\infty$, $w_3\to w_2$ as $t_0\to0$, and
$z_3\to z_0^-$ as $j\to\infty$ and $t_0\to0$, it follows that
$\wt \Gamma_j = P_j\cap S(x_j,r)$ tends in Hausdorff distance to $\wt \Gamma_0 = P_0\cap S(x_0,r_0)$ (recall that 
$r_0=|x-z_0^-|$ and that
$\wt \Gamma_0$
is a geodesic in $S(x_0,r_0)$ because $\HH^1(\wt\Gamma_0)<\frac12 \HH^1(\Gamma_0)$ and $\Gamma_0$ is a great circle contained in $S(x_0,r_0)$). Remark that we can choose $t_0$ indepently of $j$, so that then we can take $j$ large enough depending on $t_0$.
The fact that both $w_3$ and $z_3$ belong to $H_j^-(x_j,r)$ imply that $\wt \Gamma_j\subset H_j^-(x_j,r)$ because of the geodesic
convexity of $H_j^-(x_j,r)$ in $S(x_j,r)$. However,  $\wt \Gamma_j$ intersects $\frac12 B^+$ for $j$ large enough,
taking into account that $\wt \Gamma_0$ intersects $\frac18 B^+$ and the convergence of $\wt \Gamma_j$ to $\wt \Gamma_0$ in Hausdorff distance.
Since $\frac12 B^+\subset H^+(x_j,r)$, this is not possible. This contradiction proves the claim, and concludes the proof of the lemma in the case when the end-point $w_0$ of $\wt \Gamma_0$ coincides with $z_0^-$. 

\vv
The arguments for the case when $w_0$ is a point belonging to $\frac18 B^+$ are very similar. In this case, it turns out that
there exists some $t>0$ such that $g^-=1$ $\HH^{n+1}$-a.e.\ in $B(w_2,t)$, which cannot hold because $w_2\in L \subset Z$, arguing as above.
The existence of the ball $B(w_2,t)$ is proven by arguments very similar to the ones of the claim above, interchanging the role of $B^+,\Omega_j^+,g^+$ by the one of $B^-,\Omega_j^-,g^-$. We leave the details for the reader.
\end{proof}

\vv

\begin{proof}[Proof of Lemma \ref{lem-beta<epsilon} and  Lemma \ref{mainlemma1-text} (ii)]
The arguments are standard. For completeness, we will show the details. 

By renormalizing it suffices to prove the lemma for the ball $B_0:=B(0,1)$. We argue by contradiction: we suppose that there exists an $\epsilon>0$ such that, for all $j \in \N$, there exist disjoint Borel sets
$\Omega_j^+$, $\Omega_j^- \subset \R^{n+1}$ and a measure $\mu_j$ with $(n, 1)$-polynomial growth supported on $E_j \subset B_0$ such that 
	\begin{equation*}
		\mu_j(B_0) \geq c_0 r_0^n.
	\end{equation*}
	and  
	\begin{equation*}
		\int_0^{10 r_0} \big(\ve_j(x,r)^2 + \mathfrak{a}_j(x,r)^2\big) \, \frac{dr}{r} < \frac1j \quad\mbox{ for all } \, x \in E_j, 
	\end{equation*}
	while $\beta_{\infty, E_j}(B_0) > \epsilon$. 

We can assume that the sets $E_j$ converge in Hausdorff distance to a compact set $E_0$ and that the measures $\mu_j$ converge weakly to another measure $\mu_0\subset E_0$, so that we can apply Lemmas
\ref{l:analytic-variety}, \ref{l:analytic-variety2}, \ref{lemflat1} to deduce that $E_0$ is contained in a hyperplane.
That is, the sequence of sets $E_j$ converge in Hausdorff distance to a subset of a hyperplane, which contradicts the fact that $\beta_{\infty, E_j}(B_0) > \epsilon$ for every $j$.
\end{proof}

\vv


\section{Propagation of smallness of the smooth square function}

In this section we prove an auxiliary and quite easy fact: $\mathfrak{a}(x,r)$ is small everywhere near the plane which well approximates the set $E$.

\begin{lemma}\label{lem-corkscrew-everywhere}
    Let $\Omega^+$, $\Omega^-$,  $c_0$, $B_0$, $\mu$ and $E$ as in Lemma \ref{mainlemma1-text}. For all $\lambda \in (0, \tfrac{1}{2})$ there exist an $\epsilon= \epsilon(\lambda)>0$ and a $\delta=\delta(\epsilon,\lambda)$ such that the following holds. If
    \begin{equation}\label{e:ce-1}
        \int_0^{10r_0} \big(\ve_n(x,r)^2+ \mathfrak{a}(x,r)^2\big) \, \frac{dr}{r} < \delta\quad \mbox{ for all } x \in E,
    \end{equation}
    then for all points $x \in \overline{\NN_{2\epsilon r_0}(L_{B_0})\cap B_0}$, where $L_{B_0}$ is the plane minimising $\beta_{\infty, E}(B_0)$, and for all radii $r \in [10 \epsilon r_0, r_0]$,
    we have that
    \begin{equation}\label{e:mathfrak-a<lambda}
    \mathfrak{a}^\pm(x,r) < \lambda.
    \end{equation}
    
\end{lemma}

\begin{rem}	
	Before starting the proof, we pick $\lambda>0$, and then $\epsilon=\epsilon(\lambda)$ sufficiently small with respect to $\lambda$, and hence, thanks to Lemma \ref{lem-beta<epsilon}, we find a $\delta=\delta(\epsilon)$ so that $\beta_{\infty,E}(B_0)<\epsilon$ whenever \eqref{e:ce-1} is satisfied with this $\delta(\epsilon)$. Our proof will be by contradiction; thus it will output \textit{another} $\delta=\delta(\epsilon, \lambda)$ for which \eqref{e:mathfrak-a<lambda} holds as long as \eqref{e:ce-1} holds with \textit{this} $\delta=\delta(\epsilon, \lambda)$ and moreover $\epsilon \ll \lambda$.
\end{rem}

\begin{proof}
    Choose $\epsilon$ sufficiently small. Then by Lemma \ref{mainlemma1-text} (ii), we know that there is a $\delta(\epsilon)>0$ such that if \eqref{e:ce-1} holds with this $\delta$, then $\beta_{\infty, E}(B_0)<\epsilon$. Now suppose that the conclusion of Lemma \ref{lem-corkscrew-everywhere} is false. Then we find a $\lambda>0$ and a sequence of Borel sets $\Omega_j^\pm$ and measures $\mu_j$ with the properties as stated in the hypotheses of Lemma \ref{l:analytic-variety} (so in particular \eqref{e:small-alpha-compactness}), such that there exists a point $z_j \in \overline{\NN_{2 \epsilon r_0}(L_{j,B_0}) \cap B_0}$ (where $L_{j,B_0}$ is the $n$-plane minimising $\beta_{\infty, E_j}(B_0)$) and a radius $r_j \in [10\epsilon r_0, r_0]$ such that 
    \begin{equation*}
        \mbox{either} \quad\, \mathfrak{a}^+_j(z_j,r_j) \geq \lambda,\,\, \mbox{ or } \,\, \mathfrak{a}^-_j(z_j,r_j)  \geq \lambda.
    \end{equation*}
    Recall that $E_j \to E_0$ in the Hausdorff distance. Recall moreover that we proved in Lemma \ref{lemflat1} that $E_0 \subset L_0 \subset Z$, where $L_0$ is an $n$-plane and $Z$ is an $n$-dimensional analytic variety where $\mathfrak{a}_0(x,r) = 0$ for all $x \in Z$ and $r>0$ (to recall the definition of $\mathfrak{a}_0$, see \eqref{e:def-mathfrak-a-0}). Choose a subsequence, immediately relabelled, such that
    \begin{itemize}
        \item[(i)] $L_{j, B_0} \to L_0'$ for some $n$-plane $L_0'$;
        \item[(ii)] $r_j \to \mathfrak{r}$ for some $\mathfrak{r} \in [10 \epsilon r_0, r_0]$;
        \item[(iii)] $z_j \to z_0 \in \overline{\NN_{2 \epsilon r_0}(L_0') \cap B_0}$.
    \end{itemize} 
    It is clear that this subsequence exists, by compactness. Note moreover, that we have that $\dist(B_0 \cap L_0', B_0 \cap L_0) < \epsilon r_0$. If not, then for sufficiently large $j$, $\dist(x, L_{j, B_0})< \epsilon r_0$ for each $x \in E_j$, but also, by the convergence in the Hausdorff metric $E_j \to E_0$ we would have that $\dist_{H}(E_j, L_0) \ll \epsilon r_0$, and this is impossible.  Suppose that $\mathfrak{a}_j^+(z_j,r_j)> \lambda$ (as the other case is treated in the same way). Let $z_0'$ the orthogonal projection of $z_0$ onto $L_0$. We see that $|z_0-z_0'|< 2\epsilon r_0$. Then we compute
    \begin{align*}
        \lambda & = \left| \mathfrak{a}_j^+(z_j, r_j) - \mathfrak{a}_0^+(z_0', \mathfrak{r}) \right| \\
        & \leq \left| \int \one_{\Omega_j^+} e^{-|z_j-y|^2/r_j^2}\, dy - \int g^+ e^{-|z_0'-y|^2/\mathfrak{r}^2}\, dy \right| \\
        & \leq \left| \int \one_{\Omega_j^+} (e^{- |z_j-y|^2/r_j^2} - e^{-|z_0'-y|^2/\mathfrak{r}^2})\, dy\right| - \left|\int \left( \one_{\Omega_j^+} - g^+ \right) e^{-|z_0'-y|^2/\mathfrak{r}^2}\, dy\right| 
    \end{align*}
    Making $j$ sufficiently large, we can insure that the above terms are controlled by $2\epsilon$. Then, choosing 
    \begin{equation}\label{e:epsilon<lambda}
   10 \epsilon < \lambda,
    \end{equation} we reach a contradiction. The lemma is proven.
\end{proof}

\vv


\section{Domain splitting}\label{sec-splitting}

In this section we prove the third conclusion of Lemma \ref{mainlemma1-text}, that is, \eqref{e:mainlemma1-b'}. Roughly speaking, this asserts that whenever the square functions are sufficiently small in $E \cap B_0$, then a large part of the component $D^+$ is in $\Omega^+$; here $D^+$ is the component of $\frac12B_0\setminus L_0$ which contains $B^+$, the quasi-corkscrew mostly in $\Omega^+$;  $L_0$ is the plane infimising $\beta_{\infty, E}(B_0)$. The same can of course be said about $D^-$. Let us be more precise.

\begin{lemma}\label{lem-99domain}
	Let $\Omega^+$, $\Omega^- \subset \R^{n+1}$ be disjoint Borel sets. Fix $c_0 \in (0,1)$. For each $\gamma>0$, we can choose the parameters $\beta, \epsilon, \lambda>0$ from Lemma \ref{lemcork1}, \ref{lem-beta<epsilon}, and \ref{lem-corkscrew-everywhere}, respectively, so that there exists a $\delta=\delta(\beta, \epsilon, \lambda, \gamma)>0$ such that the following holds.
	Let $B_0=B(x_0, r_0)$ be a ball centered in a subset $E \subset B_0$. Suppose moreover that there is a Radon measure $\mu$ with $(n, 1)$-polynomial growth such that
	\begin{equation*}
		\mu(B_0) \geq c_0 r_0^n.
	\end{equation*}
	If
	\begin{equation}\label{e:sec11-ve<delta**}
		\int_0^{20r_0} \big(\ve_n(x,r)^2+ \mathfrak{a}(x,r)^2\big) \, \frac{dr}{r} < \delta\quad \mbox{ for all } x \in E
	\end{equation}
	and  $E\cap (B_0\setminus \frac45B_0)\neq\varnothing$,
	then the two components $D^+, D^-$ of $\frac12B_0\setminus \NN_{2\epsilon r_0}(L_0)$ satisfy (up to a relabelling)
	\begin{equation}\label{e:mainlemma1-b-proof}
		\HH^{n+1}(D^+ \setminus \Omega^+) \leq \gamma\, \HH^{n+1}(D^+) \quad \mbox{and} \quad \HH^{n+1}(D^- \setminus \Omega^-) \leq\gamma\, \HH^{n+1}(D^-),
	\end{equation}
where $L_0$ denotes a hyperplane minimizinging $\beta_{\infty, E}(B_0)$.
\end{lemma}

\vv
To prove this result we will first show the following.

\begin{lemma}\label{lem11.2**}
	Let $\Omega^+$, $\Omega^- \subset \R^{n+1}$ be disjoint Borel sets. 
Let $B_0=B(x_0, r_0)$ be a ball and let $E \subset B_0$ be a subset with $x_0\in E$. Suppose moreover that there is a Radon measure $\mu$ with $(n, 1)$-polynomial growth such that, for a given $c_0 \in (0,1)$,
	\begin{equation*}
		\mu(B_0) \geq c_0 r_0^n.
	\end{equation*}
Let $L_0$ be a hyperplane minimizing $\beta_{\infty,E}(B_0)$ and let $P_0$ be the hyperplane parallel to $L_0$ through $x_0$. 
For a given $\kappa\in(0,1/10)$, suppose there is a ball $B^+\subset B_0$ such that 
\begin{equation}\label{eq2eq**}
\HH^{n+1}(B^+\setminus \Omega^+)\leq \beta\,\HH^{n+1}(B^+)\quad \text{ and }\quad \rad(B^+)\geq\kappa\,r_0.
\end{equation}
Let $H_0^+$ be the a closed half-space with boundary $P_0$ and suppose that $4B^+\subset H_0^+$. 
Let $x^+$ be the center of $B^+$
 and consider the annulus and the half-annulus defined by
 \begin{equation*}
        A_{B^+}:= A\left(x_0, |x_0-x^+|- \tfrac12\rad(B^+),  |x_0-x^+|+ \tfrac12\rad(B^+)\right) \,\, \mbox{ and } \,\, A^+_{B^+} := A_{B^+} \cap H_0^+.
    \end{equation*}
Suppose that 	
\begin{equation}\label{e:sec11-ve<delta}
		\int_0^{10r_0} \big(\ve_n(x,r)^2+ \mathfrak{a}(x,r)^2\big) \, \frac{dr}{r} < \delta\quad \mbox{ for all } x \in E,
	\end{equation}
    For all $c_0>1$, $\sigma_1>0$ and $\kappa\in(0,1/10)$,  if $\delta$ and $\beta$ small enough, then the following holds:
    \begin{equation}\label{e:annuli-99percent}
        \HH^{n+1}(A^+_{B^+} \setminus \Omega^+) \leq \sigma_1 \HH^{n+1}(A^+_{B^+}).
    \end{equation}
\end{lemma}

Remark that the ball $B^+$ above is assumed to be a $\beta$-corkscrew ball. We allow its radius to be pretty small compared to $r_0$, as soon as it is at least $\kappa \,r_0$. To prove the lemma, we will need to take $\delta,\epsilon\ll\kappa$.

\begin{figure}\label{figure3}
	\centering
	\includegraphics[scale=1.2]{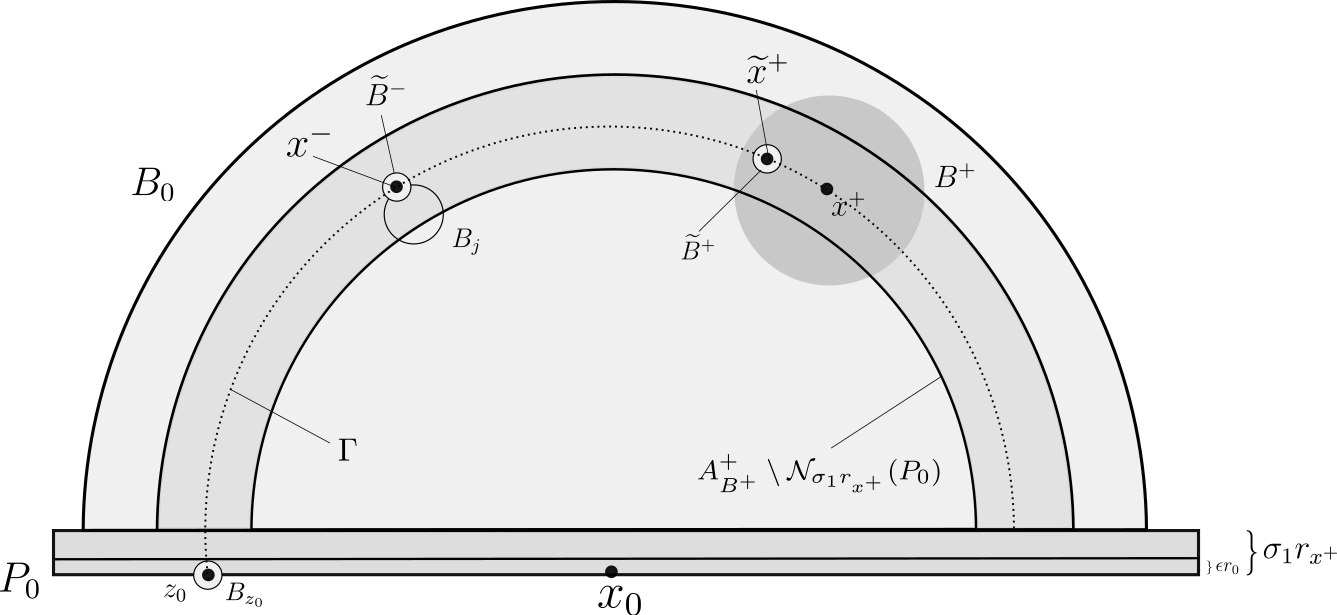}
	\caption{The setup of the proof of Lemma \ref{lem11.2**}}
\end{figure}

\begin{figure}\label{figure3-3D}
	\centering
	\includegraphics[scale=1.2]{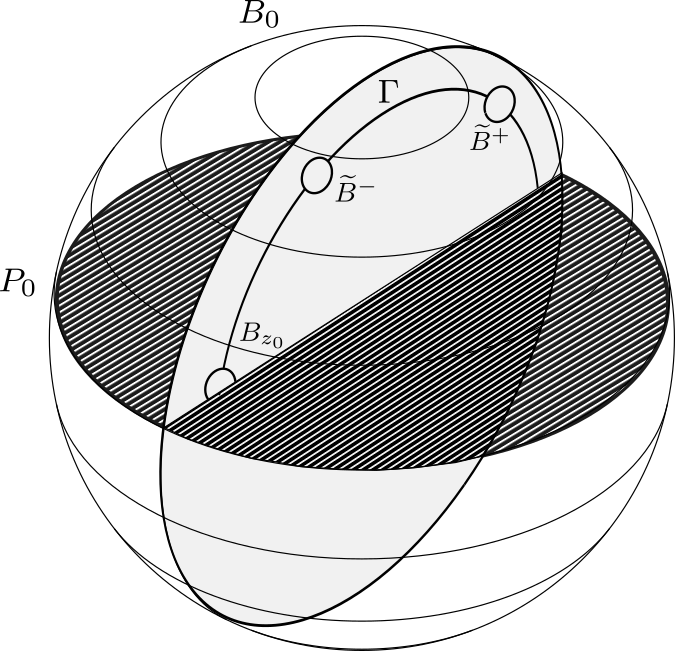}
	\caption{The setup of the proof of Lemma \ref{lem11.2**} in the three dimensional case. Note, however, that the plane which contains $\Gamma$ needn't be orthogonal to $P_0$.}
\end{figure}

\begin{proof}
Without loss loss of generality, we assume that $\sigma_1\ll \min(\kappa,1/10)$.
 By Lemma \ref{lem-beta<epsilon}, we know that $\beta_{\infty, E}(B_0)< \epsilon$, so that $\dist(P_0,L_0)\leq \epsilon\,r_0$.
Also, by Lemma \ref{lem-corkscrew-everywhere}, if $x \in \NN_{2\epsilon r_0}(L_0)\cap B_0$, then $\mathfrak{a}^\pm(x,r) < \lambda$ for each $r \in [10\epsilon r_0,r_0]$. See Figures 3 and 4 for the setup.

We write $r_{x^+} = |x_0-x^+|$, so that
$$A_{B^+}:= A\left(x_0, r_{x^+}- \tfrac12\rad(B^+),  r_{x^+}+ \tfrac12\rad(B^+)\right).$$
Notice that 
$$r_{x^+} \geq 4\,\rad(B^+) \geq 8\kappa\,r_0,$$
because $x_0\not\in 4B^+$.
To prove \rf{e:annuli-99percent}, 
we split
\begin{equation}\label{eqsplit290}
A^+_{B^+} \setminus \Omega^+  \subset \big[A^+_{B^+}\cap \Omega^- \setminus \mathcal{N}_{\sigma_1 r_{x^+}}(P_0)\big] \cup 
\big[A^+_{B^+}\cap \mathcal{N}_{\sigma_1 r_{x^+}}(P_0)\big] \cup  \big[A^+_{B^+}\cap F\big],
\end{equation}
where $F=\R^{n+1}\setminus (\Omega^+\cup\Omega^-$).
Clearly we have
\begin{equation}\label{eqskl28}
\HH^{n+1}(A^+_{B^+}\cap \mathcal{N}_{\sigma_1 r_{x^+}}(P_0))\lesssim \sigma_1\,\HH^{n+1}(A^+_{B^+}).
\end{equation}
Concerning the term $A^+_{B^+}\cap F$, since $A^+_{B^+}\subset B_0$ and for any $r>0$,
$$\HH^n(S(x_0,r)\cap F) \leq \HH^n(S^+(x_0,r)\setminus \Omega^+) +\HH^n(S^-(x_0,r)\setminus \Omega^-)
\leq r^n\,\ve(x_0,r),$$
then we have
\begin{align}\label{eqsplit291}
\HH^{n+1}(A^+_{B^+}\cap F) & \leq \HH^{n+1}(B_0\cap F) = \int_0^{r_0}\HH^n(S(x_0,r)\cap F)\,dr
\leq \int_0^{r_0} r^n\,\ve(x_0,r)\,dr \notag\\
&\lesssim \bigg(\int_0^{r_0} \ve(x_0,r)^2\,\frac{dr}r\bigg)^{1/2} r_0^{n+1}\lesssim \delta^{1/2} \,\HH^{n+1}(B_0) .
\end{align}
Since $\HH^{n+1}(B_0)\lesssim_\kappa \HH^{n+1}(B^+)\lesssim_\kappa \HH^{n+1}(A^+_{B^+})$, for $\delta$ small enough (depending on $\kappa$ and $\sigma_1$) we deduce
\begin{equation}\label{eqskl29}
\HH^{n+1}(A^+_{B^+}\cap F)\leq \sigma_1\,\HH^{n+1}(A^+_{B^+}).
\end{equation}

It remains to estimate $\HH^{n+1}(A^+_{B^+}\cap \Omega^- \setminus \mathcal{N}_{\sigma_1 r_{x^+}}(P_0))$. To this end,
by a trivial application of the Besicovitch covering theorem, we find a family of balls $B_j$, $j\in J$, such that:
\begin{itemize}
\item Each $B_j$ is centered in $A^+_{B^+}\setminus \mathcal{N}_{\sigma_1 r_{x^+}}(P_0)$ and satisfies
$\rad(B_j) = \frac18\sigma_1r_{x^+}$.
\item $A^+_{B^+}\setminus \mathcal{N}_{\sigma_1 r_{x^+}}(P_0)\subset \bigcup_{j\in J} B_j$.
\item The balls $B_j$ have finite superposition, i.e., $\sum_{j\in J}\chi_{B_j}\leq C$.
\end{itemize}
We will show below that
\begin{equation}\label{eqBj}
\HH^{n+1}(B_j\cap \Omega^-) \leq \sigma_1\HH^{n+1}(B_j)\quad \mbox{ for all $j\in J$}.
\end{equation}
From this estimate, the above properties of the balls $B_j$, and the fact that these balls are contained in
$A\left(x_0, r_{x^+}- \rad(B^+),  r_{x^+}+\rad(B^+)\right)$, we get
\begin{align*}
\HH^{n+1}(A^+_{B^+}\cap \Omega^-\setminus \mathcal{N}_{\sigma_1 r_{x^+}}&(P_0)) 
\leq \sum_{j\in J}\HH^{n+1}(B_j\cap \Omega^-) \leq \sigma_1\,\sum_{j\in J}\HH^{n+1}(B_j) \\
& \lesssim \sigma_1\,\HH^{n+1}\left( A\left(x_0, r_{x^+}- \rad(B^+),  r_{x^+}+\rad(B^+)\right)\right) \approx \sigma_1\,
\HH^{n+1}(A^+_{B^+}).
\end{align*}
Together with \rf{eqskl28} and \rf{eqskl29}, this yields
$$\HH^{n+1}(A^+_{B^+} \setminus \Omega^+) \leq \sigma_1 \HH^{n+1}(A^+_{B^+})$$
and proves \rf{e:annuli-99percent} up to harmless constant factor.

Next we turn our attention to \rf{eqBj}. For the sake of contradiction, suppose that
$\HH^{n+1}(B_j\cap \Omega^-) > \sigma_1\HH^{n+1}(B_j)$ for a given $j\in J$.
From Lemma \ref{lemcork2}  we infer that there exists  a $\beta$-corkscrew ball $\wt B^-\subset 2B_j$ for $\Omega^-$
such that $\rad(\wt B^-)\gtrsim_{\sigma_1,\kappa}\rad(B_j) \approx _{\sigma_1,\kappa}\rad(B^+)$.
Taking into account that  $\wt B^-\subset 2B_j$, that $B_j$ is centered in $A^+_{B^+}\setminus \mathcal{N}_{\sigma_1 r_{x^+}}(P_0)$, and that
$\rad(B_j) = \frac18\sigma_1r_{x^+}$, it follows that 
$$\wt B^-\subset A\left(x_0, r_{x^+}- \tfrac34\rad(B^+),  r_{x^+}+ \tfrac34\rad(B^+)\right)\cap H_0^+ 
\setminus \mathcal{N}_{\frac12\sigma_1 r_{x^+}}(P_0).$$
In particular, denoting $r_{x^-} :=|x_0-x^-|$, we also have
$|r_{x^+} - r_{x^-}|\leq \tfrac34\,\rad(B^+).$
Hence there is a ball $\wt B^+\subset B^+$ with the same radius as $\wt B^-$ such that its center $\wt x^+$ satisfies
$$|x_0-\wt x^+| = |x_0- x^-| = r_{x^-}.$$ 
So there is a unique great circle $\Gamma\subset S(x_0,r_{x^-})$ passing through $\wt x^+$ and $x^-$. Since both $\wt x^+$ and $x^-$ belong to the interior of $H_0^+$, it follows that $\Gamma$ intersects $P_0=\partial H_0^+$ in two points. Call
$z_0$ one such point and let $B_{z_0} = B(z_0,\eta\,\rad(\wt B^-))$, where $\eta\in (0,1/10)$ is some small constant which will 
be fixed below depending on $\sigma_1$ and $\kappa$, but neither on $\delta$ nor $\beta$. 

Notice that the three balls $\wt B^+,\wt B^-,B_{z_0}$ are all centered in $\Gamma\cap H_0^+$. Further,
\begin{equation}\label{equpper99}
\wt B^+\cup \wt B^- \subset  H_0^+ 
\setminus \mathcal{N}_{\frac12\sigma_1 r_{x^+}}(P_0)
\end{equation}
and $B_{z_0}$ is centered in $\partial H_0^+$. By Lemma \ref{lem-corkscrew-everywhere}, the latter condition implies that, for any given  small $\lambda>0$, we can assume that 
$$\mathfrak{a}(z_0, r)< \lambda \quad\text{ for all $r \in [10 \epsilon r_0, r_0]$,}$$
assuming $\delta,\epsilon$ small enough.
 In turn, it is easy to check that this implies that there is an absolute constant $\cseven>0$ such that
    \begin{equation}\label{eqbz07}
        \HH^{n+1}(B_{z_0} \cap \Omega^\pm) \geq \cseven\,\HH^{n+1}(B_{z_0})
    \end{equation}
    (recall that $\eta$ will be chosen below depending just on $\sigma_1$ and $\kappa$).
This condition, together with the fact that most of $\wt B^+$ is in $\Omega^+$ and most of $\wt B^-$ in $\Omega^-$, will lead
to a contradiction. The arguments are quite similar to the ones used at the end of Lemma \ref{lemflat1}. However, for completeness we will show the details.

Indeed, recall that $B^+$ and $\wt B^-$ are $\beta$-corkscrews for $\Omega^+$ and $\Omega^-$ respectively. This tells us that
$$\HH^{n+1}(\wt B^+\setminus \Omega^+) \leq \HH^{n+1}(B^+\setminus \Omega^+)\leq \beta\,\HH^{n+1}(B^+) \leq C(\sigma_1,
\kappa)\,\beta\,\HH^{n+1}(\wt B^+) =: \beta_1 \HH^{n+1}(\wt B^+),$$
and also 
$$\HH^{n+1}(\wt B^-\setminus \Omega^-) \leq \beta\,\HH^{n+1}(\wt B^-) \leq \beta_1\HH^{n+1}(\wt B^-).$$
Suppose, for example that $x^-$ is contained in shortest arc from $\Gamma$ joining $z_0$ and $x^+$ (when this does not happen, $x^+$ is contained in shortest arc from $\Gamma$ joining $z_0$ and $x^-$ and the arguments are similar).
Let $I_1= [r_{x^-}- \rad(B_{z_0}), r_{x^-}+ \rad(B_{z_0})]$ and denote by $G$ the subset of those $r\in I_1$ such that
\begin{enumerate}
\item $\HH^n(\wt B^+ \cap S(x_0,r)\setminus \Omega^+)<\beta_1^{1/2}\rad(\wt B^+)^n$,
\item $\HH^n(\wt B^- \cap S(x_0,r)\setminus\Omega^-)< \beta_1^{1/2}\rad(\wt B^-)^n$,
\item $\HH^n(B_{z_0} \cap S(x_0,r)\cap \Omega^+) \geq \frac{1}2\,\cseven\,\HH^{n}(B_{z_0}\cap S(x_0,r))
$, and
\item $\ve(x_0,r) \leq \delta^{1/4}$.
\end{enumerate}
For $k=1,\ldots,4$, we write $r\in G_k$ if $r\in I_1$ satisfies the condition previous condition (k), so that
$G=G_1\cap\ldots\cap G_4$. 

We claim that $\HH^1(G)\gtrsim\rad(B_{z_0})$. Indeed,
by Chebyshev,
\begin{align*}
\HH^1(I_1\setminus G_1) & \leq  \frac1{\beta_1^{1/2}\rad(\wt B^+)^n} \int_{I_1} \HH^n(\wt B^+ \cap S(x_0,r)\setminus\Omega^+)
\,dr  \leq  \frac1{\beta_1^{1/2}\rad(\wt B^+)^n} \,\HH^{n+1}(\wt B^+ \setminus\Omega^+)\\
&  \leq \frac{\beta_1}{\beta_1^{1/2}\rad(\wt B^+)^n} \,\HH^{n+1}(\wt B^+)  \leq C\,\beta_1^{1/2}\,\rad(\wt B^-) = C\,\eta^{-1}\beta_1^{1/2}\HH^1(I_1).
\end{align*}
The same estimate holds for $\HH^1(I_1\setminus G_2)$. Concerning $G_3$, we have
$$\int_{I_1\setminus G_3} \HH^n(B_{z_0} \cap S(x_0,r)\cap \Omega^+)\,dr \leq\frac{1}2\,\cseven\,
\int_{I_1\setminus G_3}\HH^{n}(B_{z_0}\cap S(x_0,r))\,dr\leq \frac{1}2\,\cseven\,\HH^{n+1}(B_{z_0}).$$
Thus, by \rf{eqbz07},
\begin{align*}
\HH^1(G_3)\,\rad(B_{z_0})^n &\gtrsim \int_{G_3} \HH^n(B_{z_0} \cap S(x_0,r)\cap \Omega^+)\,dr \\
& \geq 
\int_{I_1} \HH^n(B_{z_0} \cap S(x_0,r)\cap \Omega^+)\,dr - 
\frac{1}2\,\cseven\,\HH^{n+1}(B_{z_0}) \\
& = \HH^{n+1}(B_{z_0}\cap \Omega^+) - \frac{1}2\,\cseven\,\HH^{n+1}(B_{z_0})
\geq \frac{1}2\,\cseven\,\HH^{n+1}(B_{z_0})\approx \cseven\,\rad(B_{z_0})^{n+1}.
\end{align*}
Consequently, $\HH^1(G_3)\gtrsim \cseven\,\HH^1(I_1)$.
Regarding $G_4$, by Chebyshev and the assumption \rf{e:sec11-ve<delta} it follows
\begin{align*}
\HH^1(I_1\setminus G_4) & \leq \delta^{-1/4}\int_0^{2r_0}\ve(x_0,r)\,dr\leq 
\delta^{-1/4}r_0\left(\int_0^{2r_0}\ve(x_0,r)^2\,\frac{dr}r\right)^{1/2} \leq C(\sigma_1,\kappa,\eta)\,\delta^{1/4}\,\HH^1(I_1).
\end{align*}
So assuming both $\beta_1$ and $\delta$ small enough, we deduce that $\HH^1(G)\gtrsim \rad(B_{z_0})$, as claimed. 

From the claim just proved, we deduce that there exists some $r\in G$ such that $\HH^n(B_{z_0} \cap S(x_0,r))\approx \rad(B_{z_0})^n$ and $\HH^n(\wt B^+ \cap S(x_0,r)) = \HH^n( \wt B^- \cap S(x_0,r))
\approx \rad(\wt B^-)^n$, besides the properties (1)-(4) above. For this $r$, let 
$$\wt \Delta^+ = \wt B^+\cap S(x_0,r), \quad  \wt \Delta^- =  \wt B^-\cap S(x_0,r),\quad \Delta_0 =  B_{z_0}\cap S(x_0,r).$$
It is immediate to check that $\wt \Delta^+$, $\wt \Delta^-$ and $\Delta_0$ are spherical balls centered in a great circle
$\Gamma_r$ from $S(x_0,r)$. Moreover, $\wt \Delta^+\cup \wt \Delta^- \subset  H_0^+ 
\setminus \mathcal{N}_{\frac12\sigma_1 r_{x^+}}(P_0)$ and $\Delta_0$ is centered in $\Gamma_0\cap P_0$.
By Lemma \ref{lembola*} (b), if $\delta$ and $\beta$ are assumed small enough, then
$$\tfrac34\wt \Delta^+\subset H_{x_0,r}^+, \qquad \tfrac34 \wt \Delta^-\subset H_{x_0,r}^-,$$
and by Lemma \ref{lembola*} (a),
$$ \Delta_0 \cap H_{x_0,r}^+\neq\varnothing.$$

Let $z_1\in \Delta_0 \cap H_{x_0,r}^+$ and let $\Gamma_r'$ be the great circle passing through $z_1$ and the center of 
$\wt \Delta^+$. It is easy to check that if $\eta$ is assumed small enough, then $\Gamma_r'$ intersects $\tfrac34\wt \Delta^-$,
and the required smallness of $\eta$ depends only on $\sigma_1$ and $\kappa$, because the condition \rf{equpper99} guarantees that
the centers of $\Delta_0$ and $\wt \Delta_+$ are far from being antipodal points in $S(x_0,r)$ in terms of $\sigma_1$ and $r_{x_+}\approx_{\sigma_1,\kappa} r$. This leads to a contradiction because the arc of $\Gamma_r'$ 
whose end-points are $z_1$ and the center of $\wt \Delta^+$ and intersects $\tfrac34\wt \Delta^-$ is shorter than a half great circle, by \rf{equpper99} and the closeness of $z_1$ to $\partial H_0^+$, and then by the geodesic convexity of $S^+(x_0,r)$, it follows that $\tfrac34\wt \Delta^-\cap S^+(x_0,r)\neq\varnothing$, which is not possible.
So \rf{eqBj} holds and the lemma follows.
\end{proof}

\vv

\begin{proof}[Proof of Lemma \ref{lem-99domain}]
From the condition \rf{e:sec11-ve<delta} it easily follows that 
$$\HH^{n+1}(A(x_0,\tfrac8{10}r_0,\tfrac9{10}r_0)\cap \Omega^+) \approx 
\HH^{n+1}(A(x_0,\tfrac8{10}r_0,\tfrac9{10}r_0)\cap \Omega^-)\approx r_0^{n+1}.$$
So there exist balls $B_1^\pm$ centered in $A(x_0,\tfrac8{10}r_0,\tfrac9{10}r_0)$ with $\rad(B_1^\pm)=\frac1{100}r_0$ such that 
$$\HH^{n+1}(B_1^+\cap \Omega^+) \approx 
\HH^{n+1}(B_1^-\cap \Omega^-)\approx r_0^{n+1}.$$
Since the condition \rf{e:sec11-ve<delta} holds for $x_0$ and for another point $y_0\in E\cap (B_0\setminus \frac45B_0)$, from Lemma \ref{lemcork2}  we infer that there exists  two $\beta$-corkscrew balls $B^\pm$ for $\Omega^\pm$ such that
$$B^\pm\subset 2B_1^\pm\subset
A(x_0,\tfrac3{4}r_0,r_0) \quad \text{ and }\quad \rad(B^\pm)\approx r_0.$$ 
By reducing and moving $B^\pm$ suitably and replacing $\beta$ by a suitable multiple if necessary, we can assume
$8B^\pm \cap L_0=\varnothing$.

As in Lemma \ref{lem11.2**}, let $P_0$ be the hyperplane parallel to $L_0$ through $x_0$.
From the fact that $8B^+ \cap L_0=\varnothing$ and $\rad(B^+)\approx r_0$ it follows that $4B^+\cap P_0=\varnothing$
for $\epsilon$ small enough in \rf{beta-infty-small}. Let $H_0^+$ be the closed half-space with boundary $P_0$ that contains $4B^+$, and let
$$ A_{B^+}:= A\left(x_0, |x_0-x^+|- \tfrac12\rad(B^+),  |x_0-x^+|+ \tfrac12\rad(B^+)\right) \,\, \mbox{ and } \,\, A^+_{B^+} := A_{B^+} \cap H_0^+,$$
where $x^+$ is the center of $B^+$.
Then Lemma \ref{lem11.2**} ensures that, given $\sigma_1>0$ to be chosen below (with $\sigma_1\ll\gamma$), if $\delta$ is small enough, then
\begin{equation}\label{eqsigma100}
        \HH^{n+1}(A^+_{B^+} \setminus \Omega^+) \leq \sigma_1 \HH^{n+1}(A^+_{B^+}).
\end{equation}

We let $D^+$ be the component of $\frac12B_0\setminus \NN_{2\epsilon r_0}(L_0)$ which is contained in $ H_0^+$. 
We split
$$D^+  \setminus \Omega^+  \subset \big[D^+\cap \Omega^- \setminus \mathcal{N}_{\gamma r_0}(P_0)\big] \cup 
\big[D^+\cap \mathcal{N}_{\gamma r_0}(P_0)\big] \cup  \big[D^+\cap F\big].$$
Notice the analogy with \rf{eqsplit290}.
It is clear that
\begin{equation}\label{eqskl28'}
\HH^{n+1}(D^+\cap \mathcal{N}_{\gamma r_0}(P_0))\lesssim \gamma\,r_0^{n+1} \approx \gamma \,\HH^{n+1}(D^+).
\end{equation}
To deal with  $D^+\cap F$, we argue as in \rf{eqsplit291} and \rf{eqskl29}, and then we obtain
\begin{equation}\label{eqskl29**}
\HH^{n+1}(D^+\cap F)\leq \gamma\,r_0^{n+1}\approx \gamma \,\HH^{n+1}(D^+).
\end{equation}

To estimate  $\HH^{n+1}(D^+\cap \Omega^- \setminus \mathcal{N}_{\gamma r_0}(P_0))$, as in the proof of Lemma
\ref{lem11.2**}, by the Besicovitch covering theorem, we find a family of balls $B_j$, $j\in J$, such that:
\begin{itemize}
\item Each $B_j$ is centered in $D^+\setminus \mathcal{N}_{\gamma r_0}(P_0)$ and satisfies
$\rad(B_j) = \frac18\gamma r_0$.
\item $D^+\setminus \mathcal{N}_{\gamma r_0}(P_0)\subset \bigcup_{j\in J} B_j$.
\item The balls $B_j$ have finite superposition.
\end{itemize}
Arguing as in Lemma \ref{lem11.2**}, to prove the first inequality in \rf{e:mainlemma1-b-proof} (up to replacing $\gamma$ by a constant mutiple of $\gamma$), it suffices to show that
\begin{equation}\label{eqBj**}
\HH^{n+1}(B_j\cap \Omega^-) \leq \gamma\,\HH^{n+1}(B_j) \quad \mbox{ for all $j\in J$}.
\end{equation}

For contradiction, suppose that the preceding estimate does not hold for some ball $B_j$.
Then, from Lemma \ref{lemcork2}  we infer that there exists  a $\beta$-corkscrew ball $\wt B^-\subset 2B_j$ for $\Omega^-$
such that $\rad(\wt B^-)\gtrsim_{\gamma}\rad(B_j)  \approx _{\gamma}r_0$.
Taking into account that  $\wt B^-\subset 2B_j$, that $B_j$ is centered in $D^+\setminus \mathcal{N}_{\gamma r_0}(P_0)$, and that
$\rad(B_j) = \frac18\gamma r_0$, for $\gamma$ small enough it follows that 
$$\wt B^-\subset \tfrac{11}{20} B_0\cap H_0^+ 
\setminus \mathcal{N}_{\frac12\gamma r_0}(P_0).$$

Recall that there exists a point $y_0\in E\cap (B_0\setminus \frac45B_0)$. Then we apply Lemma \ref{lem11.2**} with the ball $B(y_0,2r_0)$ in place of $B_0$, with $\wt B^-$, $\Omega^-$ in place of $B^+$ and $\Omega^+$,
and with $c\gamma$ in place of $\kappa$ and also $\gamma$ in place of $\sigma_1$ in that lemma.
Let $P_{y_0}$ be the hyperplane parallel to $L_0$ through $y_0$ and let $H_{y_0}^+$ be the a closed half-space with boundary $P_{y_0}$ that contains $4\wt B^-$ (notice that $4\wt B^-\subset H_0^+ \cap H_{y_0}^+$).
Let $\wt x^-$ be the center of $\wt B^-$
 and consider the annulus and the half-annulus defined by
 \begin{equation*}
        A_{\wt B^-}:= A\left(y_0, |y_0-\wt x^-|- \tfrac12\rad(\wt B^-),  |y_0-\wt x^-|+ \tfrac12\rad(\wt B^-)\right) \,\, \mbox{ and } \,\, A^+_{\wt B^-} := A_{\wt B^-} \cap H_{y_0}^+.
    \end{equation*}
Then, assuming $\delta$ and $\beta$ small enough (depending on $\gamma$ among other parameters), it holds:
    \begin{equation}\label{e:annuli-99percent***}
        \HH^{n+1}(A^+_{\wt B^-} \setminus \Omega^-) \leq \gamma\, \HH^{n+1}(A^+_{\wt B^-}).
    \end{equation}

From the fact that $y_0\in B_0\setminus \frac45 B_0$ and $\wt x^-\in \frac{11}{20} B_0$ it follows that $|y_0-\wt x^-|\geq \frac14r_0$. Then one can check that the half annuli $A^+_{B^+}$ and $A^+_{\wt B^-}$ intersect in a transversal way and that
$$\HH^{n+1}(A^+_{B^+} \cap A^+_{\wt B^-}) \gtrsim \rad(\wt B^-)\,r_0^n,$$
taking into account that the thickness (i.e., the difference between the outer radius and the inner radius) of 
$A^+_{\wt B^-}$ equals $\rad(\wt B^-)$ and the thickness of $A^+_{B^+}$ is comparable to $r_0$. Consequently,
$$\HH^{n+1}(A^+_{B^+} \cap A^+_{\wt B^-}) \gtrsim \HH^{n+1}( A^+_{\wt B^-}),$$
where the implicit constant is absolute.
Therefore,
$$
\HH^{n+1}(A^+_{B^+} \cap A^+_{\wt B^-} \setminus \Omega^-)
\leq \HH^{n+1}(A^+_{\wt B^-} \setminus \Omega^-) \leq \gamma\,\HH^{n+1}(A^+_{\wt B^-})\leq C\,
\HH^{n+1}(A^+_{B^+} \cap   A^+_{\wt B^-}).$$ 
So, assuming $C\gamma\leq1/2$,
\begin{align*}
\HH^{n+1}(A^+_{B^+}\setminus\Omega^+) & \geq \HH^{n+1}(A^+_{B^+} \cap A^+_{\wt B^-} \cap \Omega^-)
\geq (1-C\,\gamma)\,\HH^{n+1}(A^+_{B^+} \cap   A^+_{\wt B^-})\\
& \geq \frac12\,\HH^{n+1}(A^+_{B^+} \cap   A^+_{\wt B^-}) \gtrsim \rad(\wt B^-)\,r_0^n \geq c(\gamma)\,r_0^{n+1} \approx c(\gamma)\,\HH^{n+1}(A^+_{B^+}).
\end{align*}
However, this estimate contradicts \rf{eqsigma100} if $\sigma_1$ is chosen small enough (depending on $\gamma$).
So \rf{eqBj**} holds and the first inequality in \rf{e:mainlemma1-b-proof} follows. The second inequality is proven in a similar way. Further, the estimates in \rf{e:mainlemma1-b-proof} imply that $D^+$ and $D^-$ are located at different
sides from $L_0$.
\end{proof}



\part{Fourier analytic estimates and stopping time arguments}

\section{Sketch of the proof of the second main lemma}


Recall the definitions
\begin{equation}\label{eqcoefgamma}
\gamma(x,r)  = 
\frac1{r^{n}}\,\max_{i=+,-}\HH^n\big((\Omega^i\triangle\, T_x(\Omega^{i,c}))\cap S(x,r)\big),
\end{equation}
where $T_x(y) = 2x-y$.
We also set
\begin{equation}\label{eqcoefg}
\mathfrak g(x,r)  = 
\frac1{r^{n+1}}\,\max_{i=+,-}\HH^{n+1}\big((\Omega^i\triangle\, T_x(\Omega^{i,c}))\cap B(x,r)\big).
\end{equation}

\vv

Now recall the Main Lemma \ref{mainlemma2}.
\begin{lemma}\label{l:main2}
	Let $\Omega^+,\Omega^-\subset\R^{n+1}$ be disjoint Borel sets.
	Fix $c_0\in(0,1)$, $\theta>0$ and $\ve>0$.  Let $B_0$ be a ball, let $E$ be compact, and let $\mu$ be a measure with $n$-polynomial growth with constant $1$ supported on $E$
	 satisfying the following conditions:
	\begin{itemize}
		\item[(a)]  $\mu(B_0\cap E)\geq c_0 r(B_0)^n$.
		
		\item[(b)]   For any ball $B$ with radius at most $200\, \rad(B_0)$ centered at $E$  such that
		$\mu(B)\geq \theta \, \rad(B)^n$, it holds $\beta_{\infty,E}(B)\leq \epsilon$ and moreover, if 
		$E\cap (B\setminus \frac45B)\neq\varnothing$
		and 
		$L_B$ stands for the hyperplane minimizing $\beta_{\infty,E}(B)$, the two components $D^+$, $D^-$ of $\frac12B\setminus L_B$ satisfy \textup{(}up to relabeling\textup{)}
		$$\HH^{n+1}(D^i \cap \Omega^i) \geq (1-\epsilon)\,\HH^{n+1}(D^i) \, \, \mbox{ for } \, i=+,-.$$

		\item[(c)] It holds that
		$$\int_0^{10\rad(B_0)}\mathfrak g(x,r)^2\frac{dr}{r} \leq \ve \,\text{ for every }x\in E.$$
	\end{itemize}
	If $\theta$ is small enough in terms of $c_0$, and $\ve$ is small enough in terms of $c_0$ and $\theta$, then there exists an $n$-dimensional Lipschitz graph $\Gamma$ with slope at most $1/10$ such that
	$$  \mu(\Gamma) \geq c \,\mu(B_0),$$
	for some absolute constant $0<c\leq 1$.
	Moreover, the slope of $\Gamma$ can be made arbitrarily small assuming $\delta$ small enough. 
\end{lemma}

We sketch how this part goes. We follow a strategy that goes back to  David and Semmes \cite[Section 8]{DS1} and L\'eger \cite{leger}. We need to construct a Lipschitz graph $\Gamma$ that  intersects in a portion of $\supp (\mu)$. To this end, we will use a stopping time argument which stops either because the angle from the approximating hyperplane becomes too big ($\BA$, `big angle') or because the measure $\mu$ has too low density ($\LD$).
In order to show that the graph intersects a portion of large measure $\supp(\mu)$, we must show that the places where we stopped, both $\BA$ and $\LD$, have small measure. That $\LD$ has small measure is standard; indeed having low density along a small slope graph means that by a packing argument it can only have a small proportion of the whole measure. The difficulty lies in showing that $\BA$ has small measure. 

It can be shown that if $\Gamma = \{(x, A(x))\}$ is the Lipschitz graph we construct, then
\begin{equation*}
	\mu(\BA) \lesssim\|\nabla A\|_2^2.
\end{equation*}
Thus if $\mu(\BA)$ is large, this forces $\|\nabla A\|_2^2$ to be large. From this, some contradiction should (and will) arise.
By a Fourier inequality, we will estimate $\|\nabla A\|_2^2$ in terms of a smooth square function over the approximating Lipschitz graph $\Gamma$. In turn, this smooth square function will be controlled by the square function $\mathfrak g$ associated with the symmetry coefficients for $\Omega^\pm$. The key lemma will be \ref{lemkey40}, which allow us to estimate the symmetry square function $\mathfrak g$ associated to the Lipschitz domains induced by $\Gamma$ 
in terms of the symmetry square function associated with $\Omega^\pm$.
 The arguments in this step require to argue with the symmetry square function $\mathfrak g$, instead of the smooth square funcion $\mathfrak a_\psi$, as in \cite{JTV}, say. The subtlety here is that we do not have control of $\beta_{\infty}$ over the whole set $\R^{n+1}\setminus(\Omega^+\cup\Omega^-)$ (which equals $\partial\Omega^+$ in \cite{JTV}), but rather of $\beta_{\infty}$ over the support of the measure $\mu$. The advantage of the coefficients $\mathfrak g$, which somewhat resemble the local symmetry condition of David-Semmes \cite{DS1} (but at the level of the domains), is that they contain more information than the smooth square function $\mathfrak a_\psi$.
\vv


\section{The smooth square function $\Apsi$ on Lipschitz graphs} \label{sec-fourier}
Recall that, given an integrable $C^\infty$ function $\psi:\R^{n+1}\to\R$,
and disjoint Borel sets $\Omega^\pm\subset \R^{n+1}$ and $x\in\R^{n+1}$, $r>0$, we denote
$$c_\psi = \int_{y\in\R^{n+1}_+} \psi(y)\,dy,
\qquad \mathfrak a_\psi^\pm(x,r) = \bigg|c_\psi - \frac1{r^{n+1}}\int_{\Omega^\pm} \psi\bigg(\frac{y-x}{r}\bigg)\,dy\bigg|.
$$
We also set
$$\Apsi(x)^2 = \int_0^\infty (\apsun(x,r)^2+\apsdu(x,r)^2)\,\frac{dr}r.$$
Remark that we allow $\psi$ to be non-radial.

We fix an even $C^\infty$ function $\overline{\vphi}:\R\to\R$ such that $\one_{[-1,1]}\leq \overline{\vphi}\leq  \one_{[-1.1,1.1]}$, and we denote
$$\overline{\vphi}_r(\zeta) = \frac1{r} \overline{\vphi}\Big(\frac {\zeta}r\Big),\quad\mbox{ for $\zeta\in\R$, $r>0$.}$$
We denote by $\vphi:\R^n \rightarrow \R$ the radial function such that $\vphi(x)=\overline{\vphi}(|x|)$ and we set 
$$\vphi_r(x_0) = \frac1{r^n} \vphi\Big(\frac {x}r\Big),\quad\mbox{ for $x\in\R^n$, $r>0$.}$$

Our objective in this section is to prove the following.

\begin{lemma}\label{lemlips}
Consider a Lipschitz function $f:\R^n\to\R$ with compact support and let $\Gamma\subset\R^{n+1}$ be its Lipschitz graph. Let
$\Omega^+ =\{(x,y)\in\R^{n+1}:y>f(x)\}$ and $\Omega^- =\{(x,y)\in\R^{n+1}:y<f(x)\}$.
Let $\vphi$ be a function as above and let
$$\psi(x) = \vphi(|x|),\quad \mbox{ for $x\in\R^{n+1}$.}$$
Let $\Apsi$ and $\apsi$ be the associated coefficients defined above.
There exists some $\alpha_0>0$ such that if $\|\nabla f\|_\infty\leq \alpha_0$, then
\begin{equation*}
\int_\Gamma \Apsi(x)^2\,d\HH^n(x)\approx \|\nabla f\|_{L^2(\R^n)}^2.
\end{equation*}
\end{lemma}

We will prove this result by using the Fourier transform. This will play an
essential role in the proof of Main Lemma \ref{l:main2}.
We need first some auxiliary results.

\begin{lemma}\label{lemfourier}
Let $f:\R^n\to\R$ be a Lipschitz function with compact support. Then we have
$$\int_{\R^n}\int_0^\infty \left|\frac{f*\vphi_r(x) - c(\vphi)f(x)}r \right|^2\,\frac{dr}r\,dx = c\,\|\nabla f\|_{L^2(\R^n)}^2,$$
where $c(\vphi) = \int_{\R^n}\vphi\,dx$ and $c>0$.
\end{lemma}

\begin{proof}
By Plancherel, we have
\begin{align*}
\int_{\R^n}\int_0^\infty \left|\frac{f*\vphi_r(x) - c(\vphi)f(x)}r \right|^2\,\frac{dr}r\,dx
& =
\int_{\R^n}\int_0^\infty \left|\frac{\wh f(\xi) \wh \vphi(r\xi) - \wh f(\xi) \wh \vphi(0)}r \right|^2\,\frac{dr}r\,d\xi\\
& =
\int_{\R^n}  |f(\xi)|^2 \int_0^\infty \big| \wh \vphi(r\xi) - \wh \vphi(0)\big|^2\,\frac{dr}{r^3}\,d\xi.
\end{align*}
By the change of variable $r|\xi|=t$, and using that $\wh \vphi$ is radial we get
$$\int_0^\infty \big| \wh \vphi(r\xi) - \wh \vphi(0)\big|^2\,\frac{dr}{r^3}
= |\xi|^2\int_0^\infty \big| \wh \vphi(te_1) - \wh \vphi(0)\big|^2\,\frac{dt}{t^3} .$$
The last integral is finite because 
 $\wh \vphi(te_1) - \wh \vphi(0) = O(t^2)$ as $t\to0$, since $\vphi$
is a radial function in the Schwartz class.
Hence,
$$\int_{\R^n}\int_0^\infty \left|\frac{f*\vphi_r(x) - c(\vphi)f(x)}r \right|^2\,\frac{dr}r\,dx
= \tilde c(\vphi)\int_\R |\xi \wh f(\xi)|^2\,d\xi = c\,\|\nabla f\|_{L^2(\R^n)}^2.$$
\end{proof}

\begin{lemma}\label{lemfourier2}
Let $f:\R^n\to\R$ be a Lipschitz function with compact support. Then we have
\begin{equation}\label{eqfour4}
\int_0^\infty\!\!\int_\R^n \int_{y\in\R^n:|y-x|\leq r} \left|\frac{ c(\vphi)^{-1}(\vphi_r*\nabla f)(x) \cdot(y-x)+f(x)-f(y)}r\right|^2\,\frac{dy}{r^n}\,dx\,\frac{dr}r = c\,\|\nabla f\|_{L^2(\R^n)}^2.
\end{equation}
\end{lemma}

\begin{proof}
By replacing $\vphi$ by $c(\vphi)^{-1}\vphi$ if necessary, we may assume that
$\int \vphi\,dx=1$. This is due to the fact that, as we shall see below, the assumption that $\chara_{[-1,1]}\leq \vphi\leq  \chara_{[-1.1,1.1]}$ is not necessary for the validity of this lemma.

Appealing to the change of variable $z=y-x$ and Fubini's theorem, the left hand side of \rf{eqfour4}
(with $c(\vphi)=1$) equals
\begin{align}\label{eqali7345}
\int_0^\infty\!\!&\int_{z\in\R:|z|\leq r}\int_{x\in\R}  \left|\frac{(\vphi_r*\nabla f)(x)\cdot z+f(x)-f(x+z)}r\right|^2\,dx\,\frac{dz}{r^n}\,\frac{dr}r \\
& \stackrel{\text{Plancharel + Fubini}}{=}
\int_{\xi\in\R^n}\int_0^\infty\!\! \int_{z\in\R^n:|z|\leq r} \left|\frac{2\pi i \xi \cdot z\,\wh\vphi(r\xi)\,\wh f(\xi)+\wh f(\xi)-e^{2\pi i \xi z} \wh f(\xi)}r\right|^2\,\frac{dz}{r^n}\,\frac{dr}r\,d\xi\notag\\
& \qquad\quad= \int_{\xi\in\R^n}\int_0^\infty \!\!\int_{z\in\R^n:|z|\leq r} \left|2\pi i \xi \cdot z\,\wh\vphi(r\xi)+ 1-e^{2\pi i \xi z}\right|^2\,dz\,\frac{dr}{r^{n+3}}\,|\wh f(\xi)|^2\,d\xi.\notag
 \end{align}
 We claim that the inner integral on the right hand side only depends on $|\xi|$. Indeed, let $R:\R^{n}\to\R^n$ be a rotation such that
 $R(\xi) = |\xi|e_1$ and apply the change of variable $z'=R(z)$. Using also the fact that $\wh\vphi$ is radial, then we deduce that the inner integral equals
 \begin{multline*}
 \int_{z'\in\R^n:|z'|\leq r} \left|2\pi i \xi \cdot R^{-1}(z')\,\wh\vphi(r|\xi|e_1)+ 1-e^{2\pi i \xi R^{-1}(z')}\right|^2\,dz'
 \\= \int_{z'\in\R^n:|z'|\leq r} \left|2\pi i |\xi| e_1\cdot z'\,\wh\vphi(r|\xi|e_1)+ 1-e^{2\pi i |\xi| e_1\cdot z'}\right|^2\,dz'
 \end{multline*}
 Now we plug this identity into \rf{eqali7345} and we change variables again writing $y=|\xi|z'$. Then we get
 that the triple integral on the right hand side of \rf{eqali7345} equals
$$ \int_{\xi\in\R^n}\int_0^\infty \!\!\int_{y\in\R^n:|y|\leq |\xi|r} \left|2\pi i y_1\,\wh\vphi(r|\xi|e_1)+ 1-e^{2\pi i  y_1}\right|^2\,\frac{dy}{|\xi|^n}\,\frac{dr}{r^{n+3}}\,|\wh f(\xi)|^2\,d\xi.
 $$
Next we write $t=|\xi|r$, so that the triple integral above becomes
 \begin{multline*}
 \int_{\xi\in\R^n}\int_0^\infty \!\!\int_{y\in\R^n:|y|\leq t} \left|2\pi i y_1\,\wh\vphi(t e_1)+ 1-e^{2\pi i  y_1}\right|^2\,\frac{dy}{|\xi|^n}\,\frac{|\xi|^{n+2}\,dt}{t^{n+3}}\,|\wh f(\xi)|^2\,d\xi\\
 = \int_0^\infty \!\!\int_{y\in\R^n:|y|\leq t} \left|2\pi i y_1\,\wh\vphi(t e_1)+ 1-e^{2\pi i  y_1}\right|^2\,dy\,\frac{dt}{t^{n+3}}\,   \int_{\xi\in\R^n}|\xi\,\wh f(\xi)|^2\,d\xi.
\end{multline*}

 Hence, to prove the lemma it suffices to show that the double integral 
\begin{align*}
I:=\int_0^\infty \!\!\int_{y\in\R^n:|y|\leq t} \left|2\pi i y_1\,\wh\vphi(t e_1)+ 1-e^{2\pi i  y_1}\right|^2\,dy\,\frac{dt}{t^{n+3}}
\end{align*}
 is absolutely convergent and positive.
  It is immediate that this is positive.  To show that this is  absolutely convergent,
 we split it  as follows:
 \begin{align*}
 I & = \int_0^\infty \!
\int_{|y|\leq \min(1,t)}\cdots\; + \int_1^\infty \!\!\int_{1\leq |y|\leq t}\cdots  =: I_1 + I_2.
\end{align*}

First we estimate $I_2$:
\begin{align*}
I_2 &\lesssim \int_1^\infty \!\!\int_{1\leq |y|\leq t}\big(1+ \big|y_1\,\wh \vphi(te_1)\big|^2\big)\,dy\,\frac{dt}{t^{n+3}} \\ &\lesssim
 \int_1^\infty \frac{1}{t^{2}}\,dt +
   \int_1^\infty t^{n+2} \,\big|\wh \vphi(te_1)\big|^2\frac{dt}{t^{n+3}} \lesssim 1+  \int_1^\infty \big|\wh \vphi(te_1)\big|^2\,\frac{dt}{t} \lesssim 1,
\end{align*}
taking into account the fast decay of $\wh \vphi$ at $\infty$ for the last estimate.
Concerning $I_1$, we have
\begin{align} \label{eqfou84}
I_1  \lesssim \int_0^\infty \! \!
\int_{|y|\leq \min(1,t)}&
\left|2\pi i y_1 + 1 -e^{2\pi i y_1} \right|^2\, dy\,\frac{dt}{t^{n+3}}\notag\\
& +
\int_0^\infty \! \!
\int_{|y|\leq \min(1,t)}
\left|2\pi i y_1\,(\wh \vphi(te_1) - 1) \right|^2\, dy\,\frac{dt}{t^{n+3}}.
\end{align}
The first term on the right hand side satisfies
\begin{align*}
\int_0^\infty \! \!\int_{|y|\leq \min(1,t)}
\left|2\pi i y_1 + 1 -e^{2\pi i y_1} \right|^2\, dy\,\frac{dt}{t^{n+3}}
 & \leq
 \int_{B(0,1)} \int_{t\geq |y|}
\left|2\pi i y_1 + 1 -e^{2\pi i y_1} \right|^2\,\frac{dt}{t^{n+3}}\,dy
\\
& \lesssim
\int_{B(0,1)}
\left|2\pi i y_1 + 1 -e^{2\pi i y_1} \right|^2\,\frac{1}{|y|^{n+2}}\,dy
\lesssim 1,
\end{align*}
taking into account that $2\pi i y_1 + 1 -e^{2\pi i y_1}=O(|y|^2)$ as $y\to0$.
Finally we turn our attention to the second term on the right hand side of \eqref{eqfou84}:
\begin{align*}\int_0^\infty \! \!
\int_{|y|\leq \min(1,t)}
\left|2\pi i y_1\,(\wh \vphi(te_1) - 1) \right|^2\, dy\,\frac{dt}{t^{n+3}} &
\lesssim \int_{B(0,1)} |y|^2\, \int_{t\geq |y|}\! \left|\wh \vphi(te_1) - 1 \right|^2\,\frac{dt}{t^{n+3}}\,dy
\\&\lesssim 1 +  \int_{B(0,1)} |y|^2\, \int_{|y|\leq t\leq 1}\! \left|\wh \vphi(te_1) - 1 \right|^2\,\frac{dt}{t^{n+3}}\,dy.
\end{align*}
Since $\vphi\in C^{\infty}$ is radial, and $\wh\vphi(0)=1$, we have $\wh \vphi(t e_1) - 1= O({t}^2)$ as $t\to0$, and so the last integral is bounded above by
 $$C\int_{B(0,1)} |y|^2\, \int_{|y|\leq t\leq 1}\frac{1}{t^{n-1}}\, dt\,dy\lesssim \int_{B(0,1)} |y|^{4-n}\,dy\lesssim1.
$$
So $I_1<\infty$ and the proof of the lemma
is concluded.
\end{proof}
\vv

In the remaining of this section,  we will use the notation $x=(x_0,x_{n+1})$ for $x\in\R^{n+1}$.

\begin{lemma}\label{lem5.4}
Let $f:\R^n\to\R$ be a Lipschitz function with compact support with $\| \nabla f\|_\infty\leq 1/10$ and let $\Gamma$ and
$\Omega^\pm$ be as in Lemma \ref{lemlips}.
For $x=(x_0,x_{n+1})\in\R^n \times \R$, denote
$$\rho(x) = \vphi(x_0)\,\overline{\vphi}(x_{n+1}).$$
Then we have
$$\lca_\rho^\pm(x,r) = \frac{\vphi_r * f(x_0) - c(\vphi)\,f(x_0)}r\quad \mbox{for all $x\in\R^{n+1}$ in the graph of $f$ and all $r>0$}
$$
and
$$\int_\R \A_{\rho}((x_0,f(x_0)))^2\,dx_0 =2
\int_{\R^{n+1}}\int_0^\infty \left|\frac{f*\vphi_r(x_0) - c(\vphi)f(x_0)}r \right|^2\,\frac{dr}r\,dx_0,$$
where $c(\vphi) = \int_\R\vphi\,dx$.
\end{lemma}

\begin{proof}
Observe that
\begin{align*}
c_\rho - \frac1{r^n}\int_{\Omega^+} \rho\bigg(\frac{y-x}{r}\bigg)\,dy & =
\frac1{r^n}\int_{y_0\in\R^n}\int_{y_{n+1}> f(x_0)} \vphi\bigg(\frac{y_0-x_0}{r}\bigg)\overline{\vphi}\bigg(\frac{y_{n+1}-x_{n+1}}{r}\bigg)\,dy_{n+1}\,dy_0 \\
& \quad - \frac1{r^n}
\int_{y_0\in\R^n}\int_{y_{n+1}> f(y_0)} \vphi\bigg(\frac{y_0-x_0}{r}\bigg)\overline{\vphi}\bigg(\frac{y_{n+1}-x_{n+1}}{r}\bigg)\,dy_{n+1}dy_0 \\
& = \int_{y_0\in\R^n} \vphi_r(y_0-x_0)
\int_{f(x_0)}^{f(y_0)} \overline{\vphi}_r(y_{n+1}-x_{n+1})\,dy_{n+1} dy_0
\end{align*}
Observe also that, if $\vphi_{r}(y_0-x_0)\neq 0$, then because $\|\nabla f\|_\infty\leq 1/10$,
$$\overline{\vphi}_r(y_{n+1}-x_{n+1})=\frac1{r}\quad \mbox{ for $y_{n+1}\in [f(x_0),f(y_0)]$.}$$
 As a consequence
$$c_\rho - \frac1{r^n}\int_{\Omega^+} \rho\bigg(\frac{y-x}{r}\bigg)\,dy
= \int_{y_0\in\R^n} \vphi_r(y_0-x_0)
\frac{f(y_0)- f(x_0)}r  dy_0 = \frac{\vphi_r * f(x_0) - c(\vphi)\,f(x_0)}r
.$$
Hence,
$$\A_\vphi(x)^2= 2
\int_0^\infty \lca_\vphi^+(x,r)^2 \,\frac{dr}r= 2\int_0^\infty \left|\frac{\vphi_r * f(x_0) - c(\vphi)\,f(x_0)}r\right|^2 \,\frac{dr}r.$$
Integrating with respect to $x_0$ in $\R$, the lemma follows.
\end{proof}

\vv

\begin{lemma}\label{lemdiff1}
Let $f:\R^n\to\R$ be a Lipschitz function with compact support with $\| \nabla f \|_\infty\leq 1/10$ and let $\Gamma$ and 
$\Omega^\pm$ be as in Lemma \ref{lemlips}.
For $x=(x_0,x_{n+1})\in\R^n\times \R$, denote
$$\rho(x) = \vphi(x_0)\,\overline\vphi(x_{n+1})\quad \text{and}\quad \psi(x)=\vphi(|x|)
.$$
Then we have
$$\int_\Gamma |\A_{\rho}(x) - \Apsi(x)|^2\,d\HH^n(x) \lesssim \| \nabla f\|_\infty^4\,\|\nabla f\|_{L^2(\R^n)}^2.
$$
\end{lemma}

Notice that for the open sets equal to the two components of $\R^{n+1}\setminus \Gamma$, we have $\apsun(x,r)=\apsdu(x,r)$ and $\lca_\rho^+(x,r) = \lca_\rho^-(x,r)$. So quite often we will drop the superindices $+,-$ in this situation.

\begin{proof}
For $r>0$, $x\in\R^{n+1}$, we denote
$$\rho_r(x) = \frac1{r^{n+1}} \rho\Big(\frac xr\Big),\qquad\psi_r(x) = \frac1{r^{n+1}} \psi\Big(\frac xr\Big),
\qquad \vphi_r(x_0) = \frac1{r^n} \vphi\Big(\frac {x_0}r\Big)
.$$
Then we have
\begin{equation}\label{eq931}
\big|\lca_\rho(x,r) - \aps(x,r)\big|\leq \big|(\rho_r * \chara_{\Omega^+} - c_\rho) - (\psi_r * \chara_{\Omega^+} - c_\psi)\big|.
\end{equation}

For $x\in\Gamma$, $r>0$, we denote by $L_{x,r}$ the hyperplane passing through $x$ with slope equal to $c(\vphi)^{-1}(\vphi_r*\nabla f)(x_0)$ (as a graph over $\R^n$), and we let $H_{x,r}^+$, $H_{x,r}^-$ be two complementary half spaces whose common boundary is $L_{x,r}$, so that $H_{x,r}^+$ is above $L_{x,r}$ and $H_{x,r}^-$ is below $L_{x,r}$.

Observe that, by the radial symmetry of $\psi$,
$$c_\psi = \int_{y\in\R^{n+1}_+} \psi(y)\,dy = \psi_r*\chara_{H_{x,r}} (x)\quad \mbox{ for all $x\in\R^{n+1}$ and $r>0$.}$$
We claim that the same identity holds replacing $\psi$ by $\rho$. To check this, suppose that $x=0$ for easiness of
notation and let $y_{n+1}= b \cdot y_0$ be equation of the line $L_{0,r}$. Then, by the evenness of $\vphi$ we have
\begin{align*}
\rho_r*\chara_{H_{0,r}} (0) & =\int \vphi_r(y_0)\int_{y_{n+1}>b \cdot y_0}\overline{\vphi}_r(y_{n+1})\,dy_{n+1}\,dy_0\\
& = \frac12 \int \vphi_r(y_0)\int_{y_{n+1}>b \cdot y_0}\overline{\vphi}_r(y_{n+1})\,dy_{n+1}\,dy_0 + \frac12
\int \vphi_r(y_0)\int_{y_{n+1}<-b \cdot y_0}\overline{\vphi}_r(y_{n+1})\,dy_{n+1}\,dy_0\\
& = \frac12 \int \vphi_r(y_0)\int_{y_{n+1}>b \cdot y_0}\overline{\vphi}_r(y_{n+1})\,dy_{n+1}\,dy_0 + \frac12
\int \vphi_r(y_0)\int_{y_{n+1}<b \cdot y_0}\overline{\vphi}_r(y_{n+1})\,dy_{n+1}\,dy_0\\
& = \frac12 \int \rho_r(y)\,dy =\int_{\R^{n+1}_+} \rho_r(y)\,dy = c_\rho,
\end{align*}
which proves the claim.

From the above identities and \rf{eq931}, for $x\in\Gamma$, we obtain

\begin{align}\label{eqplu89}
\big|\lca_\rho(x,r) - \aps(x,r)\big| & \leq \big|\rho_r * (\chara_{\Omega^+} - \chara_{H_{x,r}^+})(x) - \psi_r * (\chara_{\Omega^+} - \chara_{H_{x,r}^+})(x)\big| \\
& = \big|(\rho_r- \psi_r) * (\chara_{\Omega^+} - \chara_{H_{x,r}^+})(x)\big| \nonumber\\
& \leq \int_{\Omega^+ \Delta H_{x,r}^+} |\rho_r(y-x)- \psi_r(y-x)|\,dy.\nonumber
\end{align}
 But now observe that, if $|x-y|<3r$, then using the fact that $\|\nabla f\|_{\infty}<1/10$, we have
$\overline{\vphi}_r(y_{n+1}-x_{n+1}) = \frac1r$ for all $x\in \Gamma$ and $y\in \Omega^+ \Delta H_{x,r}^+$. Thus,
by the definition of $\rho$ and $\psi$,
$$\rho_r(y-x)- \psi_r(y-x) = \vphi_r(y_0-x_0)\, \overline\vphi_r(y_{n+1}-x_{n+1}) - \frac1r \vphi_r(|y-x|) =
\frac1{r^{n+1}}\,\big( \vphi_r(y_0-x_0) - \vphi_r(|y-x|)\big).$$
Still for $x\in \Gamma$ and $y\in \Omega^+ \Delta H_{x,r}^+$,
notice that if $|x-y|\leq r/2$, then
$\vphi_r(y_0-x_0) =\frac{1}{r^{n-1}}\overline{\vphi}_r(|y-x|) =\frac1r$
and thus $\rho_r(y-x)- \psi_r(y-x)=0$; while if $|x-y|\geq r/2$
$$\big|\rho_r(y-x)- \psi_r(y-x)\big|\leq \frac1{r^n}\|(\overline{\vphi}_r)'\|_\infty \,\big|(y_0-x_0)- |y-x|\big|
\lesssim \frac{|y_{n+1}-x_{n+1}|^2}{r^{n+3}}
.$$
Since $\supp\rho_r \cup\supp\psi_r\subset B(0,3r)$, in any case we get
$$\big|\rho_r(y-x)- \psi_r(y-x)\big|
\lesssim \frac{(\|\nabla f\|_\infty\,r)^2}{r^{n+3}} =  \frac{\|\nabla f\|_\infty^2}{r^{n+1}}
\quad \mbox{for $x\in \Gamma$ and $y\in \Omega^+ \Delta H_{x,r}^+$}
.$$
Plugging this estimate into \rf{eqplu89}, we obtain
$$\big|\lca_\rho(x,r) - \aps(x,r)\big|\lesssim \frac{\|\nabla f\|_\infty^2}{r^{n+1}}\,\HH^{n+1}\big(
(\Omega^+ \Delta H_{x,r}^+)\cap B(x,3r)\big).$$
Next, using the fact that the equation of  the line $L_{x,r}$ is $$y_{n+1} = c(\vphi)^{-1}(\vphi_r*\nabla f)(x_0)\cdot(y_0-x_0) + f(x_0),$$
we get
\begin{align*}
\HH^{n+1}\big(
&(\Omega^+ \Delta H_{x,r}^+)\cap B(x,3r)\big)  \leq \int_{|y_0-x_0|\leq 3r} |c(\vphi)^{-1}(\vphi_r*\nabla f)(x_0)(y_0-x_0) + f(x_0) - f(y_0)|\,dy_0\\
& \lesssim r^{n/2}
\left(\int_{|y_0-x_0|\leq 3r} |c(\vphi)^{-1}(\vphi_r*\nabla f)(x_0)(y_0-x_0) + f(x_0) - f(y_0)|^2\,dy_0\right)^{1/2}.
\end{align*}
Hence,
\begin{align*}
\big|\lca_\rho(x,r) - &\aps(x,r)\big|\\
&\lesssim \frac{\|\nabla f\|_\infty^2}{r^{n/2+1}}\,
\left(\int_{|y_0-x_0|\leq 3r} |c(\vphi)^{-1}(\vphi_r*\nabla f)(x_0)(y_0-x_0) + f(x_0) - f(y_0)|^2\,dy_0\right)^{1/2}
.
\end{align*}
Therefore,
\begin{align*}
|\A_{\rho}&(x) - \Apsi(x)| \\
& = \left|\left(\int_0^\infty \lca_\rho(x,r)^2\,\frac{dr}r\right)^{1/2} -
\left(\int_0^\infty \aps(x,r)^2\,\frac{dr}r\right)^{1/2}
\right|\\
& \leq \left(\int_0^\infty |\lca_\rho(x,r) - \aps(x,r)|^2\,\frac{dr}r\right)^{1/2}\\
&\lesssim
\|\nabla f\|_\infty^2
\left(\int_0^\infty \!\!
\int_{|y_0-x_0|\leq 3r} |c(\vphi)^{-1}(\vphi_r*\nabla f)(x_0)(y_0-x_0) + f(x_0) - f(y_0)|^2\,dy_0
\,\frac{dr}{r^{n+2}}\right)^{1/2}.
\end{align*}
Squaring and integrating on $x$ and applying Lemma \ref{lemfourier2}, we get
\begin{align*}
&\int_\Gamma |\A_{\rho}(x) - \Apsi(x)|^2\,d\HH^1(x) \approx
\int |\A_{\rho}(x) - \Apsi(x)|^2\,dx_0\\
&\; \lesssim
\|\nabla f\|_\infty^4 \!\int_{\R^n} \int_0^\infty \!\!\!
\int_{|y_0-x_0|\leq 3r} \!\!|c(\vphi)^{-1}(\vphi_r*\nabla f)(x_0)(y_0-x_0) + f(x_0) - f(y_0)|^2\,dy_0
\,\frac{dr}{r^{n+2}}\,dx_0 \\
&\;\approx \|\nabla f\|_\infty^4\,\|\nabla f\|_{L^2(\R^n)}^2.
\end{align*}
\end{proof}

\vv
\begin{proof}[\bf Proof of Lemma \ref{lemlips}]
By Lemma \ref{lemfourier} and Lemma \ref{lem5.4}, we have
$$\int_\Gamma \A_{\rho}(x)^2\,d\HH^n(x) \approx
\|\nabla f\|_{L^2(\R^n)}^2.$$
On the other hand, by Lemma \ref{lemdiff1},
$$\int_\Gamma |\A_{\rho}(x) - \Apsi(x)|^2\,d\HH^1(x) \lesssim
\|\nabla f\|_\infty^4\,\|\nabla f\|_{L^2(\R^n)}^2.
$$
Hence,
$$\int_\Gamma \Apsi(x)^2\,d\HH^n(x)\leq 2 \int_\Gamma \A_{\rho}(x)^2\,d\HH^n(x) + 2 \int_\Gamma |\A_{\rho}(x) - \Apsi(x)|^2\,d\HH^n(x) \lesssim \|\nabla f\|_{L^2(\R^n)}^2.
$$
In the converse direction, we have
\begin{align*}
\int_\Gamma \Apsi(x)^2\,d\HH^n(x) & \geq \frac12 \int_\Gamma \A_{\rho}(x)^2\,d\HH^n(x)
- \int_\Gamma \Apsi(x)^2\,d\HH^n(x) \\
&\geq c \|\nabla f\|_{L^2(\R^n)}^2 - C\,
\|\nabla f\|_\infty^4\,\|\nabla f\|_{L^2(\R^n)}^2.
\end{align*}
So if $\|\nabla f\|_\infty^4\leq c/2C$, the lemma follows.
\end{proof}
\vv


\section{Construction of an approximate Lipschitz graph}\label{sec-corona}

 To
prove Lemma \ref{l:main2} we need to construct a Lipschitz graph which covers a fairly big proportion of the measure $\mu$. Many lemmas will not be proved, and we leave it to the reader to apply the (mostly cosmetic) changes from the planar case in \cite{JTV} (or \cite[Chapter 7]{Tolsa-llibre}).


Fix $\epsilon>0$ and $\theta>0$.  We assume that the hypotheses of Lemma \ref{l:main2} are satisfied with these choices of parameters $\epsilon$ and $\theta$.  We shall also introduce $\alpha>0$, $\alpha \ll 1$.  Here $\alpha$ will regulate the slope of a Lipschitz graph that will well approximate the support of $\mu$.  We will eventually determine $\theta$, then we will pick $\alpha$ depending on $\delta$, and finally $\epsilon$ can be chosen to depend on both $\theta$ and $\alpha$.\\

Recall that $E\cap B_0\neq \varnothing$, since  $\mu(E\cap B_0)\geq c_0 \,\rad(B_0)^n$. Also, we require $\supp\mu\subset E$, but we do {\em not} ask $\supp\mu= E$

\begin{claim}
It suffices to prove Lemma \ref{l:main2} under the following conditions:
\begin{itemize}
\item The center of $B_0$ satisfies $x_0\in E$.
\item $\supp\mu\subset B_0$. 
\item $E\cap (100 B_0\setminus 99 B_0)\neq \varnothing$.
\end{itemize}
\end{claim}

\begin{proof}
From the $(n,1)$-growth of $\mu$ and the fact that $\mu(B_0)\geq c_0\,r_0^n$, it follows that 
there exists two balls $B_1$, $B_2$ centered in $E$ such that
$$\dist(B_1,B_2) \gtrsim_{c_0} \rad(B_0),$$
$$\rad(B_0)\lesssim_{c_0}\rad(B_1)=\rad(B_2)\leq \rad(B_0),$$
and 
$$\mu(B_i)\gtrsim_{c_0}\mu(B_0),$$
for $i=1,2$. See Lemma 7.6 from \cite{Tolsa-llibre}, for example. Further, we can assume that $\dist(B_1,B_2)\geq 1000\,\rad(B_1)=1000\,\rad(B_2)$. To check this, cover each ball $B_i$ by a finite family of suitable smaller balls $B_i^j$ and for each $i=1,2$  replace $B_i$ by the ball $B_i^j$ with largest measure. 

Let $\wt B_1$ be the smallest ball concentric with $B_1$ containing $B_2$ (so 
 $\rad(\wt B_1) = \dist(B_1,B_2) + 3\,\rad(B_2)$).
It is easy to check that $B_1\subset \frac1{200}\wt B_1$ and $B_2\subset \wt B_1\setminus 0.99 \wt B_1$. So replacing the ball $B_0$ by $\frac1{100}\wt B_1$, 
$\mu$ by $\mu|_{B_1}$,
  and adjusting suitably the constant $c_0$ and other constants in Lemma \ref{l:main2}, the claim follows.
\end{proof}

\vv
We denote $r_0=\rad(B_0)$ and we set
$$E_0 = E\cap B_0.$$
Besides the conditions in the claim above, we will assume that $x_0=0$ and that the hyperplane $L_{B_0}$ that minimizes
$\beta_{\infty,E}(B_0)$ is horizontal. We denote by $L_0$ the horizontal hyperplane passing through $x_0$, and we identify it with $\R^n$. Notice that if $\dist_H(L_0,L_{B_0})\leq \epsilon\,r_0$, because $x\in E_0$.

Recall that for a ball $B$, $L_B$ denotes a best approximating plane for $\beta_{\infty,E}(B)$. Also we denote
$$\Theta_\mu(B)=\frac{\mu(B)}{\rad(B)^n}.$$
That is, $\Theta_\mu(B)$ is the $n$-dimensional density of $B$ with respect to $\mu$.

Let $x\in E_0=E\cap B_0$ and $0<r\leq 50r_0$. We call the ball $B=B(x,r)$ \emph{good}, and write $B\in\G$ if
\begin{enumerate}[label=\textup{(}\alph*\textup{)}]
\item  $\Theta_{\mu}(B)\geq \theta$, and
\item  $\angle(L_B, L_{B_0})\leq \alpha.$
\end{enumerate}
Therefore, by the assumptions of Lemma \ref{l:main2}\begin{equation}\label{goodballsflat}\beta_{\infty,E}(B)\leq \epsilon\text{ whenever }B\in \G.\end{equation}

We say that $B=B(x,r)$ is \emph{very good}, and write $B\in \VG$, if $B(x,s)\in \G$ for every $r\leq s\leq 50r_0$.
Since $\theta\ll c_0$, we have that any ball $B$ centered on $E_0$ that contains $B_0$ with $r(B)\leq 50 r_0$ is very good.  In particular, $B_0$ is  very good. Notice also that $\beta_{\infty, E}(20B_0)\leq \epsilon$, and thus
$$\dist(x, L_0)\lesssim \epsilon r_0\text{ for all }x\in E\cap 20 B_0.
$$
For $x\in E_0$ we then set
$$h(x) = \inf\{r:0<r<50r_0, \,B(x,r)\in \VG\}.
$$
Observe that $h(x)\leq 2r_0$, as $B(x, 2r_0)\supset B_0$.
Notice that, if $x\in E_0$ and $r\in (h(x), 50r_0)$,  then, from (\ref{goodballsflat}),
$$\Theta_{\mu}(B(x,r))\geq \theta,\; \beta_{\infty,E}(B(x,r))\leq \epsilon, \text{ and }\;\angle(L_{B(x,r)},L_0)\leq \alpha.
$$
Put
$$Z = E_0\cap\{h=0\}.
$$
We now set
$$\LD = \{x\in E_0\backslash Z: \Theta_{\mu}(B(x,h(x)))\leq \theta\},
$$
and
$$\BA = E_0\backslash (\LD\cup Z),
$$
so that
$$E_0 = Z\cup \LD\cup \BA.
$$
Since, for $x\in \BA$, $\Theta_{\mu}(B(x, h(x))\geq \theta$, we must have that $L_{B(x,h(x))}$ has a big angle with $L_0$, moreover

\begin{lemma}\cite[Lemma 7.13]{Tolsa-llibre}\label{l:biganglestop}     Provided that $\epsilon$ is sufficiently small in terms of $\theta$ and $\alpha$, if $x\in \BA$, then
$$\angle(L_{B(x, 2h(x))}, L_0)\geq\frac{\alpha}{2},
$$
for any approximating $n$-plane $L_{B(x,2h(x))}$.
\end{lemma}
\begin{rem}
	As remarked above, the cited Lemma is stated for the construction of a one-dimensional Lipschitz graphs in $\R^2$. The very same construction can be carried out in our context without changes. Hence we omit the details. The same remark applies to all cited lemmas in this section.
\end{rem}


We now introduce a regularized version of the function $h$.  Denote by $\Pi$ the orthogonal projection onto $L_0$ and $\Pi^{\perp}$ the orthogonal projection onto the orthogonal complement of $L_0$.
For $x\in \R^{n+1}$, set
\begin{equation}\label{e:d-smoothing}
d(x) = \inf_{B(z,r)\in \VG}\bigl[|x-z|+r\bigl].
\end{equation}
(Recall here that in order for $B(z,r)\in \VG$ we must have that $z\in E_0=E\cap B_0$.)
Now define, for $p\in L_0$,
$$D(p) = \inf_{x\in \Pi^{-1}(p)}d(x). 
$$
As infima over $1$-Lipschitz functions, we see that $d$ and then $D$ are both $1$-Lipschitz functions.
Observe that $d(x)\leq h(x)$ whenever $x\in E_0$, so the closed set
$$Z_0 = \{x\in \R^{n+1}: d(x)=0\}
$$
contains $Z$.

\begin{lemma}\label{PiperpLip}\cite[Lemma 7.19]{Tolsa-llibre}  For all $x,y\in \R^{n+1}$, we have
$$|\Pi^{\perp}(x)-\Pi^{\perp}(y)|\leq 6\alpha |\Pi(x)-\Pi(y)|+4d(x)+4d(y).
$$
\end{lemma}

As a consequence of this lemma, we have that if $x,y\in Z_0$, then
$$|\Pi^{\perp}(x) - \Pi^{\perp}(y)|\leq 6\alpha |\Pi(x)-\Pi(y)|.
$$
In particular, the map $\Pi:Z_0\to L_0$ is injective and the function
$$A:\Pi(Z_0)\to \R, \; A(\Pi(x))=\Pi^{\perp}(x) \text{ for }x\in Z_0,
$$
is Lipschitz with norm at most $6\alpha$.  To extend the definition of $A$ to $L_0$, we appeal to a Whitney decomposition.

\vv

\subsection{Whitney decomposition}

Let $\DD_{L_0}$ be the collection of dyadic cubes in $L_0$ (which we identify with $\R^{n} \subset \R^{n+1}$ under an affine transformation).
For $I \in \DD_{L_0}$,
\begin{align*}
    D(I) := \inf_{p \in I} D(p), \,\, \text{where }\, D(p) = \inf_{x\in \Pi^{-1}(p)}d(x).
\end{align*}
Set
\begin{align*}
    \Whit := \{ I \mbox{ maximal in } \DD_{L_0} \, :\, \ell(I)  < 20^{-1}\,  D(I)\}.
\end{align*}

 We index $\Whit$ as $\{R_i\}_{i \in I_\Whit}$.  The basic properties of the cubes in $\Whit$ are summarized in the following lemma. The proof of this result is standard, and can be found as Lemma 7.20 in \cite{Tolsa-llibre}.

\begin{lemma}\label{l:whitney-dec}
    The cubes $R_i$, $i\in I_\Whit$, have disjoint interiors in $L_0$ and satisfy the following properties:
    \begin{enumerate}[label=\textup{(}\alph*\textup{)}]
        \item If $x \in 15R_i$, then $5 \ell(R_i) \leq D(x) \leq 50 \ell(R_i)$.
        \item There exists an absolute constant $C>1$ such that if $15R_i \cap 15R_j \neq \varnothing$, then
        \begin{align}
            C^{-1} \ell(R_i) \leq \ell(R_j) \leq C \ell(R_i).
        \end{align}
        \item For each $i \in I_\Whit$, there are at most $N$ cubes $R_j$ such that $15 R_i \cap 15R_j \neq \varnothing$, where $N$ is some absolute constant.
        \item $L_0\setminus \Pi(Z_0) = \bigcup_{i\in I_\Whit} R_i =  \bigcup_{i\in I_\Whit} 15R_i$.
    \end{enumerate}
\end{lemma}



Now set
\begin{align*}
    I_0 := \{ i \in I_\Whit\, :\, R_i \cap B(0, 10r_0) \neq \varnothing\}.
\end{align*}


\begin{lemma}\label{lem73}\cite[Lemma 7.21]{Tolsa-llibre}
    The following two statements hold.
    \begin{itemize}
        \item If $i \in I_0$, then $\ell(R_i) \leq r_0$ and $3R_i \subset L_0 \cap 12 B_0 = (-12r_0,12r_0)^n$.
        \item If $i \notin I_0$, then
        \begin{align*}
            \ell(R_i) \approx \dist(0, R_i) \approx |p| \gtrsim r_0 \mbox{ for all } p \in R_i.
        \end{align*}
    \end{itemize}
\end{lemma}


\begin{lemma}\label{l:whitney-prop1}\cite[Lemma 7.22]{Tolsa-llibre}
    Let $i \in I_0$; there exists a ball $B_i \in \VG$ such that
    \begin{align}
        & \ell(R_i) \lesssim r(B_i) \lesssim \ell(R_i),\text{ and} \\
        & \dist(R_i, \Pi(B_i)) \lesssim \ell(R_i).
    \end{align}
\end{lemma}
For $i \in I_0$, denote by $A_i$ the affine function $L_0 \to L_0^\perp$ whose graph is the $n$-plane $L_{B_i}$. Insofar as $B_i\in \VG$, $\angle(L_{B_i}, L_{B_0})\leq \alpha$, so $A_i$ is Lipschitz with constant $\tan\alpha\lesssim \alpha$.
On the other hand, for $i \in I_\Whit\setminus I_0$, we put $A_i \equiv 0$.  We are now in a position to be able to define $A$ on $L_0$.
\vv




\subsection{Extending $A$ to $L_0$}  Consider a smooth partition of unity $\{\phi_i\}_{i\in I_{\Whit}}$ subordinate to $\{3R_i\}_{i\in I_{\Whit}}$, i.e. $\phi_i\in C^{\infty}_0(3R_i)$ with $\sum_i \phi_i\equiv 1$ on $L_0=\R^n \subset \R^{n+1}$, which moreover satisfies that for every $i\in I_{\Whit}$,

\begin{equation}\nonumber
     \|\nabla \phi_i \|_\infty \lesssim \ell(R_i)^{-1} \ \mbox{ and } \|\nabla^2\phi_i\|_{\infty} \lesssim \ell(R_i)^{-2}.
\end{equation}
(See \cite[p. 250]{Tolsa-llibre} for an explicit construction.)

Now, if $p \in L_0\backslash \Pi(Z_0)$, we set
\begin{align*}
    A(p) := \sum_{i \in I_\Whit} \phi_i(p) A_i(p)= \sum_{i \in I_0} \phi_i(p) A_i(p).
\end{align*}

We require the following lemma, which combines Lemmas 7.24 and 7.27 from \cite{Tolsa-llibre}.
\begin{lemma}\label{l:A}\
   The function $A: L_0 \to L_0^\perp$ is supported in $[-12r_0, 12r_0]^n$ and is Lipschitz with slope $\lesssim \alpha$. Moreover, if $i\in I_{\Whit}$, then for any $x \in 15 R_i$,
   \begin{align*}
       |\nabla^2A(x)| \lesssim_\theta \frac{\epsilon}{\ell(R_i)}.
   \end{align*}
\end{lemma}
We will denote the graph of $A$ by $\Gamma$, that is
\begin{align}\label{e:graphA}
    \Gamma := \mathrm{Graph}(A)=\{ (x, A(x)) \, |\, x \in L_0\}.
\end{align}

\vvv

\subsection{The Lipschitz graph $\Gamma$ and $E_0$ are close to each other}\label{closeprops}

The next four results, concerning the relationship between $\Gamma$ and $E_0$, are central to our analysis.

\begin{lemma} \label{l:dist-B0}\cite[Lemma 7.28]{Tolsa-llibre}
    Every $x \in B(0, 10r_0)$ satisfies
    \begin{align*}
        \dist(x, \Gamma) \lesssim d(x).
    \end{align*}
\end{lemma}

\begin{lemma}\label{l:distQ-L}\cite[Lemma 7.29]{Tolsa-llibre}
    For $B\in \VG$ and $x\in \Gamma\cap 3B$, it holds that
    \begin{align}\label{e:distQ-L}
        \dist(x,L_B) \lesssim_{\theta} \epsilon r(B).
    \end{align}
\end{lemma}

\begin{lemma}\label{l:distG-A}\cite[Lemma 7.30]{Tolsa-llibre}
We have    \begin{align*}
        \dist(x, \Gamma) \lesssim_{\theta} \epsilon \,d(x)\quad\text{for every }x\in E_0.
    \end{align*}
\end{lemma}

Lemmas \ref{l:distG-A} and \ref{l:distQ-L} combine to yield the following statement.

\begin{coro}\label{l:LBcloseGA} Suppose that $B\in \VG$.  As long as $\epsilon$ is small enough in terms of $\theta$,
\begin{align}
\dist(x, \Gamma)\lesssim_{\theta}\epsilon r(B) \quad \text{for all }x\in L_B\cap 2B.
\end{align}
\end{coro}

\begin{proof}  If $B=B(z, r)\in \VG$, then $z\in E_0$ so by Lemma \ref{l:distG-A}, $\dist(z, \Gamma)\lesssim _{\theta}\epsilon r$, so $\Gamma$ (as well as $L_B$) passes close to $z$. On the other hand, Lemma \ref{l:distQ-L} ensures that if $B\in \VG$, then $\Gamma\cap 3B\subset \NN_{C(\theta)\epsilon r(B)}(L_B)$.  Since both $\Gamma$ and $L_B$ are connected, we readily deduce the conclusion.
\end{proof}

\begin{lemma}\label{l:GdistL0}\cite[Lemma 7.32]{Tolsa-llibre}
We have
$$\dist(x,L_0)\lesssim_{\theta}\,\epsilon\,r_0
\quad \text{for all $x\in \Gamma$.}$$
\end{lemma}

\begin{lemma}\label{lemDd}
Assuming $\epsilon$ small enough, for all $x\in E_0$ and for all $x\in\Gamma$, we have
$$D(\Pi(x))\leq d(x) \lesssim D(\Pi(x)).$$
\end{lemma}

\begin{proof}
The inequality $D(\Pi(x))\leq d(x)$ holds for any $x\in\R^{n+1}$ by the definition of $D$. 
For the converse inequality 
suppose first that $x\in \Gamma$. For any $\tau>0$, let $x'\in \R^{n+1}$ be such that 
$\Pi(x) =\Pi(x')$ and $ d(x') \leq D(\Pi(x')) +\tau$. Then $|x'-x|\approx \dist(x',\Gamma) \lesssim d(x')$
by Lemma \ref{l:dist-B0}. Thus, since $d$ is $1$-Lipschitz,
\begin{equation}\label{eqsame47}
d(x) \leq d(x') + |x-x'|\lesssim d(x') \leq D(\Pi(x')) +\tau = D(\Pi(x)) +\tau.
\end{equation}

In the case $x\in E_0$, consider again $z\in \Gamma$ such that $\Pi(x)=\Pi(z)$, so that by Lemma \ref{l:distG-A},
$$|x-z|\approx \dist(x,\Gamma)\lesssim_\theta \epsilon \,d(x).$$
Then, 
$$d(x) \leq |x-z| + d(z) \leq C(\theta)\epsilon \,d(x) + C\,D(\Pi(z)) = C(\theta)\epsilon \,d(x) + C\,D(\Pi(x))$$
which implies that $D(\Pi(x))\leq d(x) \lesssim D(\Pi(x))$ for $\epsilon$ small enough.
\end{proof}

\vv

\section{Small measure of $\LD$ and $\BA$}\label{sec-coronaest}

\subsection{$\LD$ has small measure}


The following lemma shows that $\LD$ has small measure.  The reason for this is that $\LD$ can be covered by balls of small density that are closely aligned to the $n$-dimensional Lipschitz graph $\Gamma$.

\begin{lemma} \label{l:LD-small}\cite[Lemma 7.33]{Tolsa-llibre}
    If $\theta$ is sufficiently small, and $\epsilon>0$ is sufficiently small in terms of $\theta$, then
    \begin{align*}
    \mu(\LD) \leq \frac{1}{1000} \mu(B_0).
    \end{align*}
\end{lemma}

This lemma determines our choice of $\theta$.

\vv

\subsection{$\BA$ has small measure}

Our objective in this section is to prove the following.

\begin{lemma}\label{l:BA-small}
If $\alpha$ is chosen sufficiently small, and $\epsilon$ is chosen sufficiently small, with respect to  $\alpha$ and $\theta$, then
\begin{align}
    \mu(\BA) \lesssim \epsilon^{1/(n+3)} \mu(E_0).
\end{align}
\end{lemma}

Recall our assumption that the hyperplane $L_0$ coincides with the $\{x\in\R^{n+1}:x_{n+1}=0\}$, and so $L_0^\bot$ is the vertical axis. For a radial function $\psi$ as in Lemma \ref{lemlips}, we 
denote by $\lca_{\Omega,\psi}$ and $\A_{\Omega,\psi}$ the respective square functions $\aps$ and $\Apsi$ associated with
the sets $\Omega^+,\Omega^-$. The analogous square functions associated
with the Lipschitz domains
$$V^+=\{x\in\R^{n+1}:\Pi^\bot(x)>A(\Pi(x))\},\quad V^-=\{x\in\R^{n+1}:\Pi^\bot(x)<A(\Pi(x))\}
$$
are denoted by $\lca_{V,\psi}$ and $\A_{V,\psi}$.

\vv

For the reader's sake, we recall here the assumption (b) of Lemma \ref{l:main2}. 

\begin{assum}\label{l:Gtopology1}

For any ball $B$ with radius at most $200\, \rad(B_0)$ centered at $E$  such that
		$\mu(B)\geq \theta \, \rad(B)^n$, it holds $\beta_{\infty,E}(B)\leq \epsilon$ and moreover, if 
		$E\cap (B\setminus \frac45B)\neq\varnothing$
		and 
		$L_B$ stands for the hyperplane minimizing $\beta_{\infty,E}(B)$, the two components $D^+_{\frac12B}$ and $D^-_{\frac12B}$ of $\tfrac12B\setminus L_B$ satisfy \textup{(}up to relabeling\textup{)}
 $$\HH^{n+1}(D^+_{\frac12B} \cap \Omega^+) \geq (1-\epsilon)\,\HH^{n+1}(D^+_{\frac12B})$$
and 
$$\HH^{n+1}(D^-_{\frac12B} \cap \Omega^-) \geq (1-\epsilon)\,\HH^{n+1}(D^-_{\frac12B)}).$$
\end{assum}


 Applying this lemma to a ball $B'\in \G$ containing $15B_0$, interchanging the upper half plane by the lower half plane if necessary, we deduce that
 \begin{equation}\label{eq500con}
 \HH^{n+1}(15B_0 \cap \R_i^{n+1}\cap \Omega^i) \geq (1-C\epsilon)\,\HH^{n+1}(15B_0 \cap \R_i^{n+1})\quad 
\mbox{ for $i=+,-$.}
\end{equation}
 
\vv

\begin{lemma}\label{l:Gtopology2}
Assume $\epsilon$ small enough in terms of $\theta$. For every ball $B\in \VG$ it holds
$$\HH^{n+1}(B \cap V^+ \setminus \Omega^+) \lesssim_\theta \epsilon \,\HH^{n+1}(B)$$
and 
$$\HH^{n+1}(B \cap V^- \setminus \Omega^-) \lesssim_\theta \epsilon \,\HH^{n+1}(B).$$
\end{lemma}

\begin{proof}
	We prove the lemma for $V^+$ and $\Omega^+$. 
	We claim that for every $B\in\VG$ there exists a concentric ball $B'$ with $2\,\rad(B)\leq \rad(B')\lesssim_\theta
	\rad(B)$ such that $E\cap (B'\setminus \frac45B')\neq\varnothing$. This follows immediately from the Assumption 
	\ref{l:Gtopology1} when $B$ is a large ball, that is when $c(\theta)\,r_0\leq \rad(B)\leq 50r_0$, where $c(\theta)$ is a small constant to be fixed in a moment.
	On the other hand, for smaller balls, we use the fact that for any ball $B''\in \G$ and $\lambda\in(0,1)$, by the $(n,1)$-growth of $\mu$ and the stopping condition for low density balls, we have
	$$\mu(\lambda B'')\leq \lambda^n\,\rad(B'')^n \leq \lambda^n\,\theta^{-1}\,\mu(B'').$$
	So for $\lambda=\lambda(\theta) = (\frac\theta2)^{1/n}$, 
	$$\mu(B''\setminus \lambda B'')\geq \frac12 \,\mu(B'')).$$
Consequently, the ball $20\lambda(\theta)^{-1} B$ satisfies $\mu(20\lambda(\theta)^{-1} B\setminus 20 B)>0$, assuming 
$20\lambda(\theta)^{-1}\rad(B)\leq 50r_0$. From this fact, it easily follows that we can find a ball $B'$ as claimed above, whenever $20\lambda(\theta)^{-1}\rad(B)\leq 50r_0$.

	By the Assumption \ref{l:Gtopology1} applied to $B''$, we know that $\HH^{n+1}(D^+_{\frac12 B''} \setminus \Omega^+) \lesssim \epsilon \HH^{n+1}(D^+_{\frac12B''})$. Since $B\subset \frac12 B''$ and $\rad(B'')\approx_\theta
	\rad(B)$, it follows that
	$$\HH^{n+1}(D^+_{B} \setminus \Omega^+) \leq \HH^{n+1}(D^+_{B''} \cap B\setminus \Omega^+) +
	\HH^{n+1}(D^+_B\,\triangle\, (D^+_{B''}\cap B)) \lesssim_\theta \epsilon\,\rad(B)^{n+1},$$
	where we took into account that $\dist_H(B\cap L_B,B\cap L_{B''})\lesssim_\theta \epsilon\,\rad(B)$ for the last inequality.
 Now, recall from Lemma \ref{l:distQ-L} that if $x \in \Gamma \cap 3B$, where $\Gamma$ is the graph of $A$ (see \eqref{e:graphA}), then $\dist(x, L_B) \lesssim_\theta \epsilon\,\rad(B)$. In other words, 
	\begin{equation*}
		\Gamma \cap B \subset \NN_{C_\theta \epsilon\, \rad(B)}(L_B).
	\end{equation*}
	In particular, this implies that 
	\begin{equation*}
	(B \cap V^i \setminus \Omega^+) \setminus \NN_{C_\theta \epsilon \rad(B)}(L_B) \subset D^+_{B}\setminus \Omega^+
	\quad \text{either for $i=+$ or $i=-$.}
	\end{equation*}
	Note moreover that $\HH^{n+1}(\NN_{C_\theta \epsilon\, \rad(B)}(L_B)) \lesssim_{\theta} \epsilon \HH^{n+1}(D^+_{B})$.
	Hence we conclude that, either for $i=+$ or $i=-$,
	\begin{align*}
		& \HH^{n+1}(B \cap V^i \setminus \Omega^+)\\
		&= \HH^{n+1}((B \cap V^i \setminus \Omega^+ )\cap \NN_{C_\theta \epsilon \, \rad(B)}(L_B))  + \HH^{n+1}((B \cap V^i\setminus \Omega^+ )\setminus \NN_{C_\theta \epsilon \, \rad(B)}(L_B))\\
		& \lesssim_\theta  \epsilon \HH^{n+1}(B) + \epsilon \HH^{n+1}(D^+_{B}) \lesssim_{\theta} \epsilon \, \HH^{n+1}(B).
	\end{align*}
	This estimate holds for $B$ and also for any ball $\wt B$ concentric with $B$ such that $\rad(B)\leq \rad(\wt B)
	\leq 50r_0$, with $i=+$ or $i=-$ depending on $\wt B$. Then by the condition \rf{eq500con} and connectivity it follows easily that $i=+$ for any such ball $\wt B$ and so also for $B$. We leave the details for the reader.
\end{proof}

\vv

To prove Lemma \ref{l:BA-small} we will show first that if $\BA$ has noticeable measure, then $\|\nabla A\|_2^2$ is (relatively) large, and later we will see that this in turn contradicts the smallness assumption of the square function involving the coefficients $\mathfrak g(x,r)$ in Main Lemma \ref{l:main2}.
 The proof is split into several lemmas.
The first one is analogous to Lemma 7.35 of \cite{Tolsa-llibre}, which is a consequence of Lemmas \ref{l:biganglestop} and \ref{l:distG-A}.
\begin{lemma} \label{l:BA-a}\cite[Lemma 7.35]{Tolsa-llibre}
    Provided that $\epsilon$ is small enough in terms of $\alpha$,
    \begin{align}
        \mu(\BA) \lesssim \alpha^{-2} \|\nabla A\|_2^2.
    \end{align}
\end{lemma}

Assume that $\alpha$ is small enough to apply Lemma \ref{lemlips}.  Then
\begin{equation}\label{eqga534}
\int_{\Gamma} \A_{V,\psi}(x)^2\,d\HH^1(x)\approx \|\nabla A\|_{2}^2,
\end{equation}
where $\A_{V,\psi}$ stands for the square function $\A_\psi$ associated with the graph $\Gamma$ and the components $V^\pm$ of $\R^{n+1}\setminus \Gamma$, with $\psi$ and $\vphi$ as in Lemma \ref{lemlips}.
From this and Lemma \ref{l:BA-a} we infer that
\begin{equation}\label{eqba10}
\mu(\BA) \lesssim  \alpha^{-2} \int_{\Gamma} \A_{V,\psi}(x)^2\,d\HH^n(x).
\end{equation}
Our next objective is to estiimate
$\int_{\Gamma} \A_{V,\psi}(x)^2\,d\HH^n(x)$.

We denote
$$\ell(x):=\frac1{50}D(x)=\frac1{50}D(\Pi(x)).$$
Lemma \ref{l:whitney-dec} ensures that if $x\in 15I$, $I\in\Whit$, then
\begin{equation}\label{eqellI}
\frac1{10}\ell(I)\leq\ell(x)\leq \ell(I).
\end{equation}
We set
\begin{align*}
    \wt \A_{V, \psi}(x)^2 := \int_{\ell(x)}^{r_0} \lca_{V, \psi}(x,r)^2\, \dr.
\end{align*}

It will be convenient to denote $L^2(\Gamma)=L^2(\HH^n|_{\Gamma})$.
The following lemma is proved like Lemma 11.5 from \cite{JTV} and so we omit the detailed arguments.

\begin{lemma}\label{l:BA-est2}
    We have
    \begin{align*}
        \|\A_{V, \psi}- \wt \A_{V, \psi}\|_{L^2(\Gamma)}^2 \lesssim C(\theta)\epsilon^2 r_0^n + \alpha^4 \|\nabla A\|_{2}^2.
    \end{align*}
\end{lemma}
\vv

Observe that, from \rf{eqga534} and the previous lemma, we have that
\begin{equation}\begin{split}\nonumber\|\nabla A\|_{2}^2&\lesssim \|\wt \A_{V, \psi}\|_{L^2(\Gamma)}^2+   \|\A_{V, \psi}- \wt \A_{V, \psi}\|_{L^2(\Gamma)}^2\\&\lesssim \|\wt \A_{V, \psi}\|_{L^2(\Gamma)}^2 +
 C(\theta)\epsilon^2 r_0^n + \alpha^4 \|\nabla A\|_{2}^2.\end{split}
 \end{equation}
Hence, for $\alpha$ small enough, we obtain
$$\|\nabla A\|_{2}^2\lesssim \|\wt \A_{V, \psi}\|_{L^2(\Gamma)}^2 +
 C(\theta)\epsilon^2 r_0^n,$$
and combining this inequality with \rf{eqba10} we get
\begin{align} \label{e:BA-100}
    \mu(\BA) \lesssim    \alpha^{-2} \|\wt \A_{V, \psi} \|_{L^2(\Gamma)}^2 + C(\theta)\epsilon^2 \alpha^{-2}\,r_0^n.
\end{align}

To estimate $\|\wt \A_{V, \psi} \|_{L^2(\Gamma)}^2$, we split
\begin{align} \label{e:alpha-G-Gamma}
\|\wt \A_{V, \psi} \|_{L^2(\Gamma)}^2 &= \int_{\Gamma}\int_{\ell(x)}^{r_0} \lca_{V, \psi}(x,r)^2 \, \dr  \, d \hi(x) \nonumber\\
 &=
  \int_{\Gamma}\int_{\ell(x)}^{\min(\epsilon^{-1}\ell(x),r_0)} \cdots
  +  \int_{\Gamma}\int_{\min(\epsilon^{-1}\ell(x),r_0)}^{r_0} \cdots.
 \end{align}
Concerning the first integral on the right hand side, we have:

\begin{lemma}\label{l:lemfac64}
We have
\begin{align} \label{e:alpha-G-Gamma2}
\int_{\Gamma}\int_{\ell(x)}^{\min(\epsilon^{-1}\ell(x),r_0)} \lca_{V, \psi}(x,r)^2 \, \dr \, d \hi(x) \lesssim_\theta \epsilon^2 |\log\epsilon| r_0^n.
\end{align}
\end{lemma}

The arguments for this lemma are the same as the ones in Lemma 11.6 from \cite{JTV}, and so we skip them again.

To estimate the second integral on the right hand side of \rf{e:alpha-G-Gamma} we need some additional notation.
We let $\Pi_{\Gamma}$ be the projection $\R^{n+1} \to \Gamma$ orthogonal to $L_0$.
We let $\DD_{\Gamma}$ be the family of ``dyadic cubes" on $\Gamma$ of the form
$$\DD_{\Gamma}= \big\{\Pi_{\Gamma}(I):I\in\DD_{L_0}\big\}.$$
We define the side-length of $R\in \DD_{\Gamma}$ (and write $\ell(R)$) to be equal to $\ell(I)$, where $I\in \DD_{L_0}$ satisfies $R=\Pi_{\Gamma}(I)$.  Then $\ell(R)^n$ is comparable to $\mathcal{H}^n(R)$.
Then we set
\begin{equation}\label{defwgam}
\pStop = \big\{\Pi_{\Gamma}(I):I\in\Whit\big\}.
\end{equation}
%

Denote by $\ppStop = \{\Pi_{\Gamma}(R_i): i\in I_0\}$, so $I\in \ppStop$ if $I = \Pi_{\Gamma}(R_i)$ for some $R_i$ that intersects $B(0, 10r_0)$. The following is the analog of Lemma 11.7 from \cite{JTV}.

\begin{lemma}\label{l:15B0} If $I\in \ppStop$, $x\in I$, and $0<r\leq r_0$, then $B(x, 2r)\subset 15B_0$.
\end{lemma}

We now claim that there is an absolute constant $\Celeven$ such that for each $I \in \pStop$ that intersects $15B_0$, there exists a ball $B_I\in \VG$ (centered at $z_I\in E_0$) such that
\begin{align}\label{cond22*}
    & \Celeven^{-1} \rad(B_I) \leq \ell(I)  \leq \Celeven \rad(B_I),\\
 \label{cond23*}   & 
  I\subset \Celeven B_I,\text{ and }\Pi_{\Gamma}(z_{I})\in \Celeven B_I.
\end{align}
The arguments are the same as the ones following (11.6) in \cite{JTV}, and so we omit them again.

\vv
The following auxiliary result is proven in the same way as Lemma 7.41 in \cite{Tolsa-llibre} 
and Lemma 11.8 in \cite{JTV}.

\begin{lemma} \label{lemrepart}
For each $I\in \ppStop $ there exists some function $g_I\in L^\infty(\mu)$, $g_I\geq0$ supported
on $B_I$ such that
\begin{equation} \label{co1}
\int g_I\,d\mu = \HH^n(I),
\end{equation}
and
\begin{equation} \label{co2}
\sum_{I\in \ppStop} g_I \lesssim c(\theta).
\end{equation}
\end{lemma}

The first main step to bound the 
first integral on the right hand of \rf{e:alpha-G-Gamma} is contained in the following lemma.

\begin{lemma}\label{lemllumy}
We have
\begin{align}\label{eqx2}
\int_{\Gamma}\int_{\min(\epsilon^{-1}\ell(x),r_0)}^{r_0} \lca_{V, \psi}(x,r)^2 \, \dr  &\, d \hi(x)
\nonumber
\\ & \lesssim_{\theta} \epsilon^2r_0^n +
\int_{E_0}\int_{\min(c\epsilon^{-1}\ell(x),r_0)}^{r_0} \lca_{V, \psi}(x,r)^2 \, \dr  \, d \mu(x),
\end{align}
for some absolute constant $c>0$.
\end{lemma}

Throughout the proof of Lemma \ref{lemllumy}, we will let the implicit constant in the symbol $\lesssim$ depend on $\theta$ without further mention.

\begin{proof}
By Lemma \ref{lem73}, $\epsilon^{-1}\ell(x)\geq r_0$ for $x\in I\in\pStop\setminus \ppStop$, so we may write
\begin{align}\label{eqajg32}
\int_{\Gamma}\int_{\min(\epsilon^{-1}\ell(x),r_0)}^{r_0} \lca_{V, \psi}&(x,r)^2 \, \dr  \, d \hi(x)\nonumber
\\
&= \sum_{I\in\ppStop} \int_{I}\int_{\min(\epsilon^{-1}\ell(x),r_0)}^{r_0} \lca_{V, \psi}(x,r)^2 \, \dr  \, d \hi(x).
\end{align}

Given $x\in I\in\ppStop$, we consider an arbitrary point $x'\in B_{I}$.
Using the fact that $r\geq \epsilon^{-1}\ell(x)\approx \epsilon^{-1}\ell(I)\gg\ell(I)$ and
taking into account that $|x-x'|\leq \dist(I,B_{I})\lesssim\ell(I)$ by \rf{cond23*}, we get
 \begin{align*}
 \av{ \lca_{V, \psi} (x,r) - \lca_{V, \psi}(x',r)}   &\leq r^{-n-1}\int_{V^+}
 \bigg| \psi\ps{\frac{x-y}{r}} - \psi\ps{\frac{x'-y}{r}}\bigg|\,dy\\
 & \lesssim \|\nabla\psi\|_\infty \int_{B(x,2r)}\frac{|x-x'|}{r^{n+2}}\,dy \lesssim \frac{\ell(I)}r,
\end{align*}
Plugging this inequality into the right hand side of \rf{eqajg32}, we estimate the integral on the left side of \rf{eqx2} as follows:
\begin{align*}
\int_{\Gamma}\int_{\min(\epsilon^{-1}\ell(x),r_0)}^{r_0}& \lca_{V, \psi}(x,r)^2 \, \dr \,d \hi(x)
\\
& \lesssim_\theta \sum_{I\in \ppStop}\int_I\int_{\min(c\epsilon^{-1}\ell(I),r_0)}^{r_0} \inf_{x'\in  B_I}
\lca_{V, \psi}(x',r)^2\, \dr \,d \hi(x)\\
&\quad+ \sum_{I\in \ppStop}\int_I\int_{\min(c\epsilon^{-1}\ell(I),r_0)}^{r_0} \frac{\ell(I)^2}{r^2} \, \dr \,d \hi(x)\\
&=: T_1+ T_2.
\end{align*}
First we bound $T_2$ in a straightforward manner by evaluating the double integral:
$$T_2\lesssim
\sum_{I\in \ppStop}\ell(I)^n \int_{c\epsilon^{-1}\ell(I)}^\infty \frac{\ell(I)^2}{r^{3}} \, dr \lesssim
\epsilon^2 \sum_{I\in \ppStop}\ell(I)^n \lesssim \epsilon^2\,r_0^n.
$$

Finally we estimate the term $T_1$. To this end, we consider the functions $g_I$ constructed in
Lemma \ref{lemrepart}. It is clear that
$$T_1 =\sum_{I\in \ppStop}\iint_{\min(c\epsilon^{-1}\ell(I),r_0)}^{r_0} \inf_{x'\in  B_I}
\lca_{V, \psi}(x',r)^2\, \dr \,g_I(x)\,d \mu(x).$$
Observe now that, for each $x\in B_I$, since $\ell(\cdot)$ is $\frac1{50}$-Lipschitz and taking into account
\rf{eqellI},
denoting by $x_I$ the center of $I$,
$$\ell(x)  \leq \ell(I) + \frac1{50}\,|x-x_I|\lesssim \ell(I).$$
Therefore, using also \rf{co2} and the fact that we are assuming that $\supp\mu\subset E_0$,
$$T_1 \leq\iint_{\min(c\epsilon^{-1}\ell(x),r_0)}^{r_0} \sum_{I\in \ppStop}g_I(x)\,
\lca_{V, \psi}(x,r)^2\, \dr \,d \mu(x) \lesssim 
\int_{E_0}\int_{\min(c'\epsilon^{-1}\ell(x),r_0)}^{r_0} 
\lca_{V, \psi}(x,r)^2\, \dr \,d \mu(x)
.$$
Gathering the estimates obtained for the terms $T_1$ and $T_2$, the lemma follows.
\end{proof}

\vv

Next we intend to estimate the square integral on the right hand side of \rf{eqx2} in terms of
the coefficients $\mathfrak g(x,r)$ defined in \rf{eqcoefg}. To distinguish the coefficients associated
to $\Omega^\pm$ from the ones associated to $V^\pm$, we write
$$\mathfrak g_\Omega(x,r)  = 
\frac1{r^{n+1}}\,\max_{i=+,-}\HH^{n+1}\big((\Omega^i\triangle\, T_x(\Omega^{i,c}))\cap B(x,r)\big)
$$
and
$$\mathfrak g_V(x,r)  = 
\frac1{r^{n+1}}\,\max_{i=+,-}\HH^{n+1}\big((V^i\triangle\, T_x(V^{i,c}))\cap B(x,r)\big).
$$
We also define coefficients $\gamma_\Omega(x,r)$ and $\gamma_V(x,r)$ as in \rf{eqcoefgamma} associated to $\Omega^\pm$ and $V^\pm$, respectively.
\vv

\begin{lemma}\label{lem9.11}
We have
$$\int_{E_0}\int_{\min(c\epsilon^{-1}\ell(x),r_0)}^{r_0}\lca_{V, \psi}(x,r)^2 \, \dr  \, d \mu(x)
\lesssim 
\int_{E_0}\int_{\min(c\epsilon^{-1}\ell(x),r_0)}^{2r_0} \mfg_V(x,r)^2 \, \dr  \, d \mu(x).$$
\end{lemma}

\begin{proof}
Integrating in polar coordinates as in \rf{eqal82} and applying Lemma \ref{lemgamma}, 
 we get
\begin{align*}
\lca_{V, \psi}^\pm(x,r) 
& \leq \frac1{r^{n+1}}\int_0^\infty
\varphi\Bigl(\frac{s}{r}\Bigl) a_V(x,s)\,s^n\,ds\\
& \leq \frac1{4r^{n+1}}\int_0^\infty
\varphi\Bigl(\frac{s}{r}\Bigl) \gamma_V(x,s)\,s^n\,ds \lesssim \frac1{r^{n+1}}\int_0^{2r}
 \gamma_V(x,s)\,s^n\,ds
.
\end{align*}
By polar coordinates again, we obtain
\begin{align*}
\frac1{r^{n+1}}\int_0^{2r}
 \gamma_V(x,s)\,s^n\,ds & \leq \sum_{i=+,-}
\frac1{r^{n+1}}\int_0^{2r} \HH^{n}\big((V^i\triangle\, T_x(V^{i,c}))\cap S(x,s)\big)\,ds\\
& = \frac1{r^{n+1}} \sum_{i=+,-}
\HH^{n+1}\big((V^i\triangle\, T_x(V^{i,c}))\cap B(x,2r)\big)\lesssim\mfg_V(x,2r).
\end{align*}
Thus,
$$\lca_{V, \psi}^\pm(x,r)\lesssim \mfg_V(x,2r).$$
Integrating with respect to $\ds\dr  \, d \mu(x)$, the lemma follows.
\end{proof}
\vv

The next key step for the proof of Lemma \ref{l:BA-small} consists in estimating $\mfg_V(x,r)$ in terms 
of $\mfg_\Omega(x,r)$. To ease notation, we set
\begin{equation}
	\kappa = \kappa(\epsilon) = \epsilon^{1/(2n+4)}.
\end{equation}

\begin{lemma}\label{lemkey40}
There exists an absolute constant $C>1$ such that, for all $x\in E_0$ and $r\in[c\epsilon^{-1}\ell(x),2r_0]$,
$$ 
 \mfg_V(x,r) \lesssim \mfg_\Omega(x,2r) + \frac{\redeps}{r^{n+1}}\sum_{P\in \WW_\Gamma:P\subset B(x,Cr)} \ell(P)^{n+1}.
$$
\end{lemma}

\begin{proof}
Without loss of generality we assume that $x=0$, so that $T_{x}(A)=-A$ for any set $A$.
We define the function $\wt d:\R^{n+1}\to [0,\infty)$ given by
$$\wt d(y) = \max\big(d(y),d(-y)\big).$$
Recall that the function $d$ is defined in \eqref{e:d-smoothing}. Since this is $1$-Lipschitz, the same holds for $\wt d$.
Next we define the following family of cubes
$$\wt \WW= \{Q  \mbox{ maximal in } \DD_{\R^{n+1}} \, :\, \ell(Q)  < 20^{-1}\,\redeps  \inf_Q \wt d(y)\}.$$
Here $\DD_{\R^{n+1}}$ is the family of the usual dyadic cubes form $\R^{n+1}$, which we assume here to be closed.

The cubes from $\wt\WW$ satisfy properties analogous to the ones in Lemma \ref{l:whitney-dec}. In particular, it holds:
    \begin{enumerate}[label=\textup{(}\alph*\textup{)}]
        \item If $Q\in\wt\WW$ and $y \in 15Q$, then $\ell(Q) \approx \redeps \wt d(y)$.
        \item For each $Q\in\wt\WW$, there are at most $N$ cubes $R\in \wt\WW$ such that $15 Q \cap 15R \neq \varnothing$, where $N$ is some absolute constant.
        \item If $Q\in \wt\WW$, then $-Q\in \wt\WW$ too.
    \end{enumerate}
    The properties (a) and (b) follow by arguments analogous to the ones  in Lemma \ref{l:whitney-dec}. We leave the details for the reader. The property (c) is due to the fact that $\wt d(y)=\wt d(-y)$ for all $y\in\R^{n+1}$ and that $Q\in\DD_{\R^{n+1}}$ if and only if
$-Q\in\DD_{\R^{n+1}}$. 

To shorten notation, we write $B=B(0,r)$ (recall we are assuming $x=0$). Also, we let $\wt\WW_B$ be the family of the cubes from 
$\wt\WW$ which intersect $B$. 
We claim that if $Q\in\wt\WW_B$, then $Q\subset C\,B$ for some absolute constant $C>1$. Indeed, if $y\in Q\cap B$, then $\wt d(y) = d(y')$ either for $y'=y$ or for $y'=-y$, and taking into account that, by Lemma \ref{lemDd} and the fact that $r\in[c\epsilon^{-1}\ell(0),2r_0]$, $d(0)\approx \ell(0)\leq c\epsilon r$, we have
$$\ell(Q)\approx\redeps \,\wt d(y) = \redeps \,d(y') \leq \redeps\,d(0) + \redeps\,|y'| \leq 
2r.$$
Together with the condition that $Q\cap B\neq\varnothing$, this implies that $2Q\subset CB$. Remark that from the above estimates it follows also easily that $\ell(Q)\lesssim\kappa^2r_0$.

We say that a cube $Q\in\wt\WW_B$ is good, and we write $Q\in\wt\WW_g$
if either
\begin{itemize}
\item $Q\subset V^+$ and $-Q\subset V^-$, or
\item $Q\subset V^-$ and $-Q\subset V^+$.
\end{itemize}
We also set 
\begin{equation}\label{WW_b}
	\wt\WW_b= \wt\WW_B\setminus \wt\WW_g.
\end{equation}

For $i=+,-$, write 
$$\mfg_V^i(y,r)  = 
\frac1{r^{n+1}}\,\HH^{n+1}\big((V^i\triangle\, T_y(V^{i,c}))\cap B(y,r)\big),
$$
so that $\mfg_V(y,r) = \max_{i=+,-}\mfg_V^i(y,r)$. We also let $j=-i$ (i.e., $j=\mp$ for $i=\pm$).
Notice that if $Q\in\wt\WW_g$, then, modulo a set of Lebesgue measure zero,
$$(V^i\cap Q)\triangle\, (T_0(V^{i,c})\cap Q) = (V^i\cap Q)\triangle\, T_0(V^j\cap (-Q)) =
\varnothing.$$ 
Then, for $i=+,-$, we have
\begin{align*}
\mfg_V^i(x,r) &\leq \frac1{r^{n+1}}\sum_{Q\in \wt \WW_B} \HH^{n+1}\big((V^i\cap Q)\triangle\, (T_0(V^j)\cap Q)\big)\\
& = \frac1{r^{n+1}}\sum_{Q\in \wt \WW_b} \HH^{n+1}\big((V^i\cap Q)\triangle\, (T_0(V^j)\cap Q)\big).
\end{align*}
Next we need to distinguish different types of cubes from $Q\in \wt W_b$:
\begin{itemize}
\item[(1)] We write $Q\in\wt \WW_b^1$ if both $Q\cap \Gamma=\varnothing$ and $-Q\cap \Gamma=\varnothing$.

\item[(2)] We write $Q\in\wt \WW_b^2$ if either $Q\cap \Gamma\neq \varnothing$ or $-Q\cap \Gamma\neq \varnothing$, and moreover any cube $P\in\WW_\Gamma$ (recall that 
$\WW_\Gamma$ was defined in \eqref{defwgam} and that cubes from $\WW_\Gamma$ are subsets of $\Gamma$) which intersects either $2Q$ or $-2Q$
satisfies $\ell(P) \leq \blueps \ell(Q)$  (here $-2Q=-(2Q)$).

\item[(3)] We write $Q\in\wt \WW_b^3$ if neither (1) nor (2) hold. That is,
 either $Q\cap \Gamma\neq \varnothing$ or $-Q\cap \Gamma\neq \varnothing$, and moreover there exists some cube $P\in\WW_\Gamma$ which intersects either $2Q$ or $-2Q$
satisfying $\ell(P) > \blueps 
\ell(Q)$.
\end{itemize}
Then we split
\begin{align*}
\mfg_V^i(x,r) &\leq  \sum_{k=1}^3\frac1{r^{n+1}}\sum_{Q\in \wt \WW_b^k} \HH^{n+1}\big((V^i\cap Q)\triangle\, (T_0(V^j)\cap Q)\big) =: \sum_{k=1}^3 S_k.
\end{align*}

\noi {\bf Estimate of $S_1$.}
In this case $\pm Q\subset \R^{n+1}\setminus \Gamma$. Together with the condition that $Q\in\wt \WW_b$ (recall \eqref{WW_b}), this implies that either both $Q$ and $-Q$ are contained in $V^+$, or both $Q$ and $-Q$ are
contained in $V^-$.
Suppose, for example, that the first option holds (the arguments for the second one are analogous).
Since the center $x_Q$ of $Q$ satisfies
$$ d(x_Q) \leq \wt d(x_Q)\approx \redepsinv \ell(Q),$$
from the definition of $d(x_Q)$ it follows easily that there exists some ball $B(Q)\in \VG$ that contains $Q$ and satisfies
$\rad(B(Q))\lesssim \redepsinv\ell(Q)$. Similarly, there exists some ball $B(-Q)\in \VG$ that contains $-Q$ and satisfies
$\rad(B(-Q))\lesssim \redepsinv \ell(Q)=\redepsinv\ell(-Q)$. Since both $Q$ and $-Q$ are contained in $V^+$,
by Lemma \ref{l:Gtopology2} and recalling that $\kappa= \epsilon^{1/(2n+4)}$ we have
\begin{align}\label{eqalid72}
\HH^{n+1}(Q \setminus \Omega^+) & \leq 
\HH^{n+1}(B(Q) \cap V^+ \setminus \Omega^+) \lesssim \epsilon\,\HH^{n+1}(B(Q))\\
&  \lesssim \epsilon \,(\redepsinv)^{n+1} \,\HH^{n+1}(Q) = \redeps
 \,\HH^{n+1}(Q),\nonumber
\end{align}
and analogously,
$$\HH^{n+1}(-Q \setminus \Omega^+)  \lesssim 
\redeps\,\HH^{n+1}(-Q).$$
Hence, for $\epsilon$ small enough
$$\HH^{n+1}(Q \cap \Omega^+)\geq \frac34\,\HH^{n+1}(Q)\quad \mbox{ and }\quad
\HH^{n+1}(-Q \cap \Omega^+)\geq \frac34\,\HH^{n+1}(-Q).$$
Consequently,
\begin{align}\label{eqigusk42}
\HH^{n+1}\big((\Omega^+ \triangle \,T_0(\Omega^{+,c}))\cap Q) & =
\HH^{n+1}\big((\Omega^+\cap Q) \triangle \,(T_0(\Omega^{+,c}\cap -Q)\big)\\
&\geq \HH^{n+1}(\Omega^+\cap Q) - \HH^{n+1}\big(T_0(\Omega^{+,c}\cap -Q)\big)\nonumber\\
& = \HH^{n+1}(\Omega^+\cap Q) - \HH^{n+1}(\Omega^{+,c}\cap -Q)\nonumber\\
& \geq \frac34\,\HH^{n+1}(Q) - \frac14\,\HH^{n+1}(-Q) = \frac12\,\HH^{n+1}(Q).\nonumber
\end{align}
The same inequality holds with $\Omega^+$ interchanged by $\Omega^-$ in the case when
both $Q$ and $-Q$ are
contained in $V^-$.
Therefore,
$$S_1 \leq \frac1{r^{n+1}}\sum_{Q\in \wt \WW_b^1} \HH^{n+1}(Q)
\leq \frac2{r^{n+1}}\sum_{k=+,-}\sum_{Q\in \wt \WW_B} \HH^{n+1}\big((\Omega^k \triangle \,T_0(\Omega^{k,c}))\cap Q).$$

\vv
\noi {\bf Estimate of $S_2$.}
We claim that in this case one of the cubes $2Q$ or $-2Q$ does not intersect $\Gamma$. Indeed, either
$\wt d(x_Q) = d(x_Q)$ or $\wt d(x_Q) = d(-x_Q)$. Suppose the first option holds and suppose that $2Q\cap\Gamma\neq\varnothing$. Then take $y\in 2Q\cap\Gamma$, so that by Lemma \ref{lemDd} $d(y)\approx D(\Pi(y))$. Let $P\in\WW_\Gamma$ be such that
$y\in P$, so that by the definition of $\wt\WW_b^2$ it holds $\ell(P)\leq \blueps 
\ell(Q)$.
We deduce that
$$\wt d(x_Q) =d(x_Q) \leq d(y) + |y-x_Q| \lesssim D(\Pi(y)) + \ell(Q) \approx \ell(P) + \ell(Q) \lesssim \blueps 
\ell(Q),$$
which contradicts the fact that $\ell(Q) \approx 
 \redeps \wt d(x_Q)$, by the property (a) of the family $\wt \WW$. So our claim is proven.

So in this case, up to an interchange between $Q$ and $-Q$, we have $2Q\cap \Gamma=\varnothing$ and $-Q\cap \Gamma\neq\varnothing$. The last condition implies that
\begin{equation}\label{eqmesig4}
\HH^{n+1}(-2Q\cap V^+) \approx \HH^{n+1}(-2Q\cap V^-)\approx \ell(Q)^{n+1}.
\end{equation}
By almost the same arguments as above in the case of $S_1$, there are 
balls $B(Q),B(-Q)\in \VG$ that contain $2Q$ and $-2Q$ respectively and satisfy
$\rad(B(Q)),\rad(B(-Q))\lesssim \redepsinv\ell(Q)$.

Suppose that $2Q\subset \Omega^+$ (the arguments for $2Q\subset \Omega^-$ are analogous). Then, from
Lemma \ref{l:Gtopology2} applied to $B(Q)$, arguing as in \rf{eqalid72}, we derive
\begin{equation}\label{eqmesig5}
\HH^{n+1}(2Q \setminus \Omega^+) \lesssim  \redeps 
 \,\HH^{n+1}(2Q).
\end{equation}
We also deduce 
\begin{align*}
\HH^{n+1}(-2Q \cap V^+ \setminus \Omega^+) & \leq 
\HH^{n+1}(B(-Q) \cap V^+ \setminus \Omega^+) \lesssim \epsilon\,\HH^{n+1}(B(-Q))\\
&  \lesssim \epsilon 
 \, (\redepsinv)^{n+1} 
  \,\HH^{n+1}(-2Q) \approx \redeps 
  	\,\HH^{n+1}(-2Q\cap V^+),\nonumber
\end{align*}
using \rf{eqmesig4} in the last estimate. Thus, for $\epsilon$ small enough,
\begin{equation}\label{eqmesig6}
\HH^{n+1}(-2Q \cap  \Omega^+) \geq \HH^{n+1}(-2Q \cap V^+ \cap\Omega^+) \approx \HH^{n+1}(-2Q\cap V^+)
\approx \HH^{n+1}(-2Q).
\end{equation}
Now we argue as in \rf{eqigusk42}:
\begin{align*}
\HH^{n+1}\big((\Omega^+ \triangle \,T_0(\Omega^{+,c}))\cap 2Q) & =
\HH^{n+1}\big((\Omega^+\cap 2Q) \triangle \,(T_0(\Omega^{+,c}\cap -2Q)\big)\\
&\geq \HH^{n+1}(\Omega^+\cap 2Q) - \HH^{n+1}\big(T_0(\Omega^{+,c}\cap -2Q)\big)\\
& = \HH^{n+1}(\Omega^+\cap 2Q) - \HH^{n+1}(\Omega^{+,c}\cap -2Q).
\end{align*}
By \rf{eqmesig5} we have
$$\HH^{n+1}(\Omega^+\cap 2Q) \geq (1-C\redeps)
 \,\HH^{n+1}(2Q),$$
and by 
\rf{eqmesig6},
$$\HH^{n+1}(\Omega^{+,c}\cap -2Q) = \HH^{n+1}(-2Q) - \HH^{n+1}(-2Q \cap  \Omega^+) 
\leq (1-\cfour)\,\HH^{n+1}(-2Q)$$
for some $\cfour>0$.
Thus,
$$\HH^{n+1}\big((\Omega^+ \triangle \,T_0(\Omega^{+,c}))\cap 2Q)\geq 
(\cfour-C\redeps
)\,\HH^{n+1}(-2Q)\approx \HH^{n+1}(-2Q)$$
for $\epsilon$ small enough. The same inequality holds with $\Omega^+$ interchanged by $\Omega^-$ in the case when
$2Q\subset \Omega^-$.
Therefore,
$$S_2 \leq \frac1{r^{n+1}}\sum_{Q\in \wt \WW_b^2} \HH^{n+1}(Q)
\lesssim \frac1{r^{n+1}}\sum_{k=+,-}\sum_{Q\in \wt \WW_B} \HH^{n+1}\big((\Omega^k \triangle \,T_0(\Omega^{k,c}))\cap 2Q).$$

\vv
\noi {\bf Estimate of $S_3$.}
For $Q\in\wt \WW_b^3$, we write $Q\in\wt \WW_b^{3a}$ if there exists some cube $P\in\WW_\Gamma$ which intersects $2Q$
satisfying $\ell(P) > \blueps 
\ell(Q)$. Since $Q\in\wt \WW_b^3$ if and only if $-Q\in\wt \WW_b^3$, and in this case either 
$Q\in\wt \WW_b^{3a}$ or $-Q\in\wt \WW_b^{3a}$, it follows that
$$S_3 \leq \frac1{r^{n+1}}\sum_{Q\in \wt \WW_b^3} \HH^{n+1}(Q)\leq  \frac2{r^{n+1}}\sum_{Q\in \wt \WW_b^{3a}} \HH^{n+1}(Q).$$
Observe also that if $Q\in\wt \WW_b^{3a}$ and $P$ is as above, then $Q\subset \NN_{C \kappa
	\ell(P)} (2P)$ (recall that $2P$ is a subset of $\Gamma$).
Consequently,
$$\sum_{Q\in \wt \WW_b^{3a}} \HH^{n+1}(Q) \leq \sum_{P\in\WW_\Gamma:P\cap 2B\neq\varnothing}
\HH^{n+1}(\NN_{C\tcb{\kappa}\ell(P)} (2P)) \lesssim 
\tcb{\kappa} \sum_{P\in \WW_\Gamma:P\cap 2B\neq\varnothing}\ell(P)^{n+1}.$$
From the properties of the cubes from $\WW_\Gamma$, Lemma \ref{lemDd}, and the fact that $r\gtrsim\epsilon^{-1}\ell(0)\approx\epsilon^{-1}d(0)$, we deduce that if $y\in P\cap 2B\neq\varnothing$,
then
$$\ell(P)\approx D(\Pi(y))\approx d(y) \leq d(0) + |y| \leq d(0) + 2r\lesssim (\epsilon + 1)\,r.$$ 
Hence $P\subset CB$, for some absolute constant $C>1$. Thus,
$$S_3\lesssim \frac{\tcb{\kappa}
}{r^{n+1}}\sum_{P\in \WW_\Gamma:P\subset CB} \ell(P)^{n+1}.$$

\vv
Gathering the estimates obtained for the terms $S_1$, $S_2$, and $S_3$, and taking into account 
that the cubes $2Q$ with $Q\in\wt \WW_B$ have finite superposition (by the property (b) above) and they are contained in $2B$, we obtain
\begin{align*}
\mfg_V^i(0,r) &\lesssim \frac1{r^{n+1}}\sum_{k=+,-}\sum_{Q\in \wt \WW_B} \HH^{n+1}\big((\Omega^k \triangle \,T_0(\Omega^{k,c}))\cap 2Q) +  \frac{\tcb{\kappa}
}{r^{n+1}}\sum_{P\in \WW_\Gamma:P\subset CB} \ell(P)^{n+1} \\
& \lesssim 
\frac1{r^{n+1}}\sum_{k=+,-}\HH^{n+1}\big((\Omega^k \triangle \,T_0(\Omega^{k,c}))\cap 2B) +  \frac{\tcb{\kappa}
}{r^{n+1}}\sum_{P\in \WW_\Gamma:P\subset CB} \ell(P)^{n+1},
\end{align*}
which proves the lemma.
\end{proof}


\vv
\begin{lemma}\label{lemfi99}
We have
$$\int_{E_0}\int_{\min(c\epsilon^{-1}\ell(x),r_0)}^{2r_0} \mfg_V(x,r)^2 \, \dr  \, d \mu(x)
\lesssim
\int_{E_0}\int_{\min(c\epsilon^{-1}\ell(x),r_0)}^{4r_0} \mfg_\Omega(x,r)^2 \, \dr  \, d \mu(x) + \redeps\,r_0^n
.$$
\end{lemma}

\begin{proof}
By Lemma \ref{lemkey40}, we have
\begin{align}\label{eqal8934}
\int_{E_0}\int_{\min(c\epsilon^{-1}\ell(x),r_0)}^{2r_0} \mfg_V&(x,r)^2 \, \dr  \, d \mu(x) \\
 & \lesssim \int_{E_0}\int_{\min(c\epsilon^{-1}\ell(x),r_0)}^{2r_0} \mfg_\Omega(x,2r)^2\, \dr  \, d \mu(x)
 \nonumber\\
&\quad  + \int_{E_0}\int_{\min(c\epsilon^{-1}\ell(x),r_0)}^{2r_0}\bigg(
  \frac{\kappa}{r^{n+1}}\sum_{P\in \WW_\Gamma:P\subset B(x,Cr)} \ell(P)^{n+1}\bigg)^2\, \dr  \, d \mu(x).\nonumber
\end{align}
So we just have to show that the last summand is bounded above by $\redeps r_0^n = \epsilon^{1/(n+2)}r_0^n$.

By Cauchy-Schwarz and the fact that $\ell(P)^{n}\approx \HH^n(P)$, we get
\begin{align*}
\bigg(
  \sum_{P\in \WW_\Gamma:P\subset B(x,Cr)} \frac{\ell(P)^{n+1}}{r^{n+1}}\bigg)^2
  & \lesssim 
\bigg(\sum_{P\in \WW_\Gamma:P\subset B(x,Cr)} \frac{\ell(P)^{n+2}}{r^{n+2}}\bigg)
\bigg(\sum_{P\in \WW_\Gamma:P\subset B(x,Cr)} \frac{\ell(P)^{n}}{r^{n}}\bigg)\\
 & \lesssim \sum_{P\in \WW_\Gamma:P\subset B(x,Cr)} \frac{\ell(P)^{n+2}}{r^{n+2}}.
\end{align*}
Therefore,
\begin{align}\label{eqllarg84}
\int_{E_0}\int_{\min(c\epsilon^{-1}\ell(x),r_0)}^{2r_0}\Bigg( 
  \frac{\kappa}{r^{n+1}}&\sum_{\substack{P\in \WW_\Gamma:\\P\subset B(x,Cr)}} \ell(P)^{n+1}\Bigg)^2\, \dr  \, d \mu(x)\\
&  \lesssim \redeps 
  \int_{E_0}\int_{\min(c\epsilon^{-1}\ell(x),r_0)}^{2r_0}\sum_{\substack{P\in \WW_\Gamma:\\P\subset B(x,Cr)}} \frac{\ell(P)^{n+2}}{r^{n+2}}\, \dr  \, d \mu(x).\nonumber
\end{align}
Since all the cubes $P$ in the last sum are contained in $CB_0$ for some $C>1$, for each $x\in {E_0}$
we have, by Fubini,
$$\int_{\min(c\epsilon^{-1}\ell(x),r_0)}^{2r_0}\sum_{\substack{P\in \WW_\Gamma:\\P\subset B(x,Cr)}} \frac{\ell(P)^{n+2}}{r^{n+2}}\, \dr \leq
\sum_{\substack{P\in \WW_\Gamma:\\P\subset CB_0}}
\int_{r>0: P\subset B(x,Cr)}\frac{\ell(P)^{n+2}}{r^{n+2}}\, \dr.
$$
Notice that the condition that $P\subset B(x,Cr)$ implies that $r\gtrsim \ell(P) + \dist(x,P)=:D(x,P)$.
Thus,
$$\sum_{\substack{P\in \WW_\Gamma:\\P\subset CB_0}}
\int_{r>0: P\subset B(x,Cr)}\frac{\ell(P)^{n+2}}{r^{n+2}}\, \dr \lesssim
\sum_{\substack{P\in \WW_\Gamma:\\P\subset CB_0}} \int_{r\gtrsim D(x,P)}\frac{\ell(P)^{n+2}}{r^{n+2}}\, \dr \lesssim
\sum_{\substack{P\in \WW_\Gamma:\\P\subset CB_0}} \frac{\ell(P)^{n+2}}{D(x,P)^{n+2}}.$$
Plugging this estimate in \rf{eqllarg84}, we obtain
$$\int_{E_0}\int_{\min(c\epsilon^{-1}\ell(x),r_0)}^{2r_0}\Bigg( 
  \frac{\kappa}{r^{n+1}}\sum_{\substack{P\in \WW_\Gamma:\\P\subset B(x,Cr)}} \ell(P)^{n+1}\Bigg)^2\, \dr  \, d \mu(x)\lesssim 
\redeps
  \int_{E_0} \sum_{\substack{P\in \WW_\Gamma:\\P\subset CB_0}} \frac{\ell(P)^{n+2}}{D(x,P)^{n+2}}\, d \mu(x).
 $$
Using Fubini and the polynomial growth of  degree $n$  of $\mu$, we easily obtain
$$
  \int_{E_0} \sum_{\substack{P\in \WW_\Gamma:\\P\subset CB_0}} \frac{\ell(P)^{n+2}}{D(x,P)^{n+2}}\, d \mu(x) = \sum_{\substack{P\in \WW_\Gamma:\\P\subset CB_0}}\int_{E_0} \frac{\ell(P)^{n+2}}{D(x,P)^{n+2}}\, d \mu(x) \lesssim \sum_{\substack{P\in \WW_\Gamma:\\P\subset CB_0}} \ell(P)^{n} \approx r_0^n.$$
Putting together the last estimates and \rf{eqal8934}, the lemma follows.
\end{proof}

\vv
\begin{proof}[\bf Proof of Lemma \ref{l:BA-small}]
By \rf{e:BA-100}, Lemma \ref{l:lemfac64}, and Lemma \ref{lemllumy}, we have
\begin{align*}
 \alpha^2   \mu(\BA) &\lesssim    \|\wt \A_{V, \psi} \|_{L^2(\Gamma)}^2 + C(\theta)\epsilon^2 \,r_0^n\\
    & \lesssim \int_{E_0}\int_{\min(c\epsilon^{-1}\ell(x),r_0)}^{r_0} \lca_{V, \psi}(x,r)^2 \, \dr  \, d \mu(x)
    + C(\theta)\epsilon^2\,r_0^n + C(\theta) \epsilon^2 |\log\epsilon| r_0^n.
 \end{align*}
On the other hand, by Lemma \ref{lem9.11}, Lemma \ref{lemfi99}, and the assumption (c) of Main Lemma 
\ref{l:main2},
\begin{align*} 
\int_{E_0}\int_{\min(c\epsilon^{-1}\ell(x),r_0)}^{r_0} \lca_{V, \psi}(x,r)^2 \, \dr  \, d \mu(x) &
\lesssim 
\int_{E_0}\int_{\min(c\epsilon^{-1}\ell(x),r_0)}^{2r_0} \mfg_V(x,r)^2 \, \dr  \, d \mu(x)\\
& \lesssim
\int_{E_0}\int_{\min(c\epsilon^{-1}\ell(x),r_0)}^{4r_0} \mfg_\Omega(x,r)^2 \, \dr  \, d \mu(x) + \epsilon^{1/(n+2)}\,r_0^n\\
& \lesssim \epsilon\,\mu({E_0}) + \epsilon^{1/(n+2)}\,r_0^n \lesssim \epsilon^{1/(n+2)}\,r_0^n.
\end{align*}
Putting the estimates above together, we obtain
$$ \mu(\BA) \leq C(\theta)\,\alpha^2 \,\epsilon^{1/(n+2)}\,r_0^n\leq \epsilon^{1/(n+3)}\,\mu({E_0}),$$
assuming $\epsilon$ small enough, depending on $\alpha$ and $\theta$.
\end{proof}
\vv


\subsection{Proof of the Main Lemma \ref{l:main2}}

By Lemmas \ref{l:LD-small} and \ref{l:BA-small}, if $\theta$ is chosen small enough and then $\epsilon$ (and thus $\kappa$) also
small enough (depending on $\alpha$ and $\theta$), then we have
$$\mu(\BA \cup \LD)\leq \frac1{100}\,\mu(E_0).$$
Consequently,
$$\mu(Z)\geq \frac{99}{100}\mu(E_0),
$$
and $Z\subset Z_0\subset \Gamma$.  This completes the proof of the Lemma \ref{l:main2}.


\part{True tangents}
In this (small) Part we assume that $\Omega^+$ and $\Omega^-$ are disjoint open sets. We will show that, if in addition to a Dini-type condition on the square function $\ve(x,r) = \ve_n(x,r)$ we assume a Dini condition on $\ve_{s}(x,r)^2$ on a set $E\subset F=\R^{n+1} \setminus (\Omega^+ \cup \Omega^-)$ of positive $\HH^n$ measure, where $s \in (0, n-1)$ is not necessarily integer, and moreover we assume that $F$ is $s$-content regular, then not only $E$ is $n$-rectifiable, but it also admits true tangents.
\vv


\section{Spherical slicing with respect to points in a Lipschitz graph}\label{sec-slicing}

In this section we prove some slicing results involving the capacities $\capp_{s}$ defined in Section~\ref{seccap}.


\begin{propo}\label{lem-slicing}
Let $B(0,r_0)\subset \R^{n+1}$ and let $\Gamma$ be a Lipschitz graph through the origin with slope at most $\tau$ (with respect to $\R^n\times \{0\}$). Let $B\subset B(0,r_0)$ be a ball with $\rad(B)\leq \frac1{10}\,r_0$ such that 
$\dist(B,\Gamma)\geq 100\,\tau \,r_0$.
 Let $F\subset B$  and  $G\subset \Gamma$ both be  compact sets. 
	Then, for any $s>1$,
	$$\capp_{s}(F)\,\frac{\HH^n(G)^2}{r_0^n} \leq C(\tau) \int_{G}\int_0^\infty \capp_{s-1}(F\cap S(z,r))\,dr\,d\HH^n(z).$$
\end{propo}
\vvv

To prove this proposition we will use the following auxiliary result.

\begin{lemma}\label{lemFaux}
	Under the assumptions of Proposition \ref{lem-slicing}, for all $x,y\in F$ and all $\delta>0$, we have
	$$\HH^n\big(\{z\in G:||x-z|-|y-z||\leq\delta\}\big)\lesssim_\tau \frac{\delta\,r_0^n}{|x-y|}.$$
\end{lemma}

\begin{proof}
	Clearly, to prove the lemma we may assume that $\delta\leq \frac12\,|x-y|$.
	
	Let $L$ be the hyperplane orthogonal to the segment $[x,y]$ passing through the midpoint of that segment, so that $y$ is the point symmetric to $x$ with respect to $L$.
	Let $z\in G$ be such that $||x-z|-|y-z||\leq\delta$. We claim that this condition implies that
	\begin{equation}\label{eqcond99}
		\dist(z,L) \lesssim_\tau \frac{\delta\,r_0}{|x-y|}.
	\end{equation}
	To prove this, we assume that $z$ is at the same side of $L$ as $x$ and we
	let $L_z$ be the plane parallel to $L$ passing through $z$. We distinguish two cases. Suppose first that 
	\begin{equation}\label{eqcond100}
		\dist(z,L)\geq \dist(x,L)=\frac12\,|x-y|.
	\end{equation}
	Let $z'$ be the orthogonal projection of $z$ on the line $L_{x,y}$ passing  through $x,y$, so that the points $z',x,y$ are aligned.
	The condition \rf{eqcond100} ensures that $\dist(z',L)= \dist(z,L) \geq \dist(x,L)$. So $x$ is contained in the segment $[z',y]$ (because $z$ is at the same side of $L$ as $x$).
	Then, denoting $h = \dist(z,L_{x,y}) = |z-z'|$, by Pithagoras,
	\begin{align*}
		\delta\geq |y-z| -|x-z| & = \sqrt{h^2 + |y-z'|^2} - \sqrt{h^2 + |x-z'|^2} \\ 
		&= \frac{|y-z'|^2 - |x-z'|^2}{|y-z| +|x-z|}\approx_\tau
		\frac{|y-x|\,(|y-z'| + |x-z'|)}{r_0} \gtrsim_\tau \frac{|y-x|\,\dist(z,L)}{r_0},
	\end{align*}
	which implies \rf{eqcond99}.
	
	Suppose now that \rf{eqcond100} does not hold. That is, $\dist(z,L)< \dist(x,L)$. In this case, 
	the point $z'$, the orthogonal projection of $z$ on $L_{x,y}$, is contained in the segment $[x,y]$. Let $x'$ be the point symmetric to $x$ with respect to $L_z$ 
	(here a drawing would help). From the fact that $z$ is at the same side of $L$ as $x$, it follows that
	$x'\in [x,y]$.
	Obviously, $|x-z|=|x'-z|$, and also  
	$|x'-y| = 2\,\dist(z,L)$. Then, by Pithagoras,
	\begin{align*}
		\delta\geq |y-z| -|x-z| & = |y-z| -|x'-z|
		=\sqrt{h^2 + |y-z'|^2} - \sqrt{h^2 + |x'-z'|^2} \\ 
		&= \frac{|y-z'|^2 - |x'-z'|^2}{|y-z| +|x'-z|}\approx_\tau
		\frac{|y-x'|\,(|y-z'| + |x'-z'|)}{r_0} \geq_\tau \frac{2\,\dist(z,L)\,|x-y|}{r_0},
	\end{align*}
	where in the last inequality we took into account that $|y-z'| + |x'-z'|= |y-z'| + |x-z'| = |x-y|$.
	So \rf{eqcond99} also holds in this case and our claim is proved.
	\vv
	
	Because of the claim above, the set
	$$G_{x,y}:= \{z\in G:||x-z|-|y-z||\leq\delta\}$$
	is contained in $\NN_{\delta_0}(L)$, the $\delta_0$ neighborhood of the above hyperplane $L$, with $\delta_0= C(\tau)\,\frac{\delta\,r_0}{|x-y|}$.
	Observe that the mid point of $[x,y]$, which we denote by $p$, belongs to $B$, since both $x,y\in B$.
	Recalling that $\dist(B,\Gamma)\geq 100\tau r_0$ and that the slope of $\Gamma$ is at most $\tau$
it follows that there is some ``transversality" between $\Gamma$ and $L$ depending on $\tau$. More precisely,
the angle between $L$ and any line intersecting $\Gamma\cap B(0,r_0)$ in at least two points is bounded away from $0$
with some bound depending on $\tau$.
	It is easy to check that this implies that 
	$$\HH^n(\NN_{\delta_0}(L) \cap \Gamma) \leq C(\tau) \delta_0\,r_0^{n-1}\approx_\tau \frac{\delta\,r_0^n}{|x-y|}.$$
	Since $G_{x,y}$ is contained in $\NN_{\delta_0}(L)$, the lemma follows.
\end{proof}
\vv

For $z\in G$ and $r>0$, we consider the measures $\mu_{z,r}$ defined by
$$\int \vphi\,d\mu_{z,r} = \lim_{\delta\to0} \int_{A(z,r-\delta,r+\delta)} \vphi\,d\mu$$
for a.e. $r>0$. 
See also Chapter 10 from \cite{Mattila-gmt} for more details in the case of slicing by planes.
\vv

\begin{lemma}\label{lemabscont}
	Under the assumption of Lemma \ref{lem-slicing}, for any Borel set $A\subset [0,\infty)$, all $\vphi\in C(\R^n)$, and for $\HH^n$-a.e.\ $z\in F$, the sliced mesaures
	$\mu_{z,r}$ satisfy
	$$\int_A \int \vphi\,d\mu_{z,r}\,dr \leq \int_{x:|x-z|\in A}\vphi(x)\,d\mu(x),$$ 
	with equality if $I_1(\mu)<\infty$.
\end{lemma}

\begin{proof}
	Consider the spherical projection $P_z:\R^n\to[0,\infty)$ given by $P_z(x) = |x-z|$.
	The lemma follows arguing as in (10.6) from \cite{Mattila-gmt}. To prove the equality when  $I_1(\mu)<\infty$, it is enough to show that the measure
	$P_{z\#} \mu$ is absolutely continuous with respect to Lebesgue measure in $[0,+\infty)$ for 
	$\HH^n$-a.e.\ $z\in F$. 
	
	Denote
	$$\Theta_*^1(t,P_{z\#} \mu) = \liminf_{\delta\to0} \frac{P_{z\#} \mu(B(t,\delta))}{2\delta}.$$
	By Theorem 2.12 from \cite{Mattila-gmt}, to show that $P_{z\#} \mu$ is absolutely continuous with respect to Lebesgue measure in $[0,\infty)$, it is enough to show that $\Theta_*^1(t,P_{z\#} \mu)< \infty$
	for $P_{z\#} \mu$-a.e.\ $t\in [0,+\infty)$. To this end, using Fatou's lemma, we write
	$$\int_G\int_0^\infty \Theta_*^1(t,P_{z\#} \mu)\,dP_{z\#} \mu(t)\,d\HH^n(z)
	\leq \liminf_{\delta\to0} \frac1{2\delta} \int_G \int_0^\infty P_{z\#} \mu(B(t,\delta))\,dP_{z\#} \mu(t)\,d\HH^n(z).$$
	Notice that
	$$P_{z\#} \mu(B(t,\delta)) = \mu(A(z,t-\delta,t+\delta)) = \mu\big(\big\{x\in\R^{n+1}:\big||x-z|-t\big|\leq\delta\big\}\big).$$
	Thus,
	$$ \int P_{z\#} \mu(B(t,\delta))\,dP_{z\#} \mu(t) = 
	\int \mu\big(\big\{x\in\R^{n+1}:\big||x-z|-|y-z|\big|\leq\delta\big\}\big)\, d\mu(y).$$
	Then, by Fubini and Lemma \ref{lemFaux},
	\begin{align*}
		\int_G\int_0^\infty  \Theta_*^1(t,P_{z\#} \mu)&\,dP_{z\#} \mu(t)\,d\HH^n(z)\\
		& \leq \liminf_{\delta\to0} \frac1{2\delta} \iint \HH^n\big(\big\{z\in F:\big||x-z|-|y-z|\big|\leq\delta\big\}\big)\, d\mu(y)\,d\mu(x)\\
		&\lesssim_\tau\iint \frac{r_0^n}{|x-y|}\, d\mu(y)\,d\mu(x) \approx_\tau r_0^n\,I_1(\mu)<\infty,
	\end{align*}
	which implies that $\Theta_*^1(t,P_{z\#} \mu)< \infty$
	for $P_{z\#} \mu$-a.e.\ $t\in [0,+\infty)$, as wished.	
\end{proof}
\vv

\begin{lemma}\label{lem-slicing2}
	Let $F,B(0,r_0),\Gamma,G$ be as in Lemma \ref{lem-slicing}. Let $\mu$ be a measure supported on $F$.
	Then we have
	$$\int_G\int_0^\infty I_{s-1}(\mu_{z,r})\,dr\,d\HH^n(z) \lesssim_\tau r_0^n\,I_{s}(\mu),$$
	where $\mu_{z,r}$ is the sliced measure $\mu$ on the sphere $S(z,r)$. 
\end{lemma}

\vv
\begin{proof}
	By Fubini and Fatou's lemma, we have
	\begin{align*}
		I: & = \int_G\int_0^\infty I_{s-1}(\mu_{z,r})\,dr\,d\HH^n(z)\\
		& = \int_G\int_0^\infty \int_{S(z,r)}\int_{S(z,r)} \frac1{|x-y|^{s-1}}\,d\mu_{z,r}(x)\,d\mu_{z,r}(y)\,dr\,d\HH^n(z)\\
		&= \int_G\int_0^\infty \int_{S(z,r)}\lim_{\delta\to0}\frac1{2\delta}\int_{A(z,r-\delta,r+\delta)} \frac1{|x-y|^{s-1}}\,d\mu(x)\,d\mu_{z,r}(y)\,dr\,d\HH^n(z)\\
		& \leq \liminf_{\delta\to0}\frac1{2\delta}\int_G\int_0^\infty \int_{S(z,r)}\int_{A(z,r-\delta,r+\delta)} \frac1{|x-y|^{s-1}}\,d\mu(x)\,d\mu_{z,r}(y)\,dr\,d\HH^n(z).
	\end{align*}
	Notice that the condition $x\in A(z,r-\delta,r+\delta)$ is equivalent to saying $||z-x|-r|\leq\delta$, and for  $y\in S(z,r)$ this implies $||z-x|-|y-z||\leq\delta$. So by Fubini, for each $z\in G$,
	\begin{align*}
		\int_{S(z,r)}\int_{A(z,r-\delta,r+\delta)} \frac1{|x-y|^{s-1}}\,d\mu(x)\,d\mu_{z,r}(y)\,dr  & =
		\int_F \int_{||z-x|-r|\leq\delta} \int_{y\in S(z,r)} \frac1{|x-y|^{s-1}}\,d\mu_{z,r}(y)\,dr\,d\mu(x)\\ 
		&\leq \int_F \int_{||x-z|-|y-z||\leq\delta} \frac1{|x-y|^{s-1}}\,d\mu(y)\,d\mu(x).
	\end{align*}
	Thus, by Fubini,
	\begin{align*}
		I & \leq \liminf_{\delta\to0}\frac1{2\delta}\int_G \int_F \int_{||x-z|-|y-z||\leq\delta} \frac1{|x-y|^{s-1}}\,d\mu(y)\,d\mu(x)
		\,d\HH^n(z)\\
		& = \liminf_{\delta\to0}\frac1{2\delta} \int_F \int_F \frac1{|x-y|^{s-1}}\,\HH^n\big(\{z\in G:||x-z|-|y-z||\leq\delta\}\big)\, d\mu(y)\,d\mu(x).
	\end{align*}
	By Lemma \ref{lemFaux} we know that
	$$\HH^n\big(\{z\in G:||x-z|-|y-z||\leq\delta\}\big)\lesssim_\tau \frac{\delta\,r_0^n}{|x-y|}.$$
	Then we deduce that
	$$I\lesssim_\tau r_0^n\,\int_F \int_F \frac1{|x-y|^{s}}\, d\mu(y)\,d\mu(x) \approx_\tau r_0^n\,I_{s}(\mu),$$
	which proves the lemma.
\end{proof}

\vv

\begin{proof}[Proof of Proposition \ref{lem-slicing}]
	Let $\mu$ be a probability measure supported on $F$ such that $\capp_{s}(F)=I_{s}(\mu)^{-1}$. 
	Obviously, we may assume that $\capp_{s}(F)>0$ and $I_{s}(\mu)<\infty$, which implies that $I_1(\mu)<\infty$.
	For each $z\in G$, $r>0$, 
	consider the measure $\nu_{z,r} = \|\mu_{z,r}\|^{-1}\,\mu_{z,r}$. For $\HH^n$-a.e.\ $z\in G$, by
	Lemma \ref{lemabscont} and Cauchy-Schwarz,
	\begin{align*}
		1 & = \left(\int \mu_{z,r}(F)\,dr\right)^2 = \left(\int \mu_{z,r}(F)\,I_{s-1}(\nu_{z,r})^{1/2}\, I_{s-1}(\nu_{z,r})^{-1/2}\,dr\right)^2\\
		& \leq \left(\int \mu_{z,r}(F)^2\,I_{s-1}(\nu_{z,r})\, dr\right)\left(\int I_{s-1}(\nu_{z,r})^{-1}\,dr\right)\\
		& \leq \left(\int I_{s-1}(\mu_{z,r})\, dr\right)\left(\int \capp_{s-1}(F\cap S(z,r))\,dr\right).
	\end{align*}
	Thus, $\left(\int \capp_{s-1}(F\cap S(z,r))\,dr\right)^{-1}\leq \int I_{s-1}(\mu_{z,r})\, dr$, and by Lemma \ref{lem-slicing2},
	\begin{align*}
		\int_G \left(\int \capp_{s-1}(F\cap S(z,r))\,dr\right)^{-1}\,d\HH^n(z)  & \leq \int_G
		\int I_{s-1}(\mu_{z,r})\, dr\,d\HH^n(z)\\
		& \lesssim_\tau r_0^n\,I_{s}(\mu) \approx_\tau r_0^n\,\capp_{s}(F)^{-1}.
	\end{align*}
	Then, again by Cauchy-Schwarz,
	\begin{align*}
		\HH^n(G) & \leq \int_G \left(\int \capp_{s-1}(F\cap S(z,r))\,dr\right)^{-1/2}\left(\int \capp_{s-1}(F\cap S(z,r))\,dr\right)^{1/2}\,d\HH^n(z)\\
		&  \leq \bigg( \int_G \left(\int \capp_{s-1}(F\cap S(z,r))\,dr\right)^{-1}\!d\HH^n(z)\bigg)^{1/2}
		\left(\int_G \int\! \capp_{s-1}(F\cap S(z,r))\,dr\,d\HH^n(z)\right)^{1/2}\\
		& \lesssim_\tau  \big(r_0^n\,\capp_{s}(F)\big)^{-1/2}\left(\int_G \int \capp_{s-1}(F\cap S(z,r))\,dr\,d\HH^n(z)\right)^{1/2},
	\end{align*}
	which proves the proposition.
\end{proof}
\vv


\section{Weak flatness for $F$}\label{sec-mainlemma3}

In this section we prove Main Lemma \ref{mainlemma3}, which we restate for the reader's sake. 
\begin{lemma}\label{lem-small-betaF}
	Let $\Omega^+$, $\Omega^- \subset \R^{n+1}$ be open and disjoint, and $F=\R^{n+1} \setminus \Omega^+ \cup \Omega^-$. Fix $c_0,c_0'\in (0,1)$. For any $\tau>0$, there exists $\delta>0$ small enough such that the following holds. Let 
	$\Gamma\subset\R^{n+1}$ be a Lipschitz graph with slope at most $\delta$, 
	 let  $B_0=B(x_0, r_0)$ be a ball centered in $\Gamma$, and let $E\subset \Gamma\cap B_0$ be such that $\HH^n(E \cap \tfrac14 B_0) \geq c_0 r_0^n$ and $E \cap (B_0 \setminus \tfrac{4}{5}B_0) \neq \varnothing$.
	If
	$\Omega^+\cup\Omega^-$ satisfies the $s$-CDC for some $s \in (n-1, n)$  and
		\begin{equation}\label{e:lem17.1ve}
			\int_0^{1000 \rad (B_0)} \big(\ve_n(x,r)^2 + \mathfrak{a}^\pm(x,r)^2 + \ve_{s-1}(x,r)^2\big)\,\frac{dr}r < \delta \,\, \mbox{ for all } \, x \in E,
		\end{equation}
		where $a, \mathfrak{c}_0$ in the definition of $\ve_s$ are chosen appropriately depending on $\tau$,
		then
	\begin{equation}\label{small-beta-F'}
		\beta_{\infty, F}(\tfrac14B_0) < \tau.
	\end{equation}
\end{lemma}

Remark that the constant $\delta$ in the preceding lemma depends on $\tau$, $c_0$, $c_0'$, and also on $s$ and the parameters involved in the
$s$-CDC condition and in the definition of the coefficients $\ve_s(x,r)$.

\vv
\begin{rem}\label{rem-mainlemma3}
	For the the proof of the Lemma \ref{lem-small-betaF}, we take $\epsilon,\gamma>0$ with $\epsilon,\gamma\ll\tau$ and then we assume $\delta$ small enough so that the conclusion of Main Lemma \ref{mainlemma1} holds.  In particular, 
	this ensures that $\beta_{\infty,E}(B_0)\leq \epsilon$, and,
	if we denote $L_0=L_{B_0}$ and we let $D^\pm_0$ be the two components of $\tfrac12 B_0\setminus L_{B_0}$, we have	
	\begin{equation}\label{e-780}
	\HH^{n+1}( D^\pm_0 \setminus \Omega^\pm) < \gamma \, \HH^{n+1}( D^\pm_0).
\end{equation}
Note that the condition $E \cap B_0 \setminus \tfrac{4}{5}B_0 \neq \varnothing$ is needed to apply Main Lemma \ref{mainlemma1}.
\end{rem}
\vv

We start with an easy technical lemma. Before stating it, we introduce some notation. Let $z \in \R^{n+1}$ and let $L_z$ be an affine $n$-plane containing $z$; put $\wt L_z=L_z-z$. We define the spherical projection onto $L_z$ with center $x$ by 
\begin{equation*}
	\Pi_{z, L_z}(y):= \frac{|z-y|}{|\Pi_{\wt L_z}(z-y)|}\Pi_{\wt L_z}(z-y) + z,
\end{equation*}
where $\Pi_{\wt L_z}$ is the orthogonal projection on $\wt L_z$.
When $z, L_z$ are understood we will simply write $\Pi_\circ(y)$.

\begin{lemma}\label{auxlem-part3-1}
	With notation and assumptions of Lemma \ref{lem-small-betaF}, if $0< \epsilon < \gamma$ and $\gamma$ is sufficiently small, then the following holds. Let $L_0$ be an $n$-plane minimizing $\beta_{\infty,E}(B_0)$, let
	$z\in E\cap \frac14B_0$, let $L_z$ be the $n$-plane parallel to $L_0$ through $z$, and
	denote by $\Pi_\circ(y)=\Pi_{z, L_z}(y)$ the spherical projection defined above.
	Let $B$ be any ball satisfying
	\begin{enumerate}
		\item $B$ is centered on $L_z\cap\tfrac14B_0$ and $B \subset \tfrac12 B_0$;
		\item $\rad(B)=\gamma^{1/(n+2)}r_0$; 
		\item $B \subset X_z$, where, denoting by $u_{L_{z}}$ the unit vector perpendicular to $L_{z}$, 
		\begin{equation*}
			X_z= \{ x \in \R^{n+1} \, :\, |(x-z)\cdot u_{L_z}| \leq \tfrac34 |x-z| \}.
		\end{equation*}
	\end{enumerate}  
 Denote by $D_B^\pm$ be the two components of $B\setminus \NN_{2\epsilon r_0}(L_z)$, so that $D_B^\pm\subset 
D^\pm_0\cap B$. 
 Then we have
	\begin{equation}\label{e:60}
		\HH^{n+1}(\{y \in \Omega^+\cap D_B^+ :\,  \Pi_{\circ}^{-1}(\Pi_{\circ}(y)) \cap \Omega^{-} \cap D_B^- \neq \varnothing\}) \geq (1-\Cone\gamma^{1/(n+2)})\,\HH^{n+1}(D_B^+),
	\end{equation}
	for some constant $\Cone$ only depending on $n$.
\end{lemma}

If $\gamma$ is sufficiently small such balls can be found. Note also that condition (3) implies that the set on the right hand side of \eqref{e:60} is well-defined, and moreover $\mathrm{Lip}(\Pi_\circ|_B) \lesssim 1$.

\begin{proof}
	Without loss of generality, we can assume $z=0$ and $L_z=\R^n \subset \R^{n+1}$. 
	Note that $B\cap L_0\neq \varnothing$ because $B$ is centered in $L_{z} \cap \frac14B_0$ and 
	$$\beta_{\infty,E}(B_0) r_0 < \epsilon r_0 \ll \gamma^{1/(n+2)}r_0 = \rad(B)$$
	 since $\epsilon \leq \gamma$. We also have that
	$\HH^{n+1}(D_B^\pm) \approx \HH^{n+1}(B)\approx \gamma^{(n+1)/(n+2)} \HH^{n+1} (B_0)$.
	Using \eqref{e-780} and the fact that $B \subset \tfrac12 B_0$, and thus $D^\pm_B \subset D^\pm_0$, we obtain
	\begin{align*}
		\HH^{n+1}(D^+_B\setminus \Omega^+)& \leq \HH^{n+1}(D^+_0 \setminus \Omega^+) \leq \gamma \HH^{n+1}(D^+_0) \\
		&= \gamma \, \frac{\HH^{n+1}(D^+_0)}{\HH^{n+1}(D^+_B)}\, \HH^{n+1}(D^+_B) \leq \Cone \,\gamma^{1/(n+2)} \HH^{n+1} (D^+_B),
	\end{align*} 
for some constant $\Cone=\Cone(n)$.
Clearly the same holds for $\HH^{n+1}(D^-_B\setminus \Omega^-)$. 
 Now set
	\begin{align}\label{e:A1A2}
		A_{1} := \{ y \in \Omega^+\cap D_B^+ :\, \Pi_\circ^{-1}(\{\Pi_\circ(y)\}) \cap \Omega^{-}\cap B \neq \varnothing\} \,\, \mbox{ and } \,\, A_{2} := \Omega^+\cap D_B^+ \setminus A_{1}.
	\end{align}
Consider the reflection $\psi$ with respect to $L_z$ defined by $\psi(y) = (y_1,...,y_n,-y_{n+1}).$
As $\psi(y)\in \Pi_\circ^{-1}(\{\Pi_\circ(y)\})$, it follows that 
$$y\in A_2 \;\Rightarrow\; \psi(y)\in (\Omega^+\cup F) \cap D_B^-.$$
Consequently,
\begin{align*}
\HH^{n+1}(A_2) & = \HH^{n+1}(\psi(A_2)) \leq \HH^{n+1}((\Omega^+\cup F) \cap D_B^-)\\
& = 
\HH^{n+1}(D_B^-\setminus \Omega^-)\leq \Cone \gamma^{1/(n+2)}\,\HH^{n+1}(D_B^-) = \Cone \gamma^{1/(n+2)}\,\HH^{n+1}(D_B^+),
\end{align*}
which is equivalent to \rf{e:60}.
\end{proof}
\vv

An immediate consequence of the previous lemma is that $\Pi_\circ(F\cap B)$ is large.

\begin{lemma}\label{auxlemma-part3-2}
	Notation and assumptions as in Lemma \ref{auxlem-part3-1}. Then, for any $B$ satisfying (1)-(3) there, we have
	\begin{equation*}
		\HH^n(\Pi_\circ(F \cap B)) \geq (1-\Ctwo\gamma^{1/(n+2)})\, \HH^n(B\cap L_0),
	\end{equation*}
	where $\Ctwo=\Ctwo(n)$.
\end{lemma}

\begin{proof}
	Let $D_B^\circ:=\Pi_\circ(D_B^+)=\Pi_\circ(D_B^-)$ and $F_B=\Pi_\circ(F \cap B)$.
	Suppose that $\HH^n(D_B^\circ \setminus F_B) > \Cthree \gamma^{1/(n+2)} \HH^{n}(D_B^\circ)$ for some constant $\Cthree$ to be fixed below. Set 
	\begin{align*}
		A^\pm := \{ y \in D_B^\circ \setminus F_B : \Pi_\circ^{-1}(\{y\}) \subset \Omega^\pm \cap B \}.
	\end{align*}
	Clearly $D_B^0 \setminus F_B = A^+ \cup A^-$ and $A^+ \cap A^- = \varnothing$.
	 Suppose that $A^+$ has the largest $\HH^n$ measure out of $A^\pm$, so that $\HH^{n}(A^+) \geq \frac{\Cthree \gamma^{1/(n+2)}}{2} \HH^n(D_B^\circ)$ (the argument in the other case is the same). Then, using the fact that $\Pi_0|_B$ is Lipschitz with constant $\lesssim 1$, and recalling the definition of $A_2$ in \eqref{e:A1A2}, by the coarea formula we obtain
	\begin{align*}
		\HH^{n+1}(A_2)  \gtrsim  \int_{A^+} \HH^1 (\Pi_\circ^{-1}(\{y\}) \cap A_2) \, d\mathcal{L}^n(x)  \gtrsim \Cthree \gamma^{1/(n+2)} \,\HH^{n+1}(D_B^+).
	\end{align*}
	This, however, contradicts the conclusion of Lemma \ref{auxlem-part3-1}, i.e. \eqref{e:60}, whenever $\Cthree$ is chosen appropriately with respect to $\Cone$ in \eqref{e:60}. Hence, $\HH^n(D_B^\circ \setminus F_B) < \Cthree \gamma^{1/(n+2)} \HH^{n}(D_B^\circ)$.
	
	On the other hand, note that $\HH^{n}(\Pi_\circ(B)\setminus \Pi_\circ(D^+_B\cup D^-_B)) \approx \left( \frac{\epsilon r_0}{\rad(B)}\right)^{2n} \rad(B)^n \approx \epsilon^{n} \HH^n(\Pi_\circ(B)) $. We conclude that
	\begin{align*}
		\HH^n(F_B)  = \HH^n(F_B \cap D_B^\circ) + \HH^n(F_B \setminus D_B^\circ) 
		& \geq (1-\Cthree \gamma^{1/(n+2)}) \HH^n(D_B^+) \\
		& \geq (1-\Cthree \gamma^{1/(n+2)})(1-C \epsilon^n)\HH^n(\Pi_\circ(B))\\
		& \geq (1-\Cthree'\gamma^{1/(n+2)}) \HH^n(\Pi_\circ(B)),
	\end{align*}
	where the last inequality holds becuase $\epsilon \leq \gamma$. Setting now $\Ctwo= \Cthree'$ gives the lemma.
\end{proof}
\vv

\begin{lemma}
	Let $\Pi_\circ$ denote the spherical projection onto $\R^n\equiv \R^n\times\{0\}$ centered at the origin. Let $x_1 \in \R^n$ be a point with $|x_1|=R$ and consider a ball $B$ in $\R^{n+1}$ centered in $x_1$ and contained in the cone $$X_\circ = \{y\in\R^{n+1}:|y_{n+1}| \leq \tfrac34 |y|\}.$$ Denote $I=[R-\rad(B)/2, R+ \rad(B)/2]$. If a set $F_0 \subset \R^{n+1}$ satisfies
		\begin{equation}\label{e:32}
			\HH^n(\Pi_\circ(F_0\cap B)) \geq (1-C\alpha) \HH^n(\Pi_\circ(B))
	\end{equation}
	with $\alpha\in (0,1/10)$, then 
	\begin{equation}\label{e:31}
		\HH^1\left( \{ r \in I: \HH^{n-1}_\infty (S(0,r)\cap F_0 \cap B) \gtrsim_n \rad(B)^{n-1}  \}\right) \geq (1-C\alpha^{1/2})\HH^1(I),
	\end{equation}
	where the constant $C$ can only depend on $n$.
\end{lemma}

\begin{proof}
	Let $B_n = \Pi_\circ(B)$  and $F_B=\Pi_\circ(B\cap F)$. Note that $B_n$ is a ball in $\R^n$ with $\rad(B_n)=\rad(B)$ and center $x_1$. 
	   By Fubini, we can write
	\begin{equation*}
		\HH^{n}(F_B) = \int_{ r \in 2I} \HH^{n-1}(F_B \cap S(0,r)) \, dr,
	\end{equation*}
	where $2I = [R-\rad(B), R+\rad(B)]$. (Note that $F_B \cap S(0,r) \subset \R^{n} \cap S(0,r)$, which is an $(n-1)$-dimensional manifold). 
	Set 
	\begin{equation*}
		I_B= \{ r \in 2I: \HH^{n-1}(S(0,r) \cap B_n \setminus F_B) > \alpha^{1/2} \rad(B)^{n-1}\}.
	\end{equation*}	
	From \rf{e:32}, we have $\HH^n(B_n\setminus F_B)) \leq \alpha\, \HH^n(B_n)$. Then, by Chebyshev's inequality, it follows
\begin{align*}
\HH^1(I_B) & \leq 	\epsilon^{-1/2}\rad(B)^{1-n}\int_{2I}\HH^{n-1}(S(0,r) \cap B_n \setminus F_B)\,dr\\ & =
\alpha^{-1/2}\rad(B)^{1-n} \,\HH^{n}(B_n\setminus F_B) \lesssim \alpha^{1/2}\,\rad(B)\approx \alpha^{1/2}\,\HH^1(I).
\end{align*}
It is easy to see that if moreover $r \in I$, then $\HH^{n-1}(S(0,r) \cap B_n) \gtrsim \rad(B)^{n-1}$. 
 Thus, if $r \in I_G= I\setminus I_B$ and $\alpha$ is small enough,
 $$\HH^{n-1}(S(0,r) \cap B_n \setminus F_B) \leq \alpha^{1/2} \rad(B)^{n-1}\leq C\, \alpha^{1/2}\HH^{n-1}(S(0,r) \cap B_n)\leq \frac12\,\HH^{n-1}(S(0,r) \cap B_n).$$
Then, using 
 that $\mathrm{Lip}(\Pi_\circ|_B)$ is comparable to 1, because $B\subset X_\circ$, we see that, for $r\in I_G$,
	\begin{align*}
		\HH^{n-1}_\infty(S(0,r)\cap B \cap F) & \gtrsim  \HH^{n-1}_\infty(\Pi_\circ|_B(S(0,r)\cap B\cap F)) \\
		& = \HH^{n-1}_\infty (S(0,r)\cap B_n\cap F_B))\\
		& \approx \HH^{n-1} (S(0,r) \cap B_n \cap F_B)\\
		& \gtrsim \HH^{n-1}(S(0,r)\cap B_n)\gtrsim \rad(B)^{n-1},
	\end{align*}
	where we used the fact that $\HH^k$ and $\HH^k_\infty$ coincide up to a constant in $k$-dimensional manifolds.	
	So adjusting constants if necessary, the set on the left hand side of \rf{e:31} contains $I_G$. 
	As shown above, $\HH^1(I\setminus I_G) \leq \HH^1(I_B)\lesssim \alpha^{1/2}\HH^1(I)$, and so	
	the lemma is proven.
\end{proof}
\vv

\begin{figure}\label{figure17}
	\centering
	\includegraphics[scale=.7]{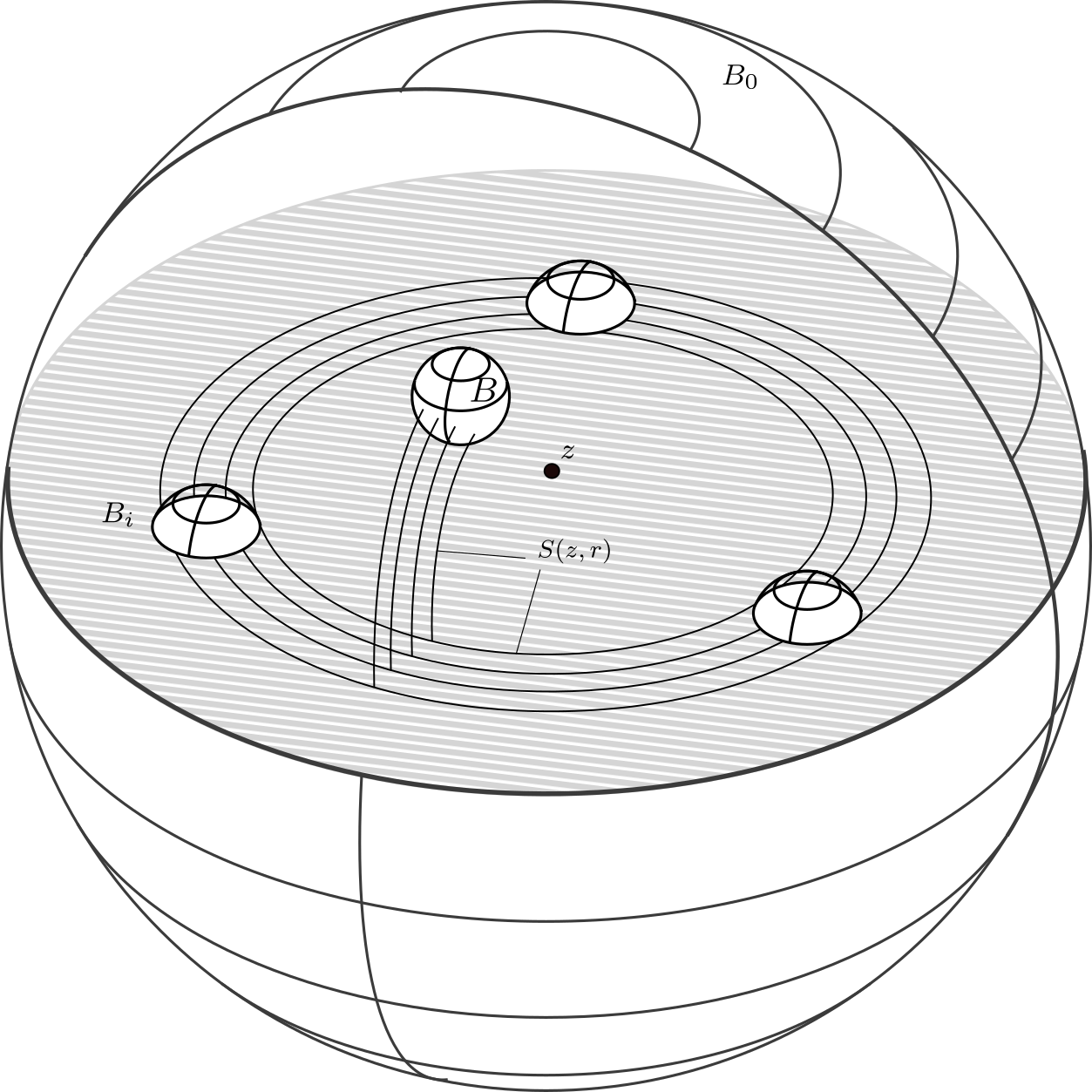}
	\caption{The setup of the proof of Lemma \ref{lem-small-betaF} in the three dimensional case. We are able to find 4 (i.e. $n+2$) balls such that, for many radii $r$, the spheres $S(z,r)$ intersect $F$ in these balls, and the intersection is large. Since no $2$-plane can be close to $4$, well spaced balls, this makes the $\ve$ coefficients large.}
\end{figure}

\begin{proof}[\bf Proof of Lemma \ref{lem-small-betaF} ]
Without loss of generality, we can assume that the $n$-plane $L_0$
infimizing $\beta_{\infty, E}(\tfrac14 B_0)$ is horizontal.
	 Suppose that there is a point $x \in F \cap \frac14 B_0$ such that $\dist(x, L_0) > \tau r_0$. Set $B=B(x, \gamma^{1/(n+2)} r_0)$. 
	 We suppose $\delta\ll\epsilon\leq \gamma \ll \tau^4$, so that $\dist(B,L_0)\approx \dist(B,\Gamma) \gtrsim \tau r_0$.
	 Set $\wt{F} := F \cap \overline B$ and note that $\wt{F} \subset B_0$ and $\diam (\wt F) \leq r_0/5$.
Set also $\widetilde{E} := E \cap \frac14 \overline B_0$.
By Lemma \ref{lem-slicing}, we deduce that
$$\capp_s(\wt F)\,\HH^n(\wt E)\leq c_0^{-1}\capp_{s}(\wt F)\,\frac{\HH^n(\wt E)^2}{r_0^n} \leq C(\tau) \int_{\wt E}\int_0^\infty \capp_{s-1}(\wt F\cap S(z,r))\,dr\,d\HH^n(z)$$
for any $s >1$.
Consequently, there exists some $z\in \wt E= E \cap \tfrac14 \overline B_0$ such that
$$\capp_{s}(\wt F) \lesssim_\tau \int_0^\infty \capp_{s-1}(\wt F\cap S(z,r))\,dr.$$
Because of the $s$-CDC, we know that $\capp_{s}(\wt F)=\capp_s (F \cap B)\gtrsim \rad(B)^{s}$.
Using also that
$\capp_{s-1}(G)\lesssim \HH_\infty^{s-1}(G)$ for any Borel set $G$, we deduce
$$
\rad(B)^s \lesssim_\tau \int_0^\infty \capp_{s-1}(\wt F\cap S(z,r))\,dr
\lesssim_\tau \int_0^\infty \HH_\infty^{s-1}(\wt F\cap S(z,r))\,dr.$$

Let $R=|x-z| \leq \tfrac14 r_0$, and set $I=[R - \rad(B), R+ \rad(B)]$. For some constant $b\in(0,1)$ to be chosen soon, let
 $I_{G} \subset I$ be the subset of radii such that 
	\begin{equation*}
		\HH^{s-1}_\infty(S(z,r) \cap \wt F) \geq b \,\rad(B)^{s-1}.
	\end{equation*}
	We then compute:
	\begin{align*}
		c(\tau)\,\rad(B)^s & \leq  \int_{I_G} \HH^{s-1}_\infty(S(z,r) \cap \wt F) \, dr + \int_{I \setminus I_G} \HH^{s-1}_\infty (S(z,r) \cap \wt F) \,dr\\
		& \leq  |I_G| \,\rad(B)^{s-1} +  b\, \rad(B)^s,
	\end{align*}
	which implies, with a choice of $b = c(\tau)/2$, say,
	\begin{equation*}
		\HH^1(I_G)\gtrsim_\tau \rad(B) \geq c(\tau) \,\gamma^{1/(n+2)}r_0.
	\end{equation*}

	For $i=1,\ldots,n$, consider the points
	$$x_i= z+ R\,e_i,$$
	where $e_1,\ldots,e_{n+1}$ is the standard orthonormal basis of $\R^{n+1}$.
	Notice that the points $x_1,\ldots,x_n$
	belong to $L_z\cap S(z,r)$, where $L_z$ is the $n$-plane parallel to $L_0$ through $z$.
	Denote by $B_i$ the ball centered in $x_i$ with radius $\rad(B_i)=\gamma^{1/(n+2)} r_0$. 
	Now, with the choice of $\epsilon$ and $\gamma$ stipulated at the beginning of this proof, we can apply Lemma \ref{auxlem-part3-1} to obtain through Lemma~\ref{auxlemma-part3-2} the conclusion \eqref{e:31} with $\alpha= \gamma^{1/(n+2)}$ for each of these balls; that is, for $R=|x-z|$ and $I= [R- \gamma^{1/(n+2)} r_0/2, R+  \gamma^{1/(n+2)} r_0/2]$, it holds
	\begin{equation*}
		\HH^{1}(I_{G,i}) = \HH^1 \left(\{ r \in I: \HH^{n-1}_\infty (S(z, r) \cap F \cap B_i) \gtrsim_n \HH^{n-1}_\infty(S(z,r) \cap B_i)\}\right) \geq (1-C\gamma^{1/(2n+4)})\HH^1(I),
	\end{equation*}
	where $C$ can only depend on $n$.

	Picking $s \in (n-1, n)$ for which \eqref{e:lem17.1ve} holds, via Chebyschev's inequality we see that the subset $I_{\ve_{s-1}}$ of radii $r\in I$ where $\ve_{s-1}(z,r) > \delta^{1/4}$ satisfies
	$$\HH^1\big(I_{\ve_{s-1}}\big)\leq \delta^{-1/2}\int_I\ve_{s-1}(z,r)^2\,dr \leq \delta^{-1/2}\HH^1(I)\int_I\ve_{s-1}(z,r)^2\,\frac{dr}r \leq \delta^{1/2}\HH^1(I).$$
	Recalling that $\HH^1(I_G)\geq c(\tau)\,\HH^1(I)$, it is then immediate to see that, for $\delta$ and $\epsilon$ small enough,
	\begin{equation*}
		\HH^1\Big( I_G\cap (I\setminus I_{\ve_{s-1}})\cap \bigcap_{i=1}^{n}I_{G,i}\Big) >0.
	\end{equation*}	
	But note that if $r  \in J_z:= (I\setminus I_{\ve_{s-1}}) \cap I_G \cap \bigcap_{i=1}^{n}I_{G,i}$, because $s-1<n-1$, then we have
	\begin{equation}\label{e:51}
		\HH^{s-1}_\infty (S(z,r) \cap B_i \cap F) \gtrsim_n \rad(B_i)^{s-1},
	\end{equation}
and the same estimate holds for $B$. 


Recall from Definition \ref{def-thickpoints} that $\ve_{s-1}(x,r)$ is given by 
\begin{equation*}
\ve_{s-1}(z,r) = \inf_{H} \frac{1}{r^s} \int_{V_{\mfc_0}^a(z,r;H) \cap F} \left(\delta_{L_H}(x)/r\right)^{n-s+1} \, d\HH^{s-1}_\infty(x),
\end{equation*}
where $V_{\mfc_0}^a(z,r; H) = T_{\mfc_0,a}^1(z,r;H) \cup T_{\mfc_0,a}^2 (z,r;H)$, and
\begin{equation*}
T_{\mfc_0,a}^i(z,r;H) =  \{y \in S_H^i(z,r) \setminus \Omega^i: \capp_{n-2}(\overline{B}(y, a \delta_{L_H}(y))\cap S(z,r)\setminus \Omega^i)\geq \mfc_0 \delta_{L_H}^{n-2}(y)\}.
\end{equation*}
Let $H$ be the half space infimizing $\ve_{s-1}(z,r)$, $r \in J_z$. Denote by $B_{far}$ a ball farthest from $L_H$ out of $B, B_1,...,B_{n}$. In particular, it must hold that 
\begin{equation}\label{e:50}
\mathrm{dist}(B_{far}, L_{H})\gg \gamma^{1/(n+2)} r_0
\end{equation}
by our choice of parameters.
Let $p \in B_{far} \cap S(z, r) \cap F$. Then from \eqref{e:50} we see that $\delta_{L_H}(p) \gg \gamma^{1/(n+2)} r_0$. We claim that
\begin{equation}\label{e:62}
p \in V_{\mfc_0}^a(z,r; H),
\end{equation}
whenever $\mfc_0$ and $a$ are appropriately chosen. Recall that $\rad(B_{far})=\gamma^{1/(n+2)} r_0$; then choosing $a \in (0,1)$ appropriately (e.g. $a=2/3$) we can insure that 
\begin{equation*}
B_{far} \cap S(z,r) \subset B(p, a \delta_{L_H}(p)) \cap S(z,r).
\end{equation*}
Hence we see that
\begin{align*}
\HH^{s-1}_\infty(B_{S(z,r)}(p, a \delta_{L_H}(p))\cap F) & \geq \HH^{s-1}_\infty (S(z, r) \cap B_{far} \cap F)  \gtrsim \rad(B_{far})^{s-1}  \geq \wt{\mathfrak{c}}_0 \delta_{L_H}^{s-1}(p),
\end{align*}
where we used \eqref{e:51} in the penultimate inequality and the choice $\wt{\mathfrak{c}}_0 = \gamma^{(n-1)/(n+2)}$ in the last one. We claim that $p \in V_{\wt{\mathfrak{c}}_0}^a(z,r;H)$ with the appropriate choice of $a, \wt{\mathfrak{c}}_0$ and $p\in B_{far} \cap F \cap S(z,r)$. Indeed, by Lemma 2.1 from \cite{FTV} we have
\begin{align*}
	\capp_{n-2}(\overline{B}_{S(z,r)}(p, a \delta_{L_H}(p))\setminus \Omega^i ) & \geq c(s,n) \HH_\infty^{s-1}(B_{S(z,r)}(p, a \delta_{L_H}(p))\setminus \Omega^i )^{\frac{n-2}{s-1}}\\
& \geq c(s,n,\wt{\mathfrak{c}}_0)\,\delta_{L_H}^{n-2}(p),
\end{align*}
where $i=+,-$.
So choosing $\mathfrak{c}_0= c(s,n,\wt{\mathfrak{c}}_0)$ our claim \eqref{e:62} holds.
 This and \eqref{e:51} imply that 
\begin{equation*}
	\HH^{s-1}_\infty(V_{\mfc_0}^a(z,r; H)) \geq \HH^{s-1}_\infty(S(z,r) \cap B_{far} \cap F) \gtrsim \rad(B_{far})^{s-1} \gtrsim_{\gamma,\tau} r^{s-1},
\end{equation*}	
recalling that $r \in J_z$ and thus $r \approx_\tau r_0$ and that $\rad(B_{far}) = \gamma^{1/(n+2)} r_0$.
		It is now easy to see that
 $\ve_{s-1}(z,r) \gtrsim_{n, \gamma} 1$. An appropriate choice of $\delta\ll \epsilon \leq \gamma$ now  contradicts that $r \in I\setminus I_{\ve_s}$ and Lemma \ref{lem-small-betaF}  is proven.
	
\end{proof}
\vv


\appendix
\section{Appendix}

In this appendix we include an argument  to end the proof of Theorem \ref{teomain1} when $\Omega^+, \Omega^-$ are open which that does not rely on \cite{DLF}. 
	
Recall from Subsection \ref{sec-proof-teoA} that to complete the proof of Theorem \ref{teomain1} it sufficed to show that $E$ has $\sigma$-finite measure
$\HH^n$. For the sake of contradiction, suppose that $\HH^n|_E$ is non-$\sigma$-finite
By a result of Sion and Sjerve \cite[Theorem 6.5]{sion-sjerve}, this implies that there exists a gauge function $h:[0,\infty)\to [0,\infty)$ 
(increasing, with $h(0)=0$) such that 
$\lim_{t\to0+} \frac{h(t)}{t^n} =0$
and $\HH^h(E)>0$. Also, there exists a compact subset $F\subset E$ such that $\HH^h(F)>0$, by 
\cite[Theorem 6.6]{sion-sjerve}. By Frostman Lemma, there exists a non-zero Radon measure $\sigma$ supported on $F$ such that $\sigma(B(x,r))\leq h(r)$ for all $x\in\R^{n+1}$, $r>0$.
As above, we can replace $F$ by a subset $F_1$ where \rf{e:unifint} holds uniformly on $F_1$.
Then, again we take $R>0$ small enough so that 
\rf{Repseq} holds for all $x\in F_1$, for some $\delta,\epsilon>0$ to be chosen soon. 

We also assume $R$ small
enough so that $h(r)\leq r^n$ for $r\leq R$. In this way, for all $r\in (0,R]$ and $x\in\R^{n+1}$,
$$\frac{\sigma(B(x,r))}{r^n} \leq \frac{\sigma(B(x,r))}{h(r)}\leq1.$$
Then, take $x_0\in F_1$ and $r_0\in (0,R]$ such that
$$\frac{\sigma(B(x_0,r_0))}{r_0^n} \geq \frac12 \,\sup_{x\in F_1,0<r\leq R}\frac{\sigma(B(x,r))}{r^n}.$$
Next take 
$\nu = \frac{r_0^n}{2\,\sigma(B(x_0,r_0))}\,\sigma|_{B(x_0,r_0)}.$
Notice that $\nu(B(x_0,r_0))=\frac12r_0^n$ and, for all $x\in F_1$, $r\in (0,r_0]$,
\begin{equation}\label{eqgrowth77}
	\frac{\nu(B(x,r))}{r^n} \leq 2\,\frac{\nu(B(x_0,r_0))}{r_0^n} =1.
\end{equation}
Then, by choosing parameters appropriately and applying the Main Lemmas \ref{mainlemma1} and
\ref{mainlemma2} to $\nu$, we infer that there exists
an $n$-dimensional Lipschitz graph $\Gamma$ such that the set $F_2 = F_1\cap \Gamma$ satisfies $\nu(F_2)>0$.
The growth condition \rf{eqgrowth77} implies that $\nu|_{F_2}$ is absolutely continuous with respect to $\HH^n|_{F_2}$.
This is not possible because for all $x\in F_1$, and in particular for $x\in F_2$,
$$\limsup_{r\to0}\frac{\nu(B(x,r))}{r^n} \leq\limsup_{r\to0}\frac{\nu(B(x,r))}{h(r)}\,\lim_{r\to0} \frac{h(r)}{r^n} = 0.$$
So we get the desired contradiction.

\vvv

\end{document}